\documentclass[reqno,11pt]{article}

\RequirePackage[OT1]{fontenc}
\RequirePackage{amsthm, amsmath, amssymb, amsfonts, natbib, xspace, graphicx, enumerate,multirow}
\usepackage{hyperref}
\hypersetup{colorlinks=true,citecolor=blue,linkcolor=blue}
\usepackage{natbib}
\setcitestyle{numbers,square,comma}

\usepackage{varioref} 
\usepackage{xr-hyper}
\usepackage{hyperref}     
\usepackage{booktabs}
\usepackage{epsfig,amsmath,amssymb,euscript,amsfonts,hyperref}
\usepackage{natbib}
\usepackage{array}
\usepackage{color,xr-hyper}
\usepackage{nccbbb}
\usepackage{verbatim} 
\usepackage[lined,boxed,commentsnumbered, ruled]{algorithm2e} 
\usepackage{graphicx}
\usepackage{subcaption}
\usepackage{bm}
\usepackage{tikz}
\usepackage{bbm}	
\usepackage{dsfont} 
\usepackage{MnSymbol} 
\usepackage{enumitem}
\usetikzlibrary{patterns}

\usepackage{epsfig,amsmath,amssymb,euscript,amsfonts,mathrsfs}
\usepackage{colortbl,color,graphics,graphicx,epsfig,multirow,float}
\usepackage{framed}
\usepackage{dsfont}
\usepackage{bbm}
\usepackage{changepage}
\usepackage{nicefrac}
\usepackage{indentfirst} 
\usepackage{comment}
\usepackage{booktabs}
\usepackage{natbib}
\usepackage{arydshln}

\usepackage[top=1in, bottom=1in, left=1in, right=1in]{geometry}

\renewcommand{\baselinestretch}{1.80}

\numberwithin{equation}{section}

\newtheorem*{assumption*}{\assumptionnumber}
\providecommand{\assumptionnumber}{}
\makeatletter
\newenvironment{assumption}[2]
{%
	\renewcommand{\assumptionnumber}{Assumption $\mathcal{#1}$.#2}%
	\begin{assumption*}%
		\protected@edef\@currentlabel{$\mathcal{#1}$.#2}%
	}
	{%
	\end{assumption*}
}
\makeatother

\theoremstyle{plain}
\newtheorem{thm}{Theorem}[section]
\newtheorem{lem}[thm]{Lemma}
\newtheorem{coro}[thm]{Corollary}

\theoremstyle{definition}
\newtheorem{example}{Example}[section]
\newtheorem{rem}{Remark}[section]

\newcommand{\ignore}[1]{}

\begin{document}

\title{\vspace*{-2cm} Asymptotic Spectral Theory for Spatial Data}

\author{
Wai Leong Ng\thanks{Department of Mathematics, Statistics and Insurance, School of Decision Sciences, The Hang Seng University of Hong Kong, Shatin, New Territories, Hong Kong. Email: {\ttfamily wlng@hsu.edu.hk}}\\
The Hang Seng University of Hong Kong
\and
Chun Yip Yau\thanks{Department of Statistics, The Chinese University of Hong Kong, Shatin, New Territories, Hong Kong.  Email: {\ttfamily cyyau@sta.cuhk.edu.hk}
} \\
The Chinese University of Hong Kong
\and 
}
\maketitle

\begin{abstract}
In this paper we study the asymptotic theory for spectral analysis of stationary random fields on $\mathbb{Z}^2$, including linear and nonlinear fields. Asymptotic properties of Fourier coefficients and periodograms, including limiting distributions of Fourier coefficients, and the uniform consistency of kernel spectral density estimators are obtained under various mild conditions on moments and dependence structures. The validity of the aforementioned asymptotic results for estimated spatial fields is also established.
	
\bigskip


\end{abstract}

\section{Introduction}

In a wide variety of disciplines, such as digital imaging, ecology, earth sciences, astronomy and meteorology, analysis of spatial data 
plays an important role. In particular, random field theory provides a theoretical modeling framework for statistical inference on spatial data, see \cite{cressie1993} and \cite{stein1999interpolation}. One popular approach for modeling and inferring spatial dependence is to consider autocovariance functions in spatial domain. Another approach for modeling spatial dependence is using spectral analysis in frequency domain. 
The primary goals of spectral analysis include estimating the spectral density function and establishing asymptotic properties of Fourier coefficients and periodograms,  
which are relevant for many applications, e.g., frequency domain bootstrap methods, specification and testing of parametric models, detecting anisotropies, signal extraction, interpolation, prediction, and smoothing; see, e.g., \cite{vidal2009automatic,YR1993}. There are two main types of estimators in nonparametric spectral density estimation for random fields, namely smoothed periodogram estimators and lag-windowed estimators, and they are closely related since the smoothed periodogram estimator can be viewed as a numerical integration approximation to the lag-windowed estimator;   
see \cite{heyde1993smoothed,ivanov1986,leonenko1999limit,priestley1964analysis,RM1984,vidal2009automatic}. 

Recently, there has been an increased interest in nonlinear random fields, see \cite{EM2013,LL2007,PD2017,HL2004}. 
However, existing asymptotic spectral theories of random fields often require linearity, summability of joint cumulants functions, strong mixing property or Gaussianity.
For example, \cite{GA2018} proposed a truncated autoregressive spectral density estimator and discussed its asymptotic properties under linearity; \cite{VG1989} established the uniform convergence of spectral density estimators 
on a homogeneous Gaussian random field; \cite{vidal2009automatic} constructed a lag-windowed estimator and a smoothed periodogram estimator under the linearity assumption, and proved the uniform convergence of the estimators. 
In this paper, we focus on the spectral theory for stationary random fields on $\mathbb{Z}^2$, and study the asymptotic properties of Fourier coefficients and periodograms, and establish the uniform consistency property of kernel spectral density estimators of random fields under mild moment and dependence structure conditions, which are applicable to a variety of nonlinear and non-Gaussian random fields.



The rest of the paper is organized as follows. In Section \ref{sec:setting}, we describe the setting and background on random fields on $\mathbb{Z}^2$. 
Section \ref{sec:assp} discusses the assumptions on the weak dependence structures of the random fields, and the kernel functions used in the spectral density estimation, which are required for establishing the main results. 
In Section \ref{sec:asymp_fourier}, we establish the asymptotic normality of the Fourier coefficients and the asymptotic behaviors of the Fourier coefficients and periodograms. Section \ref{sec:unf_kern} considers the uniform consistency of the kernel spectral density estimators. Section \ref{sec:asymp_estimated} establishes the asymptotic results for estimated spatial fields. Proofs and technical details are given in Section \ref{proof}.

\section{Setting and Preliminary}\label{sec:setting}
In this section, we describe the basic settings and preliminary about the random fields of interest. First, we introduce some notations. 
For any vector $\mathbf{a}=(a_1,a_2,\ldots,a_q) \in \mathbb{R}^q$, denote $|\mathbf{a}|=\prod_{i=1}^q |a_i|$, $\|\mathbf{a}\|=\max_{i=1,\ldots,q}\{|a_i|\}$ and $\|\mathbf{a}\|_p=(\sum_{i=1}^q |a_i|^p)^{1/p}$ for $p \geq 1$. 
Let $T=\{(t_1,t_2)\in\mathbb{Z}^2, 1 \leq t_k \leq d_k, k=1,2\}$ be a spatial rectangular lattice, and denote $\mathbf{d}_T=(d_1,d_2) \in \mathbb{N}^2$. 
For any set $G$, denote the cardinality of $G$ by $|G|$. For any random variable $X \in \mathbb{L}^p$, denote the $\mathbb{L}^p$ norm as $\|X\|_p=(\mathbb{E}(|X|^p))^{1/p}$. For any two real sequences $\{a_n\}$ and $\{b_n\}$, denote by $a_n \asymp b_n$ when $a_n = O(b_n)$ and $b_n = O(a_n)$, and by $a_n \sim b_n$ when $a_n/b_n \rightarrow 1$. For any $x\in\mathbb{R}$, $\lfloor x \rfloor$ is the greatest integer that is less than or equal to $x$. All vectors are column vectors unless specified otherwise, hence for any $\mathbf{a}=(a_1,a_2,\ldots,a_q) \in \mathbb{R}^q$ and $\mathbf{b}=(b_1,b_2,\ldots,b_q) \in \mathbb{R}^q$, the dot product between vectors $\mathbf{a}$ and $\mathbf{b}$ is defined as the vector multiplication $\mathbf{a}'\mathbf{b}=\mathbf{b}'\mathbf{a}=\sum_{i=1}^q a_ib_i$.

Let $\left\{ V(\mathbf{t}):\mathbf{t}\in \mathbb{Z}^2\right\}$ be a stationary random field on a two-dimensional grid with mean $\mu=\mathbb{E}(V(\mathbf{0}))$. Assume that we have observed $\{V(\mathbf{t}): \mathbf{t} \in T\}$ on a rectangular lattice $T=\{(t_1,t_2)\in\mathbb{Z}^2, 1 \leq t_k \leq d_k, k=1,2\}$ with $\mathbf{d}_T=(d_1,d_2) \in \mathbb{N}^2$ and cardinality $|T|=d_1d_2$. Throughout the paper, $T \rightarrow \infty$ denotes both $d_1$, $d_2 \rightarrow \infty$. 

\subsection{Fourier Coefficients for Spatial Data}
Define the Fourier coefficients $x(\mathbf{j})$ and $y(\mathbf{j})$ as
\begin{equation}
x(\mathbf{j}) = \frac{1}{\sqrt{|T|}}\sum_{\mathbf{t}\in{T}}V(\mathbf{t})\text{cos}(-\lambda_\mathbf{j}'\mathbf{t})\,, \quad		
y(\mathbf{j}) = \frac{1}{\sqrt{|T|}}\sum_{\mathbf{t}\in{T}}V(\mathbf{t})\text{sin}(-\lambda_\mathbf{j}'\mathbf{t})\,, \label{eq:a}
\end{equation}
for $\mathbf{j}=(j_1,j_2)\in T$ and $\lambda_\mathbf{j}=\left(\frac{2{\pi}j_1}{d_1}, \frac{2{\pi}j_2}{d_2}\right)$. Note that the Fourier coefficients $x(\mathbf{j})$, $y(\mathbf{j})$ depend on $T$, but to keep the notation simple we suppress this dependence. From the definition of the Fourier coefficients, we have
\begin{equation}\label{symprop}
x(\mathbf{d}_T-\mathbf{j})=x(\mathbf{j}),\ y(\mathbf{d}_T-\mathbf{j})=-y(\mathbf{j})\,.
\end{equation}

By using the symmetry property in \eqref{symprop}, we now partition $T$ as $T=N\cup\tilde{N}\cup M$ such that the Fourier coefficients defined on $\tilde{N}$ are determined by the Fourier coefficients defined on $N$.  
Also, the information about the covariance structure and mean of the random field are contained in $N$ and $M$ respectively. Hence, a spatial lattice can be reconstructed from the Fourier coefficients defined on $N$ and $M$.
In particular, 
when $d_1, d_2$ are both odd, define
$$N=\left\{(t_1,t_2): 1 \leq t_1 \leq d_1,~1 \leq t_2 \leq \frac{d_2-1}{2} \right\} \cup \left\{(t_1,t_2): 1 \leq t_1 \leq \frac{d_1-1}{2},~ t_2 = d_2 \right\}.$$
When $d_1$ is odd and $d_2$ is even (similar for $d_1$ is even and $d_2$ is odd), define 
$$N=\left\{(t_1,t_2): 1 \leq t_1 \leq d_1,~1 \leq t_2 \leq \frac{d_2}{2}-1 \right\} \cup \left\{(t_1,t_2): 1 \leq t_1 \leq \frac{d_1-1}{2},~ t_2 = \frac{d_2}{2}, d_2 \right\}.$$
When $d_1,d_2$ are both even, define 
$$N=\left\{(t_1,t_2): 1 \leq t_1 \leq d_1,~1 \leq t_2 \leq \frac{d_2}{2}-1 \right\} \cup \left\{(t_1,t_2): 1 \leq t_1 \leq \frac{d_1}{2}-1,~ t_2 = \frac{d_2}{2}, d_2 \right\}.$$
Then, using the symmetry in \eqref{symprop}, define the subset $\tilde{N}$ of $T$ as 
\begin{equation*}
\tilde{N}=\left\{(t_1,t_2)\in T: t_i= \left\{\begin{array}{cc} d_i-s_i\,, & \mbox{ if $s_i<d_i$}  \\
d_i\,,       & \mbox{ if $s_i=d_i$} 
\end{array}\right. , \mbox{for $i=1,2$}\,, (s_1,s_2)\in N\right\}\,.
\end{equation*}
Note that the Fourier coefficients at $\mathbf{j} \in \tilde{N}$ can be completely determined by the Fourier coefficients at $\mathbf{j}\in N$. From \eqref{eq:a}, for all $c \in \mathbb{R}$, the Fourier coefficients of $\{V(\mathbf{t})-c: \mathbf{t}\in T\}$ at $\mathbf{j} \in N$ are the same.  In other words, the Fourier coefficients in $N$ and $\tilde{N}$ are invariant under additive constants and thus contain no information about the mean. 
In contrast, all of the information about the mean is contained in the Fourier coefficients $x(\mathbf{j})$ for $\mathbf{j}\in M$, where 
\begin{equation}\label{eq:M}
M=\left\{ \begin{array}{cc} \{ (d_1,d_2)\} & \mbox{, when $d_1,d_2$ are odd\,,}  \\
\{ (d_1,d_2), (d_1/2,d_2)\}          & \mbox{, when only $d_1$ is even\,,} \\
\{ (d_1,d_2), (d_1,d_2/2)\}          & \mbox{, when only $d_2$ is even\,,} \\
\{ (d_1,d_2), (d_1/2,d_2),(d_1,d_2/2),(d_1/2,d_2/2)\}          & \mbox{, when $d_1,d_2$ are even\,.} 
\end{array}\right.
\end{equation}
Denote the sample mean as $\bar{V}_T=|T|^{-1}\sum_{\mathbf{t}\in{T}}V(\mathbf{t})$.
Table \ref{mean} summarizes the values of the Fourier coefficients in $M$. Figure \ref{symmetry} illustrates the sets $N$ and $M$ in the rectangular lattices with different $d_1$ and $d_2$.

\renewcommand{\baselinestretch}{1}
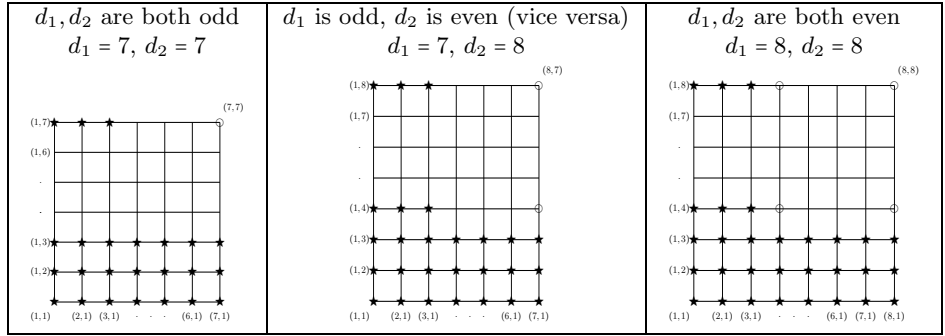
\begin{figure}[hptb]
	\centering \caption{\small {\it Illustration for extracting Fourier coefficients at $\mathbf{j} \in N$. In the grid, $\star$ are coefficients in the set $N$ and $\circ$ are coefficients containing information about the mean in the set $M$.}}
	{\footnotesize
		\begin{tabular}{|c|c|c|}
			\hline
			$d_1,d_2$ are both odd & $d_1$ is odd, $d_2$ is even (vice versa) & $d_1,d_2$ are both even\\
			$d_1=7$, $d_2=7$& $d_1=7$, $d_2=8$& $d_1=8$, $d_2=8$\\ 
			\resizebox{3.0cm}{3.0cm}{
				\begin{tikzpicture}
				\draw[-] (0,0,0) -- (6,0,0);
				\draw[-] (0,0,0) -- (0,6,0); 
				\draw[-] (0,1,0) -- (6,1,0);
				\draw[-] (1,0,0) -- (1,6,0);
				\draw[-] (0,2,0) -- (6,2,0);
				\draw[-] (2,0,0) -- (2,6,0);
				\draw[-] (0,3,0) -- (6,3,0);
				\draw[-] (3,0,0) -- (3,6,0);
				\draw[-] (0,4,0) -- (6,4,0);
				\draw[-] (4,0,0) -- (4,6,0);
				\draw[-] (0,5,0) -- (6,5,0);
				\draw[-] (5,0,0) -- (5,6,0);
				\draw[-] (0,6,0) -- (6,6,0);
				\draw[-] (6,0,0) -- (6,6,0);
				\node at (-0.5,-0.5,0) {$(1,1)$};
				\node at (-0.5,1,0) {$(1,2)$} node at (-0.5,2,0) {$(1,3)$} node at (-0.5,3,0) {.} node at (-0.5,3.,0) {.} node at (-0.5,4,0) {.} node at (-0.5,5,0) {$(1,6)$} node at (-0.5,6,0) {$(1,7)$};
				\node at (1,-0.5,0) {$(2,1)$} node at (2,-0.5,0) {$(3,1)$} node at (3,-0.5,0) {.} node at (3.5,-0.5,0) {.} node at (4,-0.5,0) {.} node at (5,-0.5,0) {$(6,1)$} node at (6,-0.5,0) {$(7,1)$};
				\node at (6.5,6.5,0) {$(7,7)$};
				\node at (0,0,0) {$\bigstar$} node at (1,0,0) {$\bigstar$} node at (2,0,0) {$\bigstar$} node at (3,0,0) {$\bigstar$} node at (4,0,0) {$\bigstar$} node at (5,0,0) {$\bigstar$} node at (6,0,0) {$\bigstar$};
				\node at (0,1,0) {$\bigstar$} node at (1,1,0) {$\bigstar$} node at (2,1,0) {$\bigstar$} node at (3,1,0) {$\bigstar$} node at (4,1,0) {$\bigstar$} node at (5,1,0) {$\bigstar$} node at (6,1,0) {$\bigstar$};
				\node at (0,2,0) {$\bigstar$} node at (1,2,0) {$\bigstar$} node at (2,2,0) {$\bigstar$} node at (3,2,0) {$\bigstar$} node at (4,2,0) {$\bigstar$} node at (5,2,0) {$\bigstar$} node at (6,2,0) {$\bigstar$};
				\node at (0,6,0) {$\bigstar$} node at (1,6,0) {$\bigstar$} node at (2,6,0) {$\bigstar$} node at (6,6,0) {$\bigcirc$} ;
				\end{tikzpicture}
			} & 
			\resizebox{3.0cm}{3.5cm}{
				\begin{tikzpicture}
				\draw[-] (0,0,0) -- (6,0,0);
				\draw[-] (0,0,0) -- (0,7,0); 
				\draw[-] (0,1,0) -- (6,1,0);
				\draw[-] (1,0,0) -- (1,7,0);
				\draw[-] (0,2,0) -- (6,2,0);
				\draw[-] (2,0,0) -- (2,7,0);
				\draw[-] (0,3,0) -- (6,3,0);
				\draw[-] (3,0,0) -- (3,7,0);
				\draw[-] (0,4,0) -- (6,4,0);
				\draw[-] (4,0,0) -- (4,7,0);
				\draw[-] (0,5,0) -- (6,5,0);
				\draw[-] (5,0,0) -- (5,7,0);
				\draw[-] (0,6,0) -- (6,6,0);
				\draw[-] (6,0,0) -- (6,7,0);
				\draw[-] (0,7,0) -- (6,7,0);
				\node at (-0.5,-0.5,0) {$(1,1)$};
				\node at (-0.5,1,0) {$(1,2)$} node at (-0.5,2,0) {$(1,3)$} node at (-0.5,3,0) {$(1,4)$} node at (-0.5,4,0) {.} node at (-0.5,5,0) {.} node at (-0.5,6,0) {$(1,7)$} node at (-0.5,7,0) {$(1,8)$};
				\node at (1,-0.5,0) {$(2,1)$} node at (2,-0.5,0) {$(3,1)$} node at (3,-0.5,0) {.} node at (3.5,-0.5,0) {.} node at (4,-0.5,0) {.} node at (5,-0.5,0) {$(6,1)$} node at (6,-0.5,0) {$(7,1)$};
				\node at (6.5,7.5,0) {$(8,7)$};
				\node at (0,0,0) {$\bigstar$} node at (1,0,0) {$\bigstar$} node at (2,0,0) {$\bigstar$} node at (3,0,0) {$\bigstar$} node at (4,0,0) {$\bigstar$} node at (5,0,0) {$\bigstar$} node at (6,0,0) {$\bigstar$};
				\node at (0,1,0) {$\bigstar$} node at (1,1,0) {$\bigstar$} node at (2,1,0) {$\bigstar$} node at (3,1,0) {$\bigstar$} node at (4,1,0) {$\bigstar$} node at (5,1,0) {$\bigstar$} node at (6,1,0) {$\bigstar$};
				\node at (0,2,0) {$\bigstar$} node at (1,2,0) {$\bigstar$} node at (2,2,0) {$\bigstar$} node at (3,2,0) {$\bigstar$} node at (4,2,0) {$\bigstar$} node at (5,2,0) {$\bigstar$} node at (6,2,0) {$\bigstar$};
				\node at (0,7,0) {$\bigstar$} node at (1,7,0) {$\bigstar$} node at (2,7,0) {$\bigstar$} node at (6,7,0) {$\bigcirc$};
				\node at (0,3,0) {$\bigstar$} node at (1,3,0) {$\bigstar$} node at (2,3,0) {$\bigstar$} node at (6,3,0) {$\bigcirc$};
				\end{tikzpicture}
			}& 
			\resizebox{3.5cm}{3.5cm}{
				\begin{tikzpicture}
				\draw[-] (0,0,0) -- (7,0,0);
				\draw[-] (0,0,0) -- (0,7,0); 
				\draw[-] (0,1,0) -- (7,1,0);
				\draw[-] (1,0,0) -- (1,7,0);
				\draw[-] (0,2,0) -- (7,2,0);
				\draw[-] (2,0,0) -- (2,7,0);
				\draw[-] (0,3,0) -- (7,3,0);
				\draw[-] (3,0,0) -- (3,7,0);
				\draw[-] (0,4,0) -- (7,4,0);
				\draw[-] (4,0,0) -- (4,7,0);
				\draw[-] (0,5,0) -- (7,5,0);
				\draw[-] (5,0,0) -- (5,7,0);
				\draw[-] (0,6,0) -- (7,6,0);
				\draw[-] (6,0,0) -- (6,7,0);
				\draw[-] (0,7,0) -- (7,7,0);
				\draw[-] (7,0,0) -- (7,7,0);
				\node at (-0.5,-0.5,0) {$(1,1)$};
				\node at (-0.5,1,0) {$(1,2)$} node at (-0.5,2,0) {$(1,3)$} node at (-0.5,3,0) {$(1,4)$} node at (-0.5,4,0) {.} node at (-0.5,5,0) {.} node at (-0.5,6,0) {$(1,7)$} node at (-0.5,7,0) {$(1,8)$};
				\node at (1,-0.5,0) {$(2,1)$} node at (2,-0.5,0) {$(3,1)$} node at (3,-0.5,0) {.} node at (3.5,-0.5,0) {.} node at (4,-0.5,0) {.} node at (5,-0.5,0) {$(6,1)$} node at (6,-0.5,0) {$(7,1)$} node at (7,-0.5,0) {$(8,1)$};
				\node at (7.5,7.5,0) {$(8,8)$};
				\node at (0,0,0) {$\bigstar$} node at (1,0,0) {$\bigstar$} node at (2,0,0) {$\bigstar$} node at (3,0,0) {$\bigstar$} node at (4,0,0) {$\bigstar$} node at (5,0,0) {$\bigstar$} node at (6,0,0) {$\bigstar$} node at (7,0,0) {$\bigstar$};
				\node at (0,1,0) {$\bigstar$} node at (1,1,0) {$\bigstar$} node at (2,1,0) {$\bigstar$} node at (3,1,0) {$\bigstar$} node at (4,1,0) {$\bigstar$} node at (5,1,0) {$\bigstar$} node at (6,1,0) {$\bigstar$} node at (7,1,0) {$\bigstar$};
				\node at (0,2,0) {$\bigstar$} node at (1,2,0) {$\bigstar$} node at (2,2,0) {$\bigstar$} node at (3,2,0) {$\bigstar$} node at (4,2,0) {$\bigstar$} node at (5,2,0) {$\bigstar$} node at (6,2,0) {$\bigstar$} node at (7,2,0) {$\bigstar$};
				\node at (0,7,0) {$\bigstar$} node at (1,7,0) {$\bigstar$} node at (2,7,0) {$\bigstar$} node at (3,7,0) {$\bigcirc$} node at (7,7,0) {$\bigcirc$};
				\node at (0,3,0) {$\bigstar$} node at (1,3,0) {$\bigstar$} node at (2,3,0) {$\bigstar$} node at (3,3,0) {$\bigcirc$} node at (7,3,0) {$\bigcirc$};
				\end{tikzpicture}
			}\\
			\hline
		\end{tabular}
	}\label{symmetry}
\end{figure}

\renewcommand{\baselinestretch}{1}
\begin{table}[hptb]
	\centering \caption{\footnotesize {\it Fourier coefficients in $M$ contain information about the mean.}}
	{\scriptsize
		\begin{tabular}{|c|c|c|c|c|}
			\hline
			$\mathbf{j}$ & $(d_1,d_2)$ & $(d_1/2,d_2)$ & $(d_1,d_2/2)$ & $(d_1/2,d_2/2)$ \\ 
			x($\mathbf{j}$) & $\sqrt{|T|}\bar{V}_T$ & $\frac{1}{\sqrt{|T|}} \sum_{\mathbf{t} \in T} (-1)^{t_1} V(\mathbf{t})$ & $\frac{1}{\sqrt{|T|}} \sum_{\mathbf{t} \in T} (-1)^{t_2} V(\mathbf{t})$ & $\frac{1}{\sqrt{|T|}}  \sum_{\mathbf{t} \in T} (-1)^{t_1+t_2} V(\mathbf{t})$\\
			y($\mathbf{j}$) & 0 & 0 & 0 & 0 \\
			\hline
		\end{tabular}
	}\label{mean}
\end{table}

\subsection{Kernel Spectral Density Estimators for Spatial Data}

Kernel estimator is introduced by \cite{RM1956} and has been studied extensively in nonparametric estimation of probability density and spectral density functions. Similar to \cite{vidal2009automatic} and \cite{RM2006}, 
we consider the following kernel spectral density estimator 
\begin{equation}\label{kernelesti}
\widehat{f}_T(\mathbf{\lambda})=\frac{1}{4 \pi^2}\sum_{\mathbf{j} \in\mathbb{Z}^2} \left(\frac{K\left(\frac{\lambda_1-\lambda_{j1}}{{h_T}_1},\frac{\lambda_2-\lambda_{j2}}{{h_T}_2}\right)}{\sum_{\mathbf{k} \in\mathbb{Z}^2}K\left(\frac{\lambda_{k1}}{{h_T}_1}, \frac{\lambda_{k2}}{{h_T}_2}\right)}\right)I(\mathbf{j})\,, \ \ \text{for} \ \lambda = (\lambda_1,\lambda_2) \in [0, 2\pi]^2\,,
\end{equation}
where $K(\cdot)$ is the kernel function, $\lambda_\mathbf{j}=(\lambda_{j1},\lambda_{j2})=\left(\frac{2 \pi j_1}{d_1}, \frac{2 \pi j_2}{d_2}\right)$ for $\mathbf{j}=(j_1,j_2)\in \mathbb{Z}^2$, and similarly, $\lambda_\mathbf{k}=(\lambda_{k1},\lambda_{k2})=\left(\frac{2 \pi k_1}{d_1}, \frac{2 \pi k_2}{d_2}\right)$ for $\mathbf{k}=(k_1,k_2)\in \mathbb{Z}^2$, $\mathbf{h}_T=({h_T}_1,{h_T}_2) \in \mathbb{R}^2$ is the bandwidth satisfying ${h_T}_k \rightarrow 0$, ${h_T}_1 \asymp {h_T}_2$ and  ${h_T}_kd_k \rightarrow \infty$ for $k=1,2$, and 
\begin{eqnarray}\label{periodogram_I}
	I(\mathbf{j})=\left\{ \begin{array}{ll}
		0\,, & \mbox{if \ } \mathbf{j}\in D= \{(c_1d_1,c_2d_2): (c_1,c_2) \in \mathbb{Z}^2\}\,, \\
		x^2(\mathbf{j})+y^2(\mathbf{j})\,, & \mbox{if \ } \mathbf{j}\in \mathbb{Z}^{2}\backslash D\,,
	\end{array}
	\right.
\end{eqnarray}
is the periodogram at frequency $\lambda_\mathbf{j}=\left(\frac{2 \pi j_1}{d_1}, \frac{2 \pi j_2}{d_2}\right)$. The periodogram can be set to $0$ on $D$ since it only contains information about the mean.	Note that it is asymptotically equivalent to the usual smoothed periodogram estimators and lag-windowed estimators in the literature, see Lemma \ref{lemequivalent} in Section \ref{proof}. For more details about smoothed periodogram estimators and lag-windowed estimators for spectral density functions, see \cite{heyde1993smoothed,priestley1964analysis,RM2006,RM1984,vidal2009automatic}.

\section{Assumptions}\label{sec:assp}
In this section, we impose some assumptions on the underlying random fields as well as kernel functions which are required for establishing the asymptotic results in Sections \ref{sec:asymp_fourier} to \ref{sec:asymp_estimated}. 

\subsection{Assumptions on the underlying random fields}
We impose the following assumptions about the lattice structure and underlying random fields to establish the asymptotic results.

\begin{assumption}{P}{1}\label{asstasymp}
	For all sufficiently large $|T|$, there exist $0 < \xi \leq 1/2$ and $c_1,~c_2 > 0$ such that $d_1 > c_1 |T|^\xi$ and $d_2 > c_2 |T|^\xi$.
\end{assumption}


\begin{assumption}{P}{2}\label{asssummable}
	The random field $\left\{ V(\mathbf{t}):\mathbf{t}\in \mathbb{Z}^2\right\}$ is stationary with absolutely summable auto-covariance function $\gamma(\cdot)$, i.e., $\sum_{\mathbf{j}\in \mathbb{Z}^2}|\gamma(\mathbf{j})| < \infty$, where $\gamma(\mathbf{j})=\text{\rm {Cov}}(V(\mathbf{0}),V(\mathbf{j}))$.
\end{assumption}

Assumption \ref{asssummable} implies that the spectral density of the random field $\left\{ V(\mathbf{t}):\mathbf{t}\in \mathbb{Z}^2\right\}$ exists and can be expressed as
\begin{eqnarray}\label{specden}
f(\mathbf{\lambda})=\frac{1}{4\pi^2}\sum\limits_{\mathbf{j}\in \mathbb{Z}^2}e^{-i\mathbf{\lambda'j}}\gamma(\mathbf{j})\,,
\end{eqnarray}	
where $\mathbf{\lambda}\in [0,2\pi]^2$. Moreover, $f(\cdot)$ is continuous and bounded. See Section 1.2.2 of \cite{GG2010} for details.

\begin{assumption}{P}{3}\label{assspecpostive}
	The spectral density is bounded from below, i.e., $f(\mathbf{\lambda})\geq c$ for some $c>0$ and all $\mathbf{\lambda}\in \left[ 0,2\pi\right]^2$.
\end{assumption}

\begin{assumption}{P}{4($v$)}\label{assautocov}	
	For some $0 < v \leq \frac{1}{2}$, the sample autocovariance function
	\begin{equation}\label{auto_sample}
	R_V(\mathbf{r})=|T|^{-1}\sum_{{\mathbf{j},\mathbf{j}+\mathbf{r}} \in T}(V(\mathbf{j})-\mu)(V(\mathbf{j}+\mathbf{r})-\mu)
	\end{equation}
	satisfies that, uniformly on $\mathbf{r} \in\mathbb{Z}^2$, $|R_V(\mathbf{r})-\mathbb{E}(R_V(\mathbf{r}))|=O_p(|T|^{-v})$.
\end{assumption}


Assumption \ref{asstasymp} asserts that the asymptotic regime requires the sample to be increasing in both directions of space at a polynomial rate with respect to the total number of observations, see Assumption A.1 in \cite{vidal2009automatic} for similar setting. Assumptions \ref{asssummable} and \ref{assspecpostive} are common conditions on the autocovariance function and spectral density for the asymptotic theory of spectral analysis. Assumption \ref{assautocov} is fulfilled for a large class of linear and nonlinear random fields with mild moment and short-range dependence conditions. In particular, Assumption \ref{assautocov} is fulfilled with $v=1/2$ for linear random fields with existing fourth moments.

\begin{assumption}{P}{5($r$)}\label{asslinear}
	The random field $\left\{ V(\mathbf{t}):\mathbf{t}\in \mathbb{Z}^2\right\}$ is a real-valued linear random fields, i.e., $V(\mathbf{j})-\mu=\sum_{\mathbf{s} \in\mathbb{Z}^2}a_\mathbf{s} \varepsilon_{\mathbf{j}-\mathbf{s}}$, where $\{\varepsilon_\mathbf{i}\}_{\mathbf{i} \in\mathbb{Z}^2}$ is an i.i.d. random field with $\mathbb{E}(\varepsilon_{\mathbf{0}})=0$, $\mathbb{E}(\varepsilon_{\mathbf{0}}^{8}) < \infty$ and $\sum_{\mathbf{s} \in\mathbb{Z}^2} |a_\mathbf{s}| |\mathbf{s}|^r < \infty$ for some $r \geq 0$. Also, the innovation $\varepsilon_{\mathbf{0}}$ satisfies Cram\'er's condition: there exist $\delta>0$ and $u_0>0$ such that, for all $u\in\mathbb{R}$ with $|u|>u_0$,
	$|\mathbb{E}\exp(iu\varepsilon_{\mathbf{0}})| \leq 1-\delta$.
\end{assumption}

Note that Assumption \ref{asslinear} with $r \geq 0$ implies Assumptions \ref{asssummable} and \ref{assautocov} with $v=1/2$.

\begin{assumption}{P}{6($p$)}\label{assdepenmeasure}
	For $\mathbf{j} \in\mathbb{Z}^2$,  assume that $V(\mathbf{j})-\mu=G(\varepsilon_{\mathbf{j}-\mathbf{s}}: \mathbf{s} \in\mathbb{Z}^2)$, where $G(\cdot)$ is a measurable function and $\{\varepsilon_\mathbf{i}\}_{\mathbf{i} \in\mathbb{Z}^2}$ is an i.i.d. random field. Let $\{\widetilde{\varepsilon}_\mathbf{i}\}_{\mathbf{i} \in\mathbb{Z}^2}$ be an i.i.d. copy of $\{\varepsilon_\mathbf{i}\}_{\mathbf{i} \in\mathbb{Z}^2}$.  Define the coupled version of $V(\mathbf{j})$ as $\widetilde{V}(\mathbf{j})-\mu=G(\varepsilon ^*_{\mathbf{j}-\mathbf{s}}: \mathbf{s} \in\mathbb{Z}^2)$,
	where 
	\begin{equation*}
	\varepsilon^*_{\mathbf{j}-\mathbf{s}} = \left\{ \begin{array}{lcr}
	\varepsilon_{\mathbf{j}-\mathbf{s}} & \mbox{if} & \mathbf{j}-\mathbf{s} \neq \mathbf{0}\,, \\ 
	{\widetilde{\varepsilon}}_{\mathbf{0}} & \mbox{if} &  \mathbf{j}-\mathbf{s} = \mathbf{0}\,.
	\end{array}\right.
	\end{equation*}
	Assume that there exists some $p > 0$ such that $V(\mathbf{j})$ belongs to $\mathbb{L}^p$ and
	\begin{equation*}
	\Delta_p := \sum_{\mathbf{j} \in \mathbb{Z}^2} \delta_{\mathbf{j},p} := \sum_{\mathbf{j} \in \mathbb{Z}^2} \|V(\mathbf{j})-\widetilde{V}(\mathbf{j})\|_p < \infty\,.
	\end{equation*}
\end{assumption}

Assumption \ref{assdepenmeasure} is the $p$-stable condition for random fields defined in \cite{EM2013} in which central limit theorems and invariance principles are established for a wide class of stationary nonlinear random fields. The next assumption is a geometric-moment contraction (GMC) condition:

\begin{assumption}{P}{7}\label{assgmc}
	Under the notation of Assumption \ref{assdepenmeasure}, 
	define another coupled version of $V(\mathbf{j})$ as $\widetilde{V}^{\dag}(\mathbf{j})-\mu=G(\varepsilon^{\dag}_{\mathbf{j}-\mathbf{s}}: \mathbf{s} \in\mathbb{Z}^2)$,
	where 
	\begin{equation*}
	\varepsilon^{\dag}_{\mathbf{j}-\mathbf{s}} = \left\{ \begin{array}{lcr}
	\varepsilon_{\mathbf{j}-\mathbf{s}} & \mbox{if} & \|\mathbf{s}\|<\|\mathbf{j}\|\,, \\ 
	{\widetilde{\varepsilon}}_{\mathbf{j}-\mathbf{s}} & \mbox{if} &  \|\mathbf{s}\| \geq \|\mathbf{j}\|\,.
	\end{array}\right.
	\end{equation*}
	Assume that there exist $\alpha > 0, C > 0$ and $0 < \rho = \rho(\alpha) < 1$ such that for all $\mathbf{j} \in\mathbb{Z}^2$,
	\begin{equation*}
	\mathbb{E}(|V(\mathbf{j})-\widetilde{V}^{\dag}(\mathbf{j})|^{\alpha}) \leq C \rho^{\|\mathbf{j}\|}\,. 
	\end{equation*}
\end{assumption}
Assumption \ref{assgmc} is the spatial extension of the geometric-moment contraction condition for time series, see \cite{SX2007}. 
This condition is fulfilled for short-range dependent linear random fields with finite variance, and a large class of nonlinear random fields 
such as nonlinearly transformed linear random fields, Volterra fields and nonlinear spatial autoregressive models, see \cite{EM2013} and \cite{deb2017asymptotic}.

%
%


\subsection{Assumptions on the kernel function $K$}

We impose the following mild regularity assumptions on the kernel function $K(\cdot)$.
\begin{assumption}{K}{1}\label{assk1}
	The kernel $K(\cdot)$ is a real, positive, even function with $\int_{\mathbb{R}^2} K(\mathbf{\lambda})d\mathbf{\lambda}=1$ and 
	\begin{equation*}
	\frac{4 \pi^2}{|\mathbf{h}_T||T|} \sum_{\mathbf{j} \in\mathbb{Z}^2}K\left(\frac{2 \pi j_1}{{h_T}_1 d_1},\frac{2 \pi j_2}{{h_T}_2 d_2}\right)=\int_{\mathbb{R}^2} K(\mathbf{\lambda})d\mathbf{\lambda}+o(1)=1+o(1)\,.
	\end{equation*}			
\end{assumption}

\begin{assumption}{K}{2}\label{assk2}
	Assume that	$\sup_{\mathbf{\lambda} \in [0,2 \pi]^2 } |K_\mathbf{h}(\mathbf{\lambda})|=O\left(|\mathbf{h}_T|^{-1}\right)=O\left(({h_T}_1{h_T}_2)^{-1}\right)$,
	where	
	\begin{equation}\label{Kh}	
	K_\mathbf{h}(\mathbf{\lambda})=\frac{1}{|\mathbf{h}_T|} \sum_{\mathbf{j} \in\mathbb{Z}^2}K\left(\frac{\lambda_1+2\pi j_1}{{h_T}_1},\frac{\lambda_2+2\pi j_2}{{h_T}_2}\right)\,.
	\end{equation}
\end{assumption}

\begin{assumption}{K}{3}\label{assk3}
	The kernel $K(\cdot)$ is absolutely integrable, i.e., $\int_{\mathbb{R}^2} |K(\mathbf{\lambda})|d{\mathbf{\lambda}} < \infty$.
	Furthermore, the inverse Fourier transform of $K(\cdot)$,
	\begin{equation}\label{smallk}
	k(\mathbf{x})=\int_{\mathbb{R}^2} K(\mathbf{\lambda})\exp(i\mathbf{x}' \mathbf{\lambda})d\mathbf{\lambda}\,,
	\end{equation}
	satisfies $|k(\mathbf{x})| \leq \widetilde{k}(\mathbf{x})$, where $\widetilde{k}(\mathbf{x})$ is monotonically decreasing with respect to $\|\mathbf{x}\|$ on $[0,\infty)^2$, i.e., $ \widetilde{k}(\mathbf{x}) \geq \widetilde{k}(\mathbf{y})$ if $\|\mathbf{x}\| \leq \|\mathbf{y}\|$, and is an even function with $\int_{\mathbb{R}_+^2}\widetilde{k}(\mathbf{x})d\mathbf{x} < \infty$.
\end{assumption}	

\begin{assumption}{K}{4}\label{assk4}
	The inverse Fourier transform $k(\mathbf{x})$ in \eqref{smallk} is Lipschitz continuous with support $[-1,1]^2$.
\end{assumption}

\begin{assumption}{K}{5}\label{assk5}
	The quantity $K_\mathbf{h}(\mathbf{\lambda})$ in (\ref{Kh}) satisfies the following uniformly Lipschitz condition: for some constant $L_K > 0$, 
	\begin{equation*}
	|\mathbf{h}_T|^{3/2}|K_\mathbf{h}(\mathbf{\lambda}_\mathbf{s})-K_\mathbf{h}(\mathbf{\lambda}_\mathbf{t})| \leq L_K \|\mathbf{\lambda}_\mathbf{s}-\mathbf{\lambda}_\mathbf{t}\|\,, 
	\end{equation*}
	uniformly in $\mathbf{\lambda}_\mathbf{s}=\left( \frac{2 \pi s_1}{d_1},\frac{2 \pi s_2}{d_2} \right)$ and  $\mathbf{\lambda}_\mathbf{t}=\left( \frac{2 \pi t_1}{d_1},\frac{2 \pi t_2}{d_2} \right)$.
\end{assumption}

\begin{rem}\label{rem4}
Assumptions \ref{assk1} and \ref{assk3} are general assumptions commonly used in kernel estimators, see \cite{RM2006} and \cite{vidal2009automatic}. Assumption \ref{assk2} is mild since it holds for any bounded kernel with compact support. The Lipschitz continuity in Assumption \ref{assk4} is also common in kernel spectral estimators; see, e.g., \cite{vidal2009automatic}. Moreover, under Assumptions \ref{assk3} or \ref{assk4}, we have
	\begin{equation*}
	K_\mathbf{h}(\mathbf{\lambda})= \frac{1}{4 \pi^2} \sum_{\mathbf{j} \in\mathbb{Z}^2} k(j_1{h_T}_1,j_2{h_T}_2)\exp(-i\mathbf{j}'\mathbf{\lambda})~~\text{and}~~	K(\mathbf{\lambda})= \frac{1}{4 \pi^2} \int k(\mathbf{x})\exp(-i\mathbf{x}'\mathbf{\lambda})d\mathbf{x}\,,
	\end{equation*}
	where $k(\cdot)$ is defined in \eqref{smallk}. From the above representations it is clear that as soon as the sum in $K_\mathbf{h}(\mathbf{\lambda})$ can be
	approximated by an integral for small enough $\mathbf{h}_T=({h_T}_1,{h_T}_2)$, it holds for large $T$, 
	\begin{equation}\label{Khregular}
	K_\mathbf{h}(\mathbf{\lambda})\cong \frac{1}{|\mathbf{h}_T|}K \left(\frac{\lambda_1}{{h_T}_1},\frac{\lambda_2}{{h_T}_2} \right)\,,
	\end{equation}
	which is of order $O\left({|\mathbf{h}_T|}^{-1}\right)$ for bounded $K(\cdot)$, hence Assumption \ref{assk2} holds.
	By \eqref{Khregular}, if the kernel $K(\cdot)$ is uniformly Lipschitz continuous with compact support, then Assumption \ref{assk5} holds for a small enough $\mathbf{h}_T=({h_T}_1,{h_T}_2)$. For infinite support kernels, if $K(\cdot)$ is bounded and continuously differentiable, then Assumption \ref{assk5} also holds. 
	In fact, Assumptions \ref{assk1} to \ref{assk5} hold for many commonly used kernels such as uniform kernels $K(t_1,t_2)=\frac{1}{4} \mathds{1}_{\left\{|t_1| \leq 1 \right\}}\mathds{1}_{\left\{|t_2| \leq 1 \right\}}$, multiplicative 2-$d$  Epanechnikov kernels $K(t_1,t_2)=\frac{9}{16} (1-t_1^2)(1-t_2^2)\mathds{1}_{\left\{|t_1| \leq 1 \right\}}\mathds{1}_{\left\{|t_2| \leq 1 \right\}}$ and Gaussian kernels $K(t_1,t_2)=\frac{1}{2\pi} \exp(\frac{-t_1^2-t_2^2}{2})$.
\end{rem}

\section{Asymptotic Properties of Fourier Coefficients and Periodograms}\label{sec:asymp_fourier}

In this section, we establish some asymptotic results on the Fourier coefficients and the periodograms. The following theorem shows that the asymptotic normality of the Fourier coefficients holds in spatial lattice data. Specifically, linear combinations of Fourier coefficients are uniformly asymptotic normal.

\begin{thm}\label{thm21}
	Let Assumptions \ref{asstasymp}, \ref{asssummable}, \ref{assspecpostive} and \ref{assdepenmeasure} hold with some $p \geq 2$. 
	Denote, for $\mathbf{j} \in N$, 
	\begin{equation*}
	\widetilde{s}_T(\mathbf{j},k) = \left\{ \begin{array}{lcr}
	\frac{x(\mathbf{j})}{\sqrt{2\pi^2f(\mathbf{\lambda}_\mathbf{j})}}, & ~k=1 \,,\\
	\frac{y(\mathbf{j})}{\sqrt{2\pi^2f(\mathbf{\lambda}_\mathbf{j})}}, & ~k=2\,,
	\end{array}\right.
	\end{equation*}
	and $F_N:=\{(\mathbf{j},k): \mathbf{j} \in N,~k=1,2\}$. Then, for each fixed $q \in\mathbb{N}$, 	as T $\rightarrow \infty$,
	\begin{equation*}
	\sup\limits_{{\substack{\mathbf{a}_i \in F_N, \forall i=1,...,q , \\ \mathbf{a}_{i_1} \neq \mathbf{a}_{i_2}, \forall i_1 \neq i_2}},~
		{\substack{~\mathbf{c} \in\mathbb{R}^q\\ \|\mathbf{c}\|_2=1}},~z \in\mathbb{R}}
	\left|\mathbb{P}\left((\widetilde{s}_T(\mathbf{a}_1),...,\widetilde{s}_T(\mathbf{a}_q))'\mathbf{c} \leq z\right)-\Phi(z)\right|=o(1)\,.
	\end{equation*}
	where $\Phi(\cdot)$ is the standard normal cumulative distribution function.
\end{thm}

Theorem \ref{thm21} is a nontrivial generalization of \cite{SX2007} from time series to spatial lattice data. 
The uniformity in Theorem \ref{thm21} is helpful to show the convergence of the empirical distribution function of the Fourier coefficients in the following corollary.

\begin{coro}\label{coro21}
	Let Assumptions \ref{asstasymp}, \ref{asssummable}, \ref{assspecpostive} and \ref{assdepenmeasure} hold with some $p \geq 2$. For any weights $\{w_{\mathbf{j},N}:\mathbf{j} \in N\}$ such that $\sum_{\mathbf{j} \in N} w_{\mathbf{j},N}=1$ and $\sum_{\mathbf{j} \in N} w^2_{\mathbf{j},N} \rightarrow 0$ as $T \rightarrow \infty $, we have
	\begin{equation*}
	\sup_{z \in\mathbb{R}}\left|\frac{1}{2} \sum_{\mathbf{j} \in N} w_{\mathbf{j},N} \left({\mathds{1}_{\{x(\mathbf{j}) \leq z \sqrt{2 \pi^2 f(\lambda_\mathbf{j})}\}}+\mathds{1}_{\{y(\mathbf{j}) \leq z \sqrt{2 \pi^2 f(\lambda_\mathbf{j})}\}}}\right)-\Phi(z)\right| \xrightarrow{p}0\,,
	\end{equation*}
	where $\Phi(\cdot)$ is the standard normal distribution function. 
	If $\{\{w_{\mathbf{j},N,\mathbf{s}}:\mathbf{j} \in N\}:\mathbf{s} \in S\}$ is a class of weights indexed by a countable index set $S$ satisfying $\sum_{\mathbf{j} \in N} w_{\mathbf{j},N,\mathbf{s}}=1$ for all $\mathbf{s} \in S$ and $\sup_{\mathbf{s} \in S}\sum_{\mathbf{j} \in N} w^2_{\mathbf{j},N,\mathbf{s}} \rightarrow 0$, then the assertion remains true in the sense that, for any $\varepsilon > 0$, 
	\begin{equation*}
	\sup_{\mathbf{s} \in S} \, \mathbb{P}\left(\sup_{z \in\mathbb{R}} \left|\frac{1}{2} \sum_{\mathbf{j} \in N} w_{\mathbf{j},N,\mathbf{s}}\left({\mathds{1}_{\{x(\mathbf{j}) \leq z \sqrt{2 \pi^2 f(\lambda_\mathbf{j})}\}}+\mathds{1}_{\{y(\mathbf{j}) \leq z \sqrt{2 \pi^2 f(\lambda_\mathbf{j})}\}}}\right)-\Phi(z) \right| \geq \varepsilon\right) \rightarrow 0\,.
	\end{equation*}
\end{coro}

	The following example fulfills the required conditions for the underlying random fields in Theorem \ref{thm21} and Corollary \ref{coro21}.
	\begin{example}[Linear Random Fields]\label{linear_field}
		Define the linear random field $\left\{ V(\mathbf{t}):\mathbf{t}\in \mathbb{Z}^2\right\}$ as
		\begin{equation}\label{ex_linear}
		V(\mathbf{j})=\sum_{\mathbf{s} \in\mathbb{Z}^2}a_\mathbf{s} \varepsilon_{\mathbf{j}-\mathbf{s}}\,,	
		\end{equation}
		where $\{\varepsilon_\mathbf{t}\}_{\mathbf{t}\in \mathbb{Z}^2}$ is an i.i.d. random field with $\mathbb{E}(\varepsilon_{\mathbf{0}})=0$.
		If we have $\mathbb{E}(|\varepsilon_{\mathbf{0}}|^p) < \infty$ with some $p \geq 2$ and $\sum_{\mathbf{j} \in \mathbb{Z}^2} |a_\mathbf{j}| < \infty$, then Assumptions  \ref{asssummable}, \ref{assspecpostive} and \ref{assdepenmeasure} with some $p \geq 2$ hold since $\sum_{\mathbf{j} \in \mathbb{Z}^2} \delta_{\mathbf{j},p} = \sum_{\mathbf{j} \in \mathbb{Z}^2} \|V(\mathbf{j})-\widetilde{V}(\mathbf{j})\|_p = \sum_{\mathbf{j} \in \mathbb{Z}^2} |a_\mathbf{j}| \|\varepsilon_{\mathbf{0}}-{\widetilde{\varepsilon}}_{\mathbf{0}}\|_p < \infty$.
	\end{example}
	\begin{example}[Nonlinear Spatial Autoregressive Models]\label{autoreg}
		Let $\mathcal{N} \subset \mathbb{Z}^2$ be a finite set and $\mathbf{0} \notin \mathcal{N}$. Define the nonlinear spatial autoregressive models $\left\{ V(\mathbf{t}):\mathbf{t}\in \mathbb{Z}^2\right\}$ as \begin{equation}\label{ex_autoreg}
		V(\mathbf{t})=G(\{V(\mathbf{t}-\mathbf{s})\}_{\mathbf{s} \in \mathcal{N}} ; \varepsilon_{\mathbf{t}})\,,
		\end{equation}
		where the function $G$ satisfies the following condition: there exists nonnegative numbers $u_{\mathbf{s}}, {\mathbf{s}} \in \mathcal{N}$, with $\sum_{{\mathbf{s}} \in \mathcal{N}} u_{\mathbf{s}} < 1$ such that for all $\{v(-\mathbf{s})\}_{\mathbf{s} \in \mathcal{N}}$ and $\{v'(-\mathbf{s})\}_{\mathbf{s} \in \mathcal{N}}$,
 		$$
 		|G(\{v(-\mathbf{s})\}_{\mathbf{s} \in \mathcal{N}} ; \varepsilon_{\mathbf{t}})-G(\{v'(-\mathbf{s})\}_{\mathbf{s} \in \mathcal{N}} ; \varepsilon_{\mathbf{t}})| \leq \sum_{{\mathbf{s}} \in \mathcal{N}} u_{\mathbf{s}} |v(-\mathbf{s})-v'(-\mathbf{s})|\,.
 		$$			
		If there exists $\{v(-\mathbf{s})\}_{\mathbf{s} \in \mathcal{N}}$ such that $G(\{v(-\mathbf{s})\}_{\mathbf{s} \in \mathcal{N}} ; \varepsilon_{\mathbf{0}}) \in \mathbb{L}^p$ for some $p \geq 2$, then following the argument in Section 5 of \cite{SX2007} or Example 2 of \cite{deb2017asymptotic}, we have $\mathbb{E}\left(|V(\mathbf{0})|^{p}\right)< \infty$ and  $\delta_{\mathbf{j},p}=O(\rho^{\|\mathbf{j}\|})$ for some $\rho \in (0,1)$, and hence Assumptions  \ref{asssummable}, \ref{assspecpostive} and \ref{assdepenmeasure} with some $p \geq 2$ hold.
	\end{example}

The following theorem establishes some asymptotic behaviors of the Fourier coefficients and periodograms relative to the spectral density function.

\begin{thm}\label{thm2}
	Suppose that Assumptions \ref{asstasymp}, \ref{asssummable} and \ref{assspecpostive} hold, and either Assumption \ref{asslinear} with $r > 1/2$ holds, or Assumption \ref{assgmc} with $\mathbb{E}\left(|V(\mathbf{0})|^{16}\right)< \infty$ holds, then with $I(\mathbf{j})$ defined as in \eqref{periodogram_I}, we have
	\begin{itemize}
		\item[(a)] 
		\begin{equation}\label{Clemeq3}
		\frac{1}{2|N|} \sum\limits_{\mathbf{j}\in N}\frac{x(\mathbf{j})+y(\mathbf{j})}{\sqrt{f(\mathbf{\lambda_j})}} \xrightarrow[]{p} 0\,.
		\end{equation}
		
		\item[(b)] If additionally the auto-covariance function $\gamma(\cdot)$ satisfies $\sum_{\mathbf{j}\in \mathbb{Z}^2} |\mathbf{j}|^u |\gamma(\mathbf{j})|< \infty$ for some $u > 0$, then
		\begin{equation}\label{Clemeq4}
		\begin{aligned}
		\sup_{\mathbf{l},\mathbf{k} \in N}| \text{\rm {Cov}}(x(\mathbf{l}),x(\mathbf{k}))-2 \pi^2f(\lambda_\mathbf{k})\delta_{\mathbf{l},\mathbf{k}}|=	\left\{ \begin{array}{lcr}
		O(|T|^{-u}), & 0 < u < 1 \,,\\
		O\left(\frac{\log|T|}{|T|}\right), & u=1 \,,\\
		O(|T|^{-1}), & u>1\,.
		\end{array}\right.\\
		\sup_{\mathbf{l},\mathbf{k} \in N}|\text{\rm {Cov}}(y(\mathbf{l}),y(\mathbf{k}))-2 \pi^2f(\lambda_\mathbf{k})\delta_{\mathbf{l},\mathbf{k}}|=	\left\{ \begin{array}{lcr}
		O(|T|^{-u}), & 0 < u < 1 \,,\\
		O\left(\frac{\log|T|}{|T|}\right), & u=1 \,,\\
		O(|T|^{-1}), & u>1\,.
		\end{array}\right.
		\end{aligned}
		\end{equation}
		\item[(c)] 
		\begin{equation}\label{Clemeq5}
		\frac{1}{4\pi^2|N|}\sum\limits_{\mathbf{j}\in N}\frac{I(\mathbf{j})}{f(\mathbf{\lambda_j})} \xrightarrow[]{p} 1\,.
		\end{equation}
		\item[(d)] 
		\begin{equation}\label{Clemeq6}
		\frac{1}{|N|} \sum_{\mathbf{j} \in N} \frac{I^2(\mathbf{j})}{f^2(\lambda_\mathbf{j})}=2(4\pi^2)^2+o_p(1)\,.
		\end{equation}
		\item[(e)] 
		There exists some $q = 4+\epsilon$ with $\epsilon \in (0,1)$ such that
		\begin{equation}\label{Clemeq7}
		\frac{1}{|N|} \sum_{\mathbf{j} \in N} \frac{I^q(\mathbf{j})}{f^q(\lambda_\mathbf{j})}<C_2+o_p(1)\,.
		\end{equation}
	\end{itemize}
\end{thm}

%
%

Some statistical applications rely on the asymptotic behaviors of the weighted means of the Fourier coefficients and periodograms, for example, local bootstrap in frequency domain, see \cite{paparoditis1999local} and \cite{kirch2011tft} in time series context. The next theorem describes the asymptotic behaviors of the weighted means of the Fourier coefficients and periodograms. 
\begin{thm}\label{thm_periodogram}
	Suppose that Assumptions \ref{asstasymp}, \ref{asssummable} and \ref{assspecpostive} hold, and either Assumption \ref{asslinear} with $r > 1/2$ holds, or Assumption \ref{assgmc} with $\mathbb{E}\left(|V(\mathbf{0})|^{16}\right)< \infty$ holds. Also, assume that the bandwidth fulfills $(|\mathbf{h}_T|^4|T|)^{-1}=o(1)$, and the kernel $K(\cdot)$ fulfills Assumptions \ref{assk1} and \ref{assk5},  and 
	\begin{eqnarray}\label{add_assump}
	\frac{1}{|T||\mathbf{h}_T|} \sum_{\mathbf{j} \in\mathbb{Z}^2}K^2 \left(\frac{2 \pi j_1}{{h_T}_1 d_1},\frac{2 \pi j_2}{{h_T}_2 d_2} \right) =O(1)\,.
	\end{eqnarray}
	Define the weights $p_{\mathbf{s},T}$ as
	\begin{eqnarray}\label{def_weight}
		p_{\mathbf{s},T}= \frac{K\left(\frac{2 \pi s_1}{{h_T}_1 d_1},\frac{2 \pi s_2}{{h_T}_2 d_2} \right)}{\sum_{\mathbf{j} \in \mathbb{Z}^2} K\left(\frac{2 \pi j_1}{{h_T}_1 d_1},\frac{2 \pi j_2}{{h_T}_2 d_2} \right)}\,.
	\end{eqnarray}
	Then, with $I(\mathbf{j})$ defined as in \eqref{periodogram_I}, the following results hold.
	\begin{itemize}
		\item[(a)] 
		If
		\begin{equation}\label{lemeq4}
		\sup_{\mathbf{l},\mathbf{k} \in N}| \text{\rm {Cov}}(x(\mathbf{l}),x(\mathbf{k}))-2 \pi^2f(\lambda_\mathbf{k})\delta_{\mathbf{l},\mathbf{k}}|=O \left(\frac{1}{|\mathbf{h}_T||T|}\right) \,,
		\end{equation}
		\begin{equation}\label{lemeq5}
		\sup_{\mathbf{l},\mathbf{k} \in N}|\text{\rm {Cov}}(y(\mathbf{l}),y(\mathbf{k}))-2 \pi^2f(\lambda_\mathbf{k})\delta_{\mathbf{l},\mathbf{k}}|=O \left(\frac{1}{|\mathbf{h}_T||T|}\right)\,,
		\end{equation}
		are satisfied, then
		\begin{equation*}
		{\sup_{\mathbf{j} \in N} \left| \sum_{\mathbf{s} \in \mathbb{Z}^2} p_{\mathbf{s},T} (x(\mathbf{j}+\mathbf{s})+y(\mathbf{j}+\mathbf{s})) \right| = o_p(1)\,.}
		\end{equation*} 
		
		\item[(b)] 
		\begin{equation*} 
		\sup_{\mathbf{j} \in N} \left| \sum_{\mathbf{s} \in \mathbb{Z}^2} p_{\mathbf{s},T} I(\mathbf{j}+\mathbf{s}) - 4 \pi^2 f(\lambda_\mathbf{j}) \right| = o_p(1).
		\end{equation*}
		
		
		\item[(c)] 
		\begin{equation*} 
		\sup_{\mathbf{j} \in N} \sum_{\mathbf{s} \in \mathbb{Z}^2} p_{\mathbf{s},T} I^2(\mathbf{j}+\mathbf{s}) \leq C_1+o_p(1)\,.
		\end{equation*}
		
		\item[(d)] 
		There exists some $q = 4+\epsilon$ with $\epsilon \in (0,1)$ such that
		\begin{equation*} 
		\sup_{\mathbf{j} \in N} \sum_{\mathbf{s} \in \mathbb{Z}^2} p_{\mathbf{s},T} I^q (\mathbf{j}+\mathbf{s}) \leq C_2+o_p(1)\,.
		\end{equation*}
	\end{itemize}
\end{thm}

Note that the condition in \eqref{add_assump} is satisfied when the kernel $K(\cdot)$ is square-integrable, i.e., $\int_{\mathbb{R}^2} K^2(\mathbf{\lambda}) d\mathbf{\lambda}< \infty$, which is satisfied by the commonly used kernels mentioned in Remark \ref{rem4}. Also, the conditions \eqref{lemeq4} and \eqref{lemeq5} are satisfied under the additional condition in Theorem \ref{thm2}(b) for $|\mathbf{h}_T|=O(|T|^{u-1})$ when $u<1$ or $|\mathbf{h}_T|=o(1)$ when $u \geq 1$. The following examples fulfill the required conditions for the underlying random fields in Theorems \ref{thm2} and \ref{thm_periodogram}, .
	\begin{example}[Linear Random Fields]
		For the linear random fields defined in \eqref{ex_linear} in Example \ref{linear_field}, if we have $\sum_{\mathbf{s} \in\mathbb{Z}^2} |a_\mathbf{s}| |\mathbf{s}|^r < \infty$ for some $r > 1/2$, $\mathbb{E}(\varepsilon_{\mathbf{0}}^{8}) < \infty$ and the innovation $\varepsilon_{\mathbf{0}}$ satisfies the Cram\'er's condition, for example, $\{\varepsilon_\mathbf{t}\}_{\mathbf{t}\in \mathbb{Z}^2}$ are centered i.i.d. Gaussian white noises, then Assumption \ref{asslinear} with $r > 1/2$ holds.
	\end{example}
	\begin{example}[Linear Random Fields]
		For the linear random fields defined in \eqref{ex_linear} in Example \ref{linear_field}, if we have $\mathbb{E}(|\varepsilon_{\mathbf{0}}|^p) < \infty$ with some $p \geq 16$ and $|a_\mathbf{j}| \leq C \rho^{\|\mathbf{j}\|}$ for some $\rho \in (0,1)$ and $C>0$, then we have $\mathbb{E}\left(|V(\mathbf{0})|^{16}\right)< \infty$ and $\delta_{\mathbf{j},p}=O(\rho^{\|\mathbf{j}\|})$ which implies Assumption \ref{assgmc}, see Section 4 of  \cite{wu2011asymptotic}.
	\end{example}
\begin{example}[Volterra Fields]\label{volterra}
	Volterra Fields is a class of nonlinear random fields which plays an important role in the nonlinear system theory. Define the second order Volterra process $\left\{ V(\mathbf{t}):\mathbf{t}\in \mathbb{Z}^2\right\}$ as 
	\begin{equation}\label{ex_volterra}
	V(\mathbf{j})=\sum_{\mathbf{s}_1,\mathbf{s}_2 \in\mathbb{Z}^2}a_{\mathbf{s}_1,\mathbf{s}_2} \varepsilon_{\mathbf{j}-\mathbf{s}_1}\varepsilon_{\mathbf{j}-\mathbf{s}_2}\,,
	\end{equation}
	where $\{\varepsilon_\mathbf{t}\}_{\mathbf{t}\in \mathbb{Z}^2}$ be an i.i.d. random field with $\mathbb{E}(\varepsilon_{\mathbf{0}})=0$ and $\{a_{\mathbf{s}_1,\mathbf{s}_2}\}$ are real coefficients with $a_{\mathbf{s}_1,\mathbf{s}_2}=0$ if $s_1=s_2$. Then, by Rosenthal inequality, there exists a constant $C_p>0$ such that
	$$
	\delta_{\mathbf{j},p} = \|V(\mathbf{j})-\widetilde{V}(\mathbf{j})\|_p \leq C_p \left(A^{1/2}_{\mathbf{j}} \|\varepsilon_{\mathbf{0}}\|_2 \|\varepsilon_{\mathbf{0}}\|_p + B^{1/p}_{\mathbf{j}} \|\varepsilon_{\mathbf{0}}\|_p^2\right)\,,
	$$
	where $A_\mathbf{j}=\sum_{\mathbf{s}_1,\mathbf{s}_2 \in\mathbb{Z}^2}(a_{\mathbf{s}_1,\mathbf{j}}^2+a_{\mathbf{j},\mathbf{s}_2}^2)$ and $B_\mathbf{j}=\sum_{\mathbf{s}_1,\mathbf{s}_2 \in\mathbb{Z}^2}(|a_{\mathbf{s}_1,\mathbf{j}}|^p+|a_{\mathbf{j},\mathbf{s}_2}|^p)$. If we have $\mathbb{E}(|\varepsilon_{\mathbf{0}}|^p) < \infty$ for some $p \geq 32$ and  $a_{\mathbf{s}_1,\mathbf{s}_2}=O(\rho^{\max\{\|\mathbf{s}_1\|,\|\mathbf{s}_2\|\}})$ for some $\rho \in (0,1)$, then we have $\mathbb{E}\left(|V(\mathbf{0})|^{16}\right)< \infty$, and $\delta_{\mathbf{j},p}=O(\rho^{\|\mathbf{j}\|})$, and hence Assumption \ref{assgmc} holds.
\end{example}
	\begin{example}[Nonlinear Spatial Autoregressive Models]
		For the nonlinear spatial autoregressive models defined in \eqref{ex_autoreg} in Example \ref{autoreg}, if there exists $\{v(-\mathbf{s})\}_{\mathbf{s} \in \mathcal{N}}$ such that $G(\{v(-\mathbf{s})\}_{\mathbf{s} \in \mathcal{N}} ; \varepsilon_{\mathbf{0}}) \in \mathbb{L}^p$ for $p \geq 16$, then we have $\mathbb{E}\left(|V(\mathbf{0})|^{16}\right)< \infty$ and $\delta_{\mathbf{j},p}=O(\rho^{\|\mathbf{j}\|})$ for some $\rho \in (0,1)$ which implies Assumption \ref{assgmc}.
\end{example}

\section{Uniform Consistency for Kernel Spectral Density Estimators}\label{sec:unf_kern}
In this section, we establish the uniform consistency of the kernel spectral density estimator for spatial lattice data. The following theorem establishes the uniform consistency of $\widehat{f}_T(\mathbf{\lambda})$ in (\ref{kernelesti}) under two different sets of regularity conditions.


\begin{thm}\label{thm3}
	Suppose that Assumptions \ref{asstasymp}, \ref{asssummable} and \ref{assk1} hold.
	\begin{itemize}
		\item[(a)] If Assumptions \ref{assk3} and \ref{assautocov} with some $0 <v \leq 1/2$ hold, and the bandwidth satisfies 
		$|\mathbf{h}_T|+|\mathbf{h}_T|^{-1}|T|^{-v} \rightarrow 0$, then
		\begin{equation}\label{unif_conv_f}
		\max_{\mathbf{\lambda} \in [0,2\pi]^2}\left|\widehat{f}_T(\mathbf{\lambda})-f(\mathbf{\lambda})\right| \xrightarrow{p} 0\,.
		\end{equation}
		\item[(b)] If Assumptions \ref{assk4}, \ref{assspecpostive}, \ref{assdepenmeasure} and \ref{assgmc} hold with some $p \geq 2$, $\mathbb{E}(|V(\mathbf{0})|^{\phi})< \infty$ for some $ 4< \phi \leq 8$, and the bandwidth satisfies $|\mathbf{h}_T| \rightarrow 0,~ (|\mathbf{h}_T||T|^\eta)^{-1}=O(1)$ for some $0< \eta < \frac{(\phi-4)}{\phi}$, then
		\begin{equation*}
		\max_{\mathbf{\lambda} \in [0,2\pi]^2}\left|\widehat{f}_T(\mathbf{\lambda})-f(\mathbf{\lambda})\right| \xrightarrow{p} 0\,.
		\end{equation*}
	\end{itemize}
\end{thm}

The following examples fulfill the required conditions for the underlying random fields in Theorem \ref{thm3}.
	\begin{example}[Linear Random Fields]
		For the linear random fields defined in \eqref{ex_linear} in Example \ref{linear_field}, if we have $\mathbb{E}(|\varepsilon_{\mathbf{0}}|^p) < \infty$ with some $p \geq 4$ and $\sum_{\mathbf{s} \in\mathbb{Z}^2} |a_\mathbf{s}| < \infty$, then Assumptions \ref{assautocov} with $v=1/2$ holds.
	\end{example}
	\begin{example}[Linear Random Fields]
		For the linear random fields defined in \eqref{ex_linear} in Example \ref{linear_field}, if we have $\mathbb{E}(|\varepsilon_{\mathbf{0}}|^p) < \infty$ with some $ 4< p \leq 8$ and $|a_\mathbf{j}| \leq C \rho^{\|\mathbf{j}\|}$ for some $\rho \in (0,1)$ and $C>0$, then $\mathbb{E}(|V(\mathbf{0})|^{\phi})< \infty$ for some $ 4< \phi \leq 8$, Assumption \ref{assdepenmeasure} with some $p \geq 2$ and  Assumption \ref{assgmc} hold since $\sum_{\mathbf{j} \in \mathbb{Z}^2} \delta_{\mathbf{j},p} = \sum_{\mathbf{j} \in \mathbb{Z}^2} \|V(\mathbf{j})-\widetilde{V}(\mathbf{j})\|_p = \sum_{\mathbf{j} \in \mathbb{Z}^2} |a_\mathbf{j}| \|\varepsilon_{\mathbf{0}}-{\widetilde{\varepsilon}}_{\mathbf{0}}\|_p < \infty$ and $\delta_{\mathbf{j},p}=O(\rho^{\|\mathbf{j}\|})$ respectively.
	\end{example}
	\begin{example}[Volterra Fields]
		For the second order Volterra process defined in \eqref{ex_volterra} in Example \ref{volterra}, if $\mathbb{E}(|\varepsilon_{\mathbf{0}}|^p) < \infty$ for some $ 8< p \leq 16$ and  $a_{\mathbf{s}_1,\mathbf{s}_2}=O(\rho^{\max\{\|\mathbf{s}_1\|,\|\mathbf{s}_2\|\}})$ for some $\rho \in (0,1)$, then we have $\mathbb{E}(|V(\mathbf{0})|^{\phi})< \infty$ for some $ 4< \phi \leq 8$,  and Assumptions \ref{assdepenmeasure} with $p \geq 2$ and \ref{assgmc} hold since $\delta_{\mathbf{j},p}=O(\rho^{\|\mathbf{j}\|})$.
	\end{example}
	\begin{example}[Nonlinear Spatial Autoregressive Models]
		For the nonlinear spatial autoregressive models defined in \eqref{ex_autoreg} in Example \ref{autoreg}, if there exists $\{v(-\mathbf{s})\}_{\mathbf{s} \in \mathcal{N}}$ such that $G(\{v(-\mathbf{s})\}_{\mathbf{s} \in \mathcal{N}} ; \varepsilon_{\mathbf{0}}) \in \mathbb{L}^p$ for some $4 < p \leq 8$. Then, we have $\mathbb{E}(|V(\mathbf{0})|^{p})< \infty$ and $\delta_{\mathbf{j},p}=O(\rho^{\|\mathbf{j}\|})$ for some $\rho \in (0,1)$ which implies Assumption \ref{assdepenmeasure} with $p \geq 2$ and Assumption \ref{assgmc}.
	\end{example}

By Assumptions \ref{assk1}, \ref{assk5}, and the continuity of the spectral density function $f(\mathbf{\lambda})$, Theorem \ref{thm_periodogram}(b) implies the uniform consistency of $\widehat{f}_T(\mathbf{\lambda})$ shown in Theorem \ref{thm3}. 
Note that Theorem \ref{thm_periodogram}(b) requires stronger conditions and thus achieves a stronger result than Theorem \ref{thm3}.
To compare the results of Theorem \ref{thm_periodogram}(b) and Theorem \ref{thm3}, we compare the required moment and linearity conditions of the two theorems. First, the moment condition in Theorem \ref{thm3} is generally weaker than that of Theorem \ref{thm_periodogram}(b). For Theorem \ref{thm3}(a), Assumption \ref{assautocov} with $v=1/2$ is satisfied under the existence of $(4+\delta)$-th moment and short-range weak dependence condition of the underlying fields, for example, mixing conditions, by which some forms of invariance principle or central limit theorem can be established. The moment condition required for Theorem \ref{thm3}(b) is similar to that of Theorem \ref{thm3}(a), whereas Theorem \ref{thm_periodogram}(b) requires the existence of $8$-th moment. Note that Assumption \ref{asslinear} with $r > 1/2$ which is required in Theorem \ref{thm_periodogram} implies Assumption \ref{assautocov} with $v=1/2$ which is required in Theorem \ref{thm3}(a). Second, Theorem \ref{thm3} allows non-linearity of the underlying fields, whereas Theorem \ref{thm_periodogram}(b) requires linearity of the underlying fields. As a result, the required conditions for Theorem \ref{thm_periodogram}(b) is in general stronger than that of Theorem \ref{thm3}. 

\section{Asymptotic Results on Estimated Spatial Fields}\label{sec:asymp_estimated}

In many applications, the Fourier coefficients and kernel spectral density estimators are not applied directly to a stationary spatial data set $\{V(\mathbf{t}): \mathbf{t} \in T\}$, but to an estimate $\{\widehat{V}(\mathbf{t}):\mathbf{t} \in T\}$ computed from the spatial data $\{Y(\mathbf{t}): \mathbf{t} \in T\}$. 
For example, if the observed spatial data is non-stationary, then it is common to apply 
detrending, smoothing, filtering, or spatial regression to obtain a residual process, which is approximately stationary for statistical analysis. The residual process can be regarded as an estimated field of an unobserved stationary random field. 
The following theorem extends the asymptotic results in Theorems \ref{thm3}, \ref{thm2} and \ref{thm_periodogram} to estimated fields. We use a subscript $\widehat{V}$ (resp. $V$) on the notation to indicate the use of $\widehat{V}$ (resp. $V$) in the calculations, e.g. $x_{\widehat{V}}(j), y_{\widehat{V}}(j)$ (resp. $x_{V}(j), y_{V}(j)$) denote the Fourier coefficients based on $\widehat{V}(\cdot)$ (resp. $V(\cdot)$).

\begin{thm}\label{thm_estimated}
	Suppose that for the spatial data $\{Y(\cdot)\}$, we have an estimator $\widehat{V}(\cdot)$ of $V(\cdot)$ such that
	\begin{equation*}
	\frac{1}{|T|} \sum_{\mathbf{t} \in T}\left(V(\mathbf{t})-\widehat{V}(\mathbf{t})\right)^2=o_p\left(\alpha^{-1}_T\right)\,,
	\end{equation*}
	as $\alpha_T \rightarrow \infty$. Furthermore, assume that the kernel function of the spectral density estimator (\ref{kernelesti}) satisfies Assumptions \ref{assk1} and \ref{assk2}; and the kernel function used in defining the weights $p_{\mathbf{s},T}$  in \eqref{def_weight} satisfies Assumptions \ref{assk1} and \ref{assk2}. 	
	Then given $\{Y(\cdot)\}$, we have the followings hold:
	\begin{itemize}
		\item[(a)] Under the setting of Theorem \ref{thm3} and $\alpha_T=O\left(|\mathbf{h}_T|^{-1}\right)$, then we have
	\begin{equation*}
		\sup_{\mathbf{j} \in T}|\widehat{f}_V(\mathbf{\lambda}_\mathbf{j})-\widehat{f}_{\widehat{V}}(\mathbf{\lambda}_\mathbf{j})|= 
		o_p(1)\,.
	\end{equation*}
	
		\item[(b)]
		Under the setting of Theorem \ref{thm2} and $\alpha_T=O\big(|T|^{\frac{q-1}{q}}\big)$ for $q = 4+\epsilon$ with $\epsilon \in (0,1)$ as defined in Theorem \ref{thm2}(e), then we have
		\begin{itemize}
			\item[(i)]
			\begin{equation*}
			\frac{1}{|N|}\sum\limits_{\mathbf{j} \in N}\frac{x_V(\mathbf{j})-x_{\widehat{V}}(\mathbf{j})+y_V(\mathbf{j})-y_{\widehat{V}}(\mathbf{j})}{\sqrt{f(\mathbf{\lambda}_\mathbf{j})}}=o_p(1)\,.
			\end{equation*}
			
			\item[(ii)]
			\begin{equation*}
			\frac{1}{|N|}\sum\limits_{\mathbf{j} \in N}\frac{I_V(\mathbf{j})-I_{\widehat{V}}(\mathbf{j})}{f(\mathbf{\lambda}_\mathbf{j})}=o_p(1)\,.
			\end{equation*}
			
			\item[(iii)]
			\begin{equation*}
			\frac{1}{|N|}\sum\limits_{\mathbf{j} \in N}\frac{I^2_V(\mathbf{j})-I^2_{\widehat{V}}(\mathbf{j})}{f^2(\mathbf{\lambda}_\mathbf{j})} = o_p(1)\,.
			\end{equation*}
			
			\item[(iv)]
			\begin{equation*}
			\frac{1}{|N|}\sum\limits_{\mathbf{j} \in N}\frac{I^q_V(\mathbf{j})-I^q_{\widehat{V}}(\mathbf{j})}{f^q(\mathbf{\lambda}_\mathbf{j})} = o_p(1)\,.
			\end{equation*}	
		\end{itemize}

		\item[(c)]
		Under the setting of Theorem \ref{thm_periodogram} and $\alpha_T=O\left(\left(|T|^{q-1}/|\mathbf{h}_T|\right)^{\frac{1}{q}}\right)$ for $q = 4+\epsilon$ with $\epsilon \in (0,1)$ as defined in Theorem \ref{thm_periodogram}(d), then we have
		\begin{itemize}
			\item[(i)]
			\begin{equation*}
			\sup_{\mathbf{k} \in N}\left|\sum\limits_{\mathbf{j} \in\mathbb{Z}^2}p_{\mathbf{j},T}[x_V(\mathbf{k}+\mathbf{j})-x_{\widehat{V}}(\mathbf{k}+\mathbf{j})+y_V(\mathbf{k}+\mathbf{j})-y_{\widehat{V}}(\mathbf{k}+\mathbf{j})]\right| = o_p(1)\,.
			\end{equation*}
			
			\item[(ii)]
			\begin{equation*}
			\sup_{\mathbf{k} \in N}\left|\sum\limits_{\mathbf{j} \in\mathbb{Z}^2}p_{\mathbf{j},T}(I_V(\mathbf{k}+\mathbf{j})-I_{\widehat{V}}(\mathbf{k}+\mathbf{j}))\right|= o_p(1)\,.
			\end{equation*}
			
			\item[(iii)]
			\begin{equation*}
			\sup_{\mathbf{k} \in N}\sum\limits_{\mathbf{j} \in\mathbb{Z}^2}p_{\mathbf{j},T}(I^2_V(\mathbf{k}+\mathbf{j})-I^2_{\widehat{V}}(\mathbf{k}+\mathbf{j})) = o_p(1)\,.
			\end{equation*}
			
			\item[(iv)]
			\begin{equation*}
			\sup_{\mathbf{k} \in N}\sum\limits_{\mathbf{j} \in\mathbb{Z}^2}p_{\mathbf{j},T}(I^q_V(\mathbf{k}+\mathbf{j})-I^q_{\widehat{V}}(\mathbf{k}+\mathbf{j}))= o_p(1)\,.
			\end{equation*}	
		\end{itemize}
	\end{itemize}
\end{thm}

\section{Proofs}\label{proof}
\subsection{Proof of Theorem \ref{thm21}}

\begin{proof}[Proof of Theorem \ref{thm21}]
	The proof of Theorem \ref{thm21} goes as follows. For presentational clarity we focus on the real parts of $x(\mathbf{j})+iy(\mathbf{j})$. The argument easily extends to general cases which include the complex parts $\{y(\mathbf{j})\}$. We first express the linear combination of Fourier coefficients $\{x(\mathbf{j})\}$ as linear combinations of $\{V(\mathbf{k})\}$, and then we construct a $2l_T$-dependent field $\{\widehat{V}(\mathbf{k})\}$, where $l_T$ is defined below, such that the central limit theorem can be applied to the linear combination of its truncated counterpart $\{U(\mathbf{k})\}$. The error of the above approximations are then examined.
	
	Consider $F_{N1}=\{(\mathbf{j},1): \mathbf{j} \in N\}$ and hence $\widetilde{s}_T(\mathbf{j},1)$ corresponds to the real parts of $x(\mathbf{j})+iy(\mathbf{j})$. For notation simplicity, denote $\widetilde{s}_T(\mathbf{j}) = \widetilde{s}_T(\mathbf{j},1)$.
	 Also, we assume here $\mu=\mathbb{E}(V(\mathbf{0}))=0$. Denote $A=(\mathbf{a}_1,...,\mathbf{a}_q) \subset F_{N1}$.
	Let $H_T=\sum\limits_{\mathbf{k} \in T} \mu_\mathbf{k} V(\mathbf{k})$, where $\mu_\mathbf{k}=\mu_\mathbf{k}(\mathbf{c},A)=\sum\limits^{q}_{i=1}\frac{c_i\cos(\mathbf{k}'\mathbf{\lambda}_{\mathbf{a}_i})}{\sqrt{2\pi^2f(\mathbf{\lambda}_{\mathbf{a}_i})}}$. Since $f_*=\min\limits_{\mathbf{\lambda} \in\mathbb{R}^2}f(\mathbf{\lambda})>0$, there exists $\mu_*$ such that $|\mu_\mathbf{k}| \leq \mu_*$ for all $\mathbf{c} \in\mathbb{R}^q, \|\mathbf{c}\|_2=1$ and $A \subset F_{N1}$. Let 
	\begin{equation*}
	d_T(\mathbf{h}) = \left\{ \begin{array}{lcr}
	\frac{1}{|T|} \sum\limits_{\mathbf{k}\in T:~\mathbf{k}-\mathbf{h} \in T} \mu_\mathbf{k} \mu_{\mathbf{k}-\mathbf{h}}\quad &\mbox{ if } \{ \mathbf{k} \in T: \mathbf{k}-\mathbf{h} \in T\} \neq \emptyset\,,\\
	0 &\mbox{ if }\{\mathbf{k} \in T: \mathbf{k}-\mathbf{h} \in T \} = \emptyset\,,
	\end{array}\right. \ \ \text{for} \ \mathbf{h} \in\mathbb{Z}^2\,,
	\end{equation*}
	where $\emptyset$ is an empty set. Note that $\sum\limits_{\mathbf{k} \in T}\cos(\mathbf{k}'\mathbf{\lambda}_{\mathbf{a}_u})\cos((\mathbf{k}+\mathbf{h})'\mathbf{\lambda}_{\mathbf{a}_v})=\frac{|T|}{2}\cos(\mathbf{h}'\mathbf{\lambda}_{\mathbf{a}_u})\mathds{1}_{\left\{\mathbf{a}_u=\mathbf{a}_v\right\}}\,$.
	Thus, it is easily seen that there exists a constant $K_0>0$ such that for all $\mathbf{h} \in\mathbb{Z}^2$,
	\begin{equation*}
	\tau_T(\mathbf{h})=\sup\limits_{\mathbf{a}_1,...,\mathbf{a}_q \in F_{N1}} \sup\limits_{\mathbf{c} \in\mathbb{R}^q,\|\mathbf{c}\|_2=1} \left|d_T(\mathbf{h})-\sum\limits^{q}_{i=1}c^2_i\frac{\cos(\mathbf{h}'\mathbf{\lambda}_{\mathbf{a}_i})}{4\pi^2f(\mathbf{\lambda}_{\mathbf{a}_i})}\right| \leq \frac{K_0|\mathbf{h}|}{|T|}\,.
	\end{equation*}
	Clearly $\tau_T(\mathbf{h}) \leq \mu^2_*+(4\pi^2f_*)^{-1} := K_1$. By using the equality $f(\mathbf{\lambda}_{\mathbf{a}_i})=\frac{1}{4\pi^2}\sum\limits_{\mathbf{h} \in\mathbb{Z}^2}\gamma(\mathbf{h})\cos(\mathbf{h}'\mathbf{\lambda}_{\mathbf{a}_i})$, hence $\sum\limits_{\mathbf{h} \in\mathbb{Z}^2}\frac{\gamma(\mathbf{h})\cos(\mathbf{h}'\mathbf{\lambda}_{\mathbf{a}_i})}{4\pi^2f(\mathbf{\lambda}_{\mathbf{a}_i})}=1$, we have uniformly over $(\mathbf{a}_1,...,\mathbf{a}_q) \subset F_{N1}$ and $\mathbf{c}$ that 
	\begin{equation*}
	\begin{aligned}
	\left|\frac{\|H_T\|^2_2}{|T|}-1\right| &= \left|\frac{1}{|T|}\mathbb{E}\left(\sum\limits_{\mathbf{k}_1 \in T}\mu_{\mathbf{k}_1}V(\mathbf{k}_1)\sum\limits_{\mathbf{k}_2 \in T}\mu_{\mathbf{k}_2}V(\mathbf{k}_2)\right)-1\right|\\
	&= \left|\frac{1}{|T|}\sum\limits_{\mathbf{k}_1 \in T}\sum\limits_{\mathbf{k}_2 \in T}\mu_{\mathbf{k}_1}\mu_{\mathbf{k}_2}\gamma(\mathbf{k}_1-\mathbf{k}_2)-1\right| \\
	&=\left|\sum\limits_{\mathbf{h} \in\mathbb{Z}^2}d_T(\mathbf{h})\gamma(\mathbf{h})-1\right|\\
	& = \left|\sum\limits_{\mathbf{h} \in\mathbb{Z}^2}d_T(\mathbf{h})\gamma(\mathbf{h})-\sum\limits^{q}_{i=1}c^2_i\right|\\
	& = \left|\sum\limits_{\mathbf{h} \in\mathbb{Z}^2}d_T(\mathbf{h})\gamma(\mathbf{h})-\sum\limits^{q}_{i=1}\left(c^2_i\sum\limits_{\mathbf{h} \in\mathbb{Z}^2}\frac{\gamma(\mathbf{h})\cos(\mathbf{h}'\mathbf{\lambda}_{\mathbf{a}_i})}{4\pi^2f(\mathbf{\lambda}_{\mathbf{a}_i})}\right)\right|\\
	&=\left|\sum\limits_{\mathbf{h} \in\mathbb{Z}^2}\gamma(\mathbf{h})\left(d_T(\mathbf{h})-\sum\limits^{q}_{i=1}c^2_i\frac{\cos(\mathbf{h}'\mathbf{\lambda}_{\mathbf{a}_i})}{4\pi^2f(\mathbf{\lambda}_{\mathbf{a}_i})}\right)\right|
	\leq \sum\limits_{\mathbf{h} \in\mathbb{Z}^2} \tau_T(\mathbf{h})\gamma(\mathbf{h})\\
	&\leq \sum\limits_{\mathbf{h} \in\mathbb{Z}^2} K_2\min \left(\frac{|\mathbf{h}|}{|T|},1\right)\gamma(\mathbf{h})\quad \mbox{where } K_2=2(K_1+K_0)\\
	& \leq \sum\limits_{\mathbf{h} \in H_1}
	K_2\frac{|\gamma(\mathbf{h})|}{|T|^{\frac{1}{2}}} + \sum\limits_{\mathbf{h} \in H_2}K_2|\gamma(\mathbf{h})| 
	\rightarrow 0\,, \quad \text{as} \quad T \rightarrow \infty\,,
	\end{aligned}
	\end{equation*}
	where $H_1=\left\{(h_1,h_2):~h_1<d_1^{\frac{1}{2}}~\text{and}~h_2<d_2^{\frac{1}{2}}\right\}$ and 
	$H_2=\left\{(h_1,h_2):~h_1 \geq d_1^{\frac{1}{2}}~\text{or}~h_2 \geq d_2^{\frac{1}{2}}\right\}$.

	Let $\widehat{H}_T=\sum\limits_{\mathbf{k} \in T}\mu_\mathbf{k}\widehat{V}(\mathbf{k})$, where $\widehat{V}(\mathbf{k})=\mathbb{E}(V(\mathbf{k})|\mathcal{F}_{\mathbf{k},<l_T>})$, $\mathcal{F}_{\mathbf{k},<l_T>}=\{\varepsilon_\mathbf{i}:\|\mathbf{i}-\mathbf{k}\| < l_T\}$, where $\|\mathbf{x}\|=\max\limits_{i=1,2}|x_i|$ for  $\mathbf{x}=(x_1,x_2) \in \mathbb{Z}^2$. Then we have that $\{\widehat{V}(\mathbf{k})\}$ are $2l_T$-dependent and  $\delta_{l_T}=\|V(\mathbf{0})-\widehat{V}(\mathbf{0})\|_2 \rightarrow 0$ as $l_T \rightarrow \infty$, and $\widehat{H}_T$ can be regarded as an approximation of $H_T$. Denote  $\mathcal{F}_{m,n}=\{\varepsilon_\mathbf{i}: \mathbf{i}=(i_1,i_2) \in \mathbb{Z}^2~\text{with}~i_1<m,~i_2<n\}$. 
	
	\renewcommand{\baselinestretch}{1}
	\begin{figure}[hptb]
		\centering \caption{\small {\it Illustration of the locations of $\varepsilon_\mathbf{i}$ included.}}
		{\footnotesize
			\begin{tabular}{|>{\centering\arraybackslash}m{3cm}|c|c|}
				\hline
				Set: & $\mathcal{F}_{\mathbf{k},<l_T>}$ & $\mathcal{F}_{\mathbf{k},<l_T>} \bigcap \mathcal{F}_{0,0}$ \\ \hline
				Locations of $\varepsilon_\mathbf{i}$ included:&
				\resizebox{4.0cm}{4.0cm}{
					\begin{tikzpicture}
					\draw[-latex] (-0.5,0,0) -- (3,0,0);
					\draw[-latex] (0,-0.5,0) -- (0,3,0); 
					\draw[latex-latex] (1.25,1.25,0) -- (1.25,2,0);
					\draw[latex-latex] (1.25,1.25,0) -- (2,1.25,0);    
					\draw[pattern=north west lines, pattern color=gray] (0.5,0.5) rectangle (2,2);
					\node at (1.25,1,0) {$\mathbf{k}$} node at (1,1.5,0) {$l_T$} node at (1.75,1.5,0) {$l_T$};
					\end{tikzpicture}
				} & 
				\resizebox{4.0cm}{4.0cm}{
					\begin{tikzpicture}
					\draw[-latex] (-0.25,-1,0) -- (-0.25,-0.4,0);
					\draw[-latex] (-0.5,0,0) -- (2,0,0);
					\draw[-latex] (0,-0.5,0) -- (0,2,0); 
					\draw[-] (1.5,-0.5,0) -- (1.5,1.5,0);
					\draw[-] (-0.5,1.5,0) -- (1.5,1.5,0);
					\draw[-] (-0.5,0,0) -- (-0.5,1.5,0);
					\draw[-] (-0.5,-0.5,0) -- (1.5,-0.5,0);
					\draw[latex-latex] (0.5,0.5,0) -- (0.5,-0.5,0);
					\draw[latex-latex] (0.5,0.5,0) -- (-0.5,0.5,0);    
					\draw[pattern=north west lines, pattern color=gray] (-0.5,-0.5) rectangle (0,0);
					\node at (0.6,0.7,0) {$\mathbf{k}$} node at (0.7,-0.1,0) {$l_T$} node at (-0.1,0.7,0) {$l_T$} node at (-0.25,-1.3,0) {$\mathcal{F}_{\mathbf{k},<l_T>} \bigcap \mathcal{F}_{0,0}$};
					\end{tikzpicture}
				}\\
				\hline
			\end{tabular}
		}
	\end{figure}
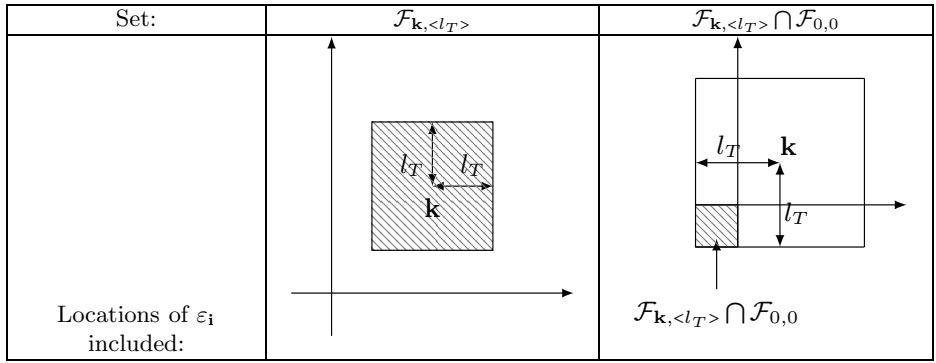
	
	Define projection operator $\mathcal{P}_\mathbf{0}$ for any $\mathcal{F}_{\infty,\infty}$-measurable random variable $X(\mathbf{k})$ as	 
	\begin{equation*}
	\begin{aligned}
	\mathcal{P}_\mathbf{0} (X(\mathbf{k})) &: = 	\mathbb{E}[X(\mathbf{k})|\mathcal{F}_{0,0}]-\mathbb{E}[X(\mathbf{k})|\mathcal{F}_{0,-1}] -\mathbb{E}[X(\mathbf{k})|\mathcal{F}_{-1,0}]+\mathbb{E}[X(\mathbf{k})|\mathcal{F}_{-1,-1}]\,.
	\end{aligned}
	\end{equation*}	
	If $\|\mathbf{k}\|<l_T$, then $\varepsilon_\mathbf{0} \in \mathcal{F}_{\mathbf{k},<l_T>}$, and we have 
	\begin{equation*}
	\begin{aligned} 
	\mathcal{P}_\mathbf{0}(\widehat{V}(\mathbf{k})) & = \mathcal{P}_\mathbf{0}[\mathbb{E}(V(\mathbf{k})|\mathcal{F}_{\mathbf{k},<l_T>})] 
	= \sum_{i,j\in\{-1,0\}} (-1)^{i+j}\mathbb{E}[\mathbb{E}(V(\mathbf{k})|\mathcal{F}_{\mathbf{k},<l_T>})|\mathcal{F}_{i,j}]\\
	&= \sum_{i,j\in\{-1,0\}} (-1)^{i+j}\mathbb{E}[V(\mathbf{k})|\mathcal{F}_{\mathbf{k},<l_T>} \bigcap \mathcal{F}_{i,j}]
	= \mathbb{E}[\mathcal{P}_\mathbf{0}(V(\mathbf{k}))|\mathcal{F}_{\mathbf{k},<l_T>} \bigcap \mathcal{F}_{0,0}]\,.
	\end{aligned}
	\end{equation*}	
	If $\|\mathbf{k}\| \geq l_T$, then $\varepsilon_\mathbf{0} \notin \mathcal{F}_{\mathbf{k},<l_T>}$, so
    $\mathcal{P}_\mathbf{0}(\widehat{V}(\mathbf{k}))=0\,.$
	Clearly, $\|\mathcal{P}_\mathbf{0}(V(\mathbf{k})-\widehat{V}(\mathbf{k}))\|_2 \leq 2\delta_{l_T}$ for all $\mathbf{k} \in\mathbb{Z}^2$.
		
	Since by Lemma 1 of \cite{EM2013}, we have $\|\mathcal{P}_\mathbf{0} V(\mathbf{j})\|_2 \leq \delta_{\mathbf{j},p}$ for all $\mathbf{j} \in \mathbb{Z}^2$, and hence by Assumption \ref{assdepenmeasure} with $p \geq 2$ and the Lebesgue dominated convergence theorem, it entails that 
	\begin{equation*}
	\begin{aligned}
	\frac{\|H_T-\widehat{H}_T\|_2}{\sqrt{|T|}} &= \left[\frac{1}{|T|}\sum\limits_{\mathbf{j} \in\mathbb{Z}^2}\|\mathcal{P}_\mathbf{j}(H_T-\widehat{H}_T)\|_2^2\right]^{\frac{1}{2}} 
= \left[\frac{1}{|T|}\sum\limits_{\mathbf{j} \in\mathbb{Z}^2}\left\|\mathcal{P}_\mathbf{j}\left\{\sum\limits_{\mathbf{k} \in T} \mu_\mathbf{k}(V(\mathbf{k})-\widehat{V}(\mathbf{k}))\right\}\right\|_2^2\right]^{\frac{1}{2}}\\
	& \leq\left [\frac{\mu^2_*}{|T|}\sum\limits_{\mathbf{j} \in\mathbb{Z}^2}\left\|\mathcal{P}_\mathbf{j}\sum\limits_{\mathbf{k} \in T}(V(\mathbf{k})-\widehat{V}(\mathbf{k}))\right\|^2_2\right]^{\frac{1}{2}}\\
	&\leq\left [\frac{\mu^2_*}{|T|}\sum\limits_{\mathbf{k} \in T}\sum\limits_{\mathbf{j} \in\mathbb{Z}^2}\left\|\mathcal{P}_\mathbf{j}(V(\mathbf{k})-\widehat{V}(\mathbf{k}))\right\|_2^2\right]^{\frac{1}{2}} 
= \left[\mu^2_*\sum\limits_{\mathbf{j} \in\mathbb{Z}^2}\left\|\mathcal{P}_\mathbf{0}(V(\mathbf{j})-\widehat{V}(\mathbf{j}))\right\|_2^2\right]^{\frac{1}{2}}\\ 
	&\leq \mu_*\sum\limits_{\mathbf{j} \in\mathbb{Z}^2}\left\|\mathcal{P}_\mathbf{0}(V(\mathbf{j})-\widehat{V}(\mathbf{j}))\right\|_2
    \leq \mu_*\sum\limits_{\mathbf{j} \in\mathbb{Z}^2}2\min(\|\mathcal{P}_\mathbf{0} V(\mathbf{j})\|_2,\delta_{l_T})\\
	&\rightarrow 0\,, \quad \text{as} \quad l_T \rightarrow \infty \,.
	\end{aligned}
	\end{equation*}
	Let $g_T(r)=r^2\mathbb{E}(\widehat{V}(\mathbf{k})^2\mathds{1}_{\{|\widehat{V}(\mathbf{k})| \geq \sqrt{|T|}/r\}})$. Since $\mathbb{E}(\widehat{V}(\mathbf{k})^2) < \infty$, $\lim\limits_{T \rightarrow \infty} g_T(r)=0$ for any fixed $r > 0$. Note that $g_T$ is non-decreasing in $r$. Then, there exists a sequence $r_T \uparrow \infty$ such that $g_T(r_T) \rightarrow 0$.
	Let $U(\mathbf{k})=\widehat{V}(\mathbf{k})\mathds{1}_{\{|\widehat{V}(\mathbf{k})| \leq \sqrt{|T|}/r_T\}}$ and $H_{T,U}=\sum\limits_{\mathbf{k} \in T} \mu_\mathbf{k} U(\mathbf{k})$. Then $\left\|U(\mathbf{k})-\widehat{V}(\mathbf{k})\right\|_2=o\left({1/r_T}\right)$. Since $U(\mathbf{k})-\widehat{V}(\mathbf{k})$ are $2l_T$-dependent,
	\begin{equation}\label{ADeq8}
	\left\|H_{T,U}-\widehat{H}_T\right\|_2 \leq \sum\limits_{\mathbf{a} \in L}\left\|\sum \mu_\mathbf{b}\left(U(\mathbf{b})-\widehat{V}(\mathbf{b})\right)\right\|_2=o\left(\frac{\sqrt{|T|}}{r_T}\right)\,,
	\end{equation}
	where $L=\{(l_1,l_2):1 \leq l_1 \leq 2l_T, 1 \leq l_2 \leq 2l_T\}$ with $|L|=4l_T^2$,  
	 and the inner sum is over $\{\mathbf{b} \in T: \mathbf{b}-\mathbf{a}=(2ml_T,2nl_T), m, n\in\mathbb{Z}^+_0=\{0,1,2,...\}\}$. See Figure \ref{partition_L} for a graphical illustration of the partition of the rectangular lattice $T$ and the set $L$.
	
	
	\begin{figure}[!htb]
		\begin{minipage}{.5\linewidth}
			\centering \caption{\small {\it The partition of the rectangular lattice $T$}}
			{\footnotesize
				\begin{tabular}{|c|}
					\hline
					\resizebox{5.0cm}{5.0cm}{
						\begin{tikzpicture}
						\draw[-latex] (-0.5,0,0) -- (2,0,0);
						\draw[-latex] (0,-0.5,0) -- (0,2,0); 
						\draw[-] (0.5,0) -- (0.5,2,0);
						\draw[-] (1,0) -- (1,2,0);
						\draw[-] (1.5,0) -- (1.5,2,0);
						\draw[-] (0,0.5,0) -- (2,0.5,0);
						\draw[-] (0,1.5,0) -- (2,1.5,0);
						\draw[-] (0,1,0) -- (2,1,0);
						\node at (0.25,0.25,0) {$L$} node at (-0.25,0.5,0) {$2l_T$} node at (0.5,-0.25,0) {$2l_T$} node at (-0.25,1,0) {$4l_T$} node at (1,-0.25,0) {$4l_T$} node at (-0.25,1.5,0) {$6l_T$} node at (1.5,-0.25,0) {$6l_T$};
						\end{tikzpicture}
					}\\
					\hline
				\end{tabular}\label{partition_L}
			}
		\end{minipage}%
		\begin{minipage}{.5\linewidth}
						\centering \caption{\small {\it The locations of the blocks $B_\mathbf{m}$}}
			{\footnotesize
				\begin{tabular}{|c|}
					\hline
					\resizebox{5.0cm}{5.0cm}{
						\begin{tikzpicture}
						\draw[-latex] (-0.5,0,0) -- (3,0,0);
						\draw[-latex] (0,-0.5,0) -- (0,3,0); 
						\draw[-] (0,0) rectangle (0.5,0.5);
						\draw[-] (0,1) rectangle (0.5,1.5);
						\draw[-] (0,2) rectangle (0.5,2.5);
						\draw[-] (1,0) rectangle (1.5,0.5);
						\draw[-] (1,1) rectangle (1.5,1.5);
						\draw[-] (1,2) rectangle (1.5,2.5);
						\draw[latex-latex] (0.5,0.25,0) -- (1,0.25,0);
						\draw[latex-latex] (0.25,0.5,0) -- (0.25,1,0);
						\draw[latex-latex] (1.25,0.5,0) -- (1.25,1,0);
						\draw[latex-latex] (0.25,1.5,0) -- (0.25,2,0);
						\draw[latex-latex] (1.25,1.5,0) -- (1.25,2,0);
						\draw[latex-latex] (0.5,1.25,0) -- (1,1.25,0);
						\draw[latex-latex] (0.5,2.25,0) -- (1,2.25,0);
						\node at (0.27,0.27,0) {$B$} node at (0.27,1.27,0) {$B$} node at (0.27,2.27,0) {$B$} node at (1.27,0.27,0) {$B$} node at (1.27,1.27,0) {$B$} node at (1.27,2.27,0) {$B$} node at (0.75,0.45,0) {$2l_T$} node at (0.75,1.45,0) {$2l_T$} node at (0.75,2.45,0) {$2l_T$} node at (0.47,0.77,0) {$2l_T$} node at (0.47,1.77,0) {$2l_T$}  node at (1.47,0.77,0) {$2l_T$} node at (1.47,1.77,0) {$2l_T$}  node at (2.1,0.25,0) {.......} node at (2.1,1.25,0) {.......} node at (2.1,2.25,0) {.......} node at (2.1,2.75,0) {.......} node at (0.25,2.75,0) {.} node at (0.25,2.55,0) {.} node at (0.25,2.95,0) {.} node at (1.25,2.75,0) {.} node at (1.25,2.55,0) {.} node at (1.25,2.95,0) {.};
						\end{tikzpicture}
					}\\
					\hline
				\end{tabular}\label{locations_B}
			}
		\end{minipage} 
	\end{figure}

	Let $p_T=\left\lfloor {r_T}^{\frac{1}{8}} \right\rfloor$ and blocks 
	\begin{eqnarray*}
		B_\mathbf{m}  = \{\mathbf{a}=(a_1,a_2) \in\mathbb{Z}^2: 1+(m_k-1)(p_T+2l_T) \leq a_k \leq p_T+(m_k-1)(p_T+2l_T)\,,~~\text{for}~k=1,2\}\,,
	\end{eqnarray*}
	where $\mathbf{m}=(m_1,m_2) \in M_B$, and 
	\begin{eqnarray*}
		M_B =  \left\{(m_1,m_2) \in \mathbb{N}^2: 1 \leq m_1 \leq m_{T_1}=\left\lfloor 1+\frac{d_1-p_T}{p_T+2l_T}\right\rfloor, 1 \leq m_2 \leq m_{T_2}=\left\lfloor 1+\frac{d_2-p_T}{p_T+2l_T}\right\rfloor\right\}.
	\end{eqnarray*}
	See Figure \ref{locations_B} for a graphical illustration of the locations of blocks in set $B_\mathbf{m}$.
	
	Define $T_\mathbf{m} = \sum\limits_{\mathbf{k} \in B_\mathbf{m}} \mu_\mathbf{k} U(\mathbf{k})$, 
	$S_T = \sum\limits_{\mathbf{m} \in M_B} T_\mathbf{m}$, 
	$R_T = H_{T,U}-S_T$, 
	$W = \frac{S_T-\mathbb{E}(S_T)}{\sqrt{|T|}}$, 
	and $\Delta = \frac{\widehat{H}_T}{\sqrt{|T|}}-W$.	Then, $T_\mathbf{m}$ are independent and $\|R_T\|_2=O\left(\sqrt{|M_B|}\right)$
	since $U(\mathbf{k})$ are $2l_T$-dependent. Note that $|\mathbb{E}(S_T)|=o\left(\frac{\sqrt{|T|}}{r_T}\right)$ with similar argument in (\ref{ADeq8}).
	Thus, by (\ref{ADeq8}),
	\begin{equation} \label{ADeq9}
	\begin{aligned}
	\sqrt{|T|} \|\Delta\|_2 &\leq |\mathbb{E}(S_T)|+\|S_T-\widehat{H}_T\|_2 \leq o\left(\frac{\sqrt{|T|}}{r_T}\right)+\|S_T-H_{T,U}\|_2+\|H_{T,U}-\widehat{H}_T\|_2\\
	&\leq o\left(\frac{\sqrt{|T|}}{r_T}\right)+O\left(\sqrt{|M_B|}\right)+o\left(\frac{\sqrt{|T|}}{r_T}\right)=O\left(\sqrt{|M_B|}\right)\,.
	\end{aligned}
	\end{equation}
	
	Since $|T_\mathbf{m}|^3 \leq \mu^3_*p^4_T\sum\limits_{\mathbf{k} \in B_\mathbf{m}}|U(\mathbf{k})|^3$ and $\mathbb{E}(U(\mathbf{k})^2) \leq \mathbb{E}(V(\mathbf{k})^2)$, we have
	$\mathbb{E}(|T_\mathbf{m}|^3)=O \left(p^6_T\frac{\sqrt{|T|}}{r_T} \right)$. 
	By the Berry-Esseen theorem, we have
	\begin{equation}\label{ADeq10}
	\begin{aligned}
	\sup_x \left|\mathbb{P}(W \leq x)-\Phi\left(\frac{x}{\|W\|_2}\right) \right| &\leq C\sum\limits_{\mathbf{m} \in M_{B}} \mathbb{E}(|T_\mathbf{m}|^3)\|S_T-\mathbb{E}(S_T)\|_2^{-3}\\
	&= O \left(|M_B|p^6_T\frac{\sqrt{|T|}}{r_T}\right) \times |T|^{-\frac{3}{2}} = O(p^{-4}_T)\,.
	\end{aligned}
	\end{equation}	
	Let $\delta=\delta_T=p^{-\frac{1}{2}}_T$. By (\ref{ADeq9}), (\ref{ADeq10}) and 
	\noindent
	\begin{equation*}
	\mathbb{P}(W \leq w-\delta)-\mathbb{P}(|\Delta| \geq \delta) \leq \mathbb{P}(W+\Delta \leq w) \leq \mathbb{P}(W \leq w+\delta)+\mathbb{P}(|\Delta| \geq \delta)\,,
	\end{equation*}
	we have 
	\begin{equation}\label{ADeq11}
	\sup_x \left|\mathbb{P}\left(\widehat{H}_T \leq \sqrt{|T|}x\right)- \Phi \left(\sqrt{|T|}x/\| \widehat{H}_T\|_2\right)\right|=O\left(p^{-4}_T+\mathbb{P}(|\Delta| \geq \delta)+ \delta +\delta^2 \right) = O( \delta)\,.
	\end{equation}
	Note that $\sup_x\left|\Phi\left(\frac{x}{\sigma_1}\right)-\Phi\left(\frac{x}{\sigma_2}\right)\right| \leq C\left|\left(\frac{\sigma_1}{\sigma_2}-1\right)\right|$ holds for some constant $C$. 
	Let $W_1=\frac{\widehat{H}_T}{\sqrt{|T|}}$, $\Delta_1=\frac{H_T-\widehat{H}_T}{\sqrt{|T|}}$ and 
	$\eta = \eta_{l_T,T} = \left(\frac{\|H_T-\widehat{H}_T\|_2}{\sqrt{|T|}}\right)^{\frac{1}{2}}$. 
	Applying (\ref{ADeq11}) with $w$, $\Delta$ replaced by $W_1$, $\Delta_1$, we have
	\begin{equation*}
	\sup_x \left|\mathbb{P}\left(\frac{H_T}{\sqrt{|T|}} \leq x\right)-\Phi \left(\frac{x}{\|H_T\|_2/\sqrt{|T|}} \right) \right|=O\left(\mathbb{P}(|\Delta_1| \geq \eta)+\delta+\eta+\eta^2\right)\,.
	\end{equation*}
	Thus, the conclusion follows by first letting $T \rightarrow \infty$ and then $l_T \rightarrow \infty.$
\end{proof}

\subsection{Proof of Corollary \ref{coro21}}

\begin{proof}[Proof of Corollary \ref{coro21}]
	The results of Corollary \ref{coro21} directly follow by the results of Theorem \ref{thm21} and applications of Levy's continuity theorem as follows. Denote 
	\begin{equation*}
	\widetilde{s}_T(\mathbf{j},k) = \left\{ \begin{array}{lcr}
	\frac{x(\mathbf{j})}{\sqrt{2\pi^2f(\mathbf{\lambda}_\mathbf{j})}}, & \mathbf{j} \in N, & k=1 \,,\\
	\frac{y(\mathbf{j})}{\sqrt{2\pi^2f(\mathbf{\lambda}_\mathbf{j})}}, & \mathbf{j} \in N, & k=2 \,.
	\end{array}\right.
	\end{equation*}
	and $F_N:=\{(\mathbf{j},k): \mathbf{j} \in N,~k=1,2\}$ as in Theorem \ref{thm21}.  The uniform asymptotic normality of vectors of $\widetilde{s}_T(\cdot)$ is proven in Theorem \ref{thm21}. We will give the arguments only for vectors of length 2, but the same results hold true for any length $p$ by induction. Precisely, we will prove that for each $\mathbf{z}=(z_1, z_2)$,
	\begin{equation}\label{ADeq12}
	\sup\limits_{\substack{\mathbf{a}_1,\mathbf{a}_2 \in F_N, \mathbf{a}_1 \neq \mathbf{a}_2}}\left|\mathbb{P}((\widetilde{s}_T(\mathbf{a}_1) \leq z_1, \widetilde{s}_T(\mathbf{a}_2) \leq z_2))-\Phi(z_1)\Phi(z_2)\right|=o(1)\,.
	\end{equation}
	
	Now, vectorize  $\widetilde{S}_{T,\mathbf{a}_1,\mathbf{a}_2}=(\widetilde{s}_T(\mathbf{a}_1), \widetilde{s}_T(\mathbf{a}_2)), \mathbf{a}_1,\mathbf{a}_2 \in F_N$, where $N$ is the corresponding subset of $T$, to form a single sequence $S_t=(S_t(1),S_t(2))$, $t \in \mathbb{N}$ in such a way that if $S_{t_1}$ corresponds to $\widetilde{S}_{T_1, \mathbf{a}_{11},\mathbf{a}_{12}}$ for some $\mathbf{a}_{11},\mathbf{a}_{12} \in F_{N_1}$, and $S_{t_2}$ corresponds to $\widetilde{S}_{T_2, \mathbf{a}_{21},\mathbf{a}_{22}}$ for some $\mathbf{a}_{21},\mathbf{a}_{22} \in F_{N_2}$, then $|T_1| \leq |T_2|$ implies that $t_1 \leq t_2$. By Levy's continuity theorem and Theorem \ref{thm21}, it holds for each $\mathbf{z}=(z_1, z_2)$ that
	\begin{equation*}
	\begin{aligned}
	\phi_{S_\mathbf{t}}(\mathbf{z}) &= \mathbb{E}(e^{i\mathbf{z}'S_t}) = \mathbb{E}\left(e^{i|\mathbf{z}|\left(\frac{\mathbf{z}}{|\mathbf{z}|}\right)'S_t}\right)= \phi_{\frac{\mathbf{z}'S_t}{|\mathbf{z}|}} (|\mathbf{z}|) \\
& \rightarrow \phi_{G_1}(|\mathbf{z}|) = \mathbb{E}\left(e^{i|\mathbf{z}|G_1}\right) = \mathbb{E}\left(e^{i|\mathbf{z}|\left(\frac{\mathbf{z}}{|\mathbf{z}|}\right)'(G_1,G_2)}\right) = \phi_{(G_1,G_2)}(\mathbf{z})\,,
	\end{aligned}
	\end{equation*}
	where ${\phi_X}$ denote the characteristic function of $X$, and $G_1, G_2$ are two independent standard normal random variables.
	Next, a second application of Levy's continuity theorem yields
	\begin{equation*}
	|\mathbb{P}(S_t(1) \leq z_1, S_t(2) \leq z_2)-\Phi(z_1)\Phi(z_2)|=o(1)\,,
	\end{equation*}
	and by definition of $S_t$ we get (\ref{ADeq12}).
	
	Consider $F_{N1}=\{(\mathbf{j},1): \mathbf{j} \in N\}$ and hence $\widetilde{s}_T(\mathbf{j},1)$ corresponds to the real parts of $x(\mathbf{j})+iy(\mathbf{j})$. For notation simplicity, denote $\widetilde{s}_T(\mathbf{j}) = \widetilde{s}_T(\mathbf{j},1)$.
	The argument easily extends to general cases. Define $\mathbb{P}_\mathbf{j}(z)=\mathbb{P}(\widetilde{s}_T(\mathbf{j}) \leq z)$ and $\mathbb{P}_{\mathbf{j}_1,\mathbf{j}_2}(z)=\mathbb{P}(\widetilde{s}_T(\mathbf{j}_1) \leq z, \widetilde{s}_T(\mathbf{j}_2) \leq z).$ Then it holds by (\ref{ADeq12}) and Theorem \ref{thm21} that
	\begin{equation*}
	\begin{aligned}
	&\left|\mathbb{E}\left(\sum\limits_{\mathbf{j} \in N} w_{\mathbf{j},N}\mathds{1}_{\{\widetilde{s}_T(\mathbf{j}) \leq z\}}\right)-\Phi(z)\right| \leq \sup\limits_{\mathbf{l} \in F_{N1}}|\mathbb{P}_\mathbf{l}(z)-\Phi(z)|\sum\limits_{\mathbf{j} \in N}w_{\mathbf{j},N} = o(1)\,,
	\end{aligned}
	\end{equation*}	
	and	
	\begin{equation*}
	\begin{aligned}			
	& \left|\mathbb{E}\left(\sum\limits_{\mathbf{j} \in N} w_{\mathbf{j},N}\mathds{1}_{\{\widetilde{s}_T(\mathbf{j}) \leq z\}}\right)^2-\Phi^2(z) \right|\\
	\leq\ & \sup\limits_{\substack{\mathbf{j}_1,\mathbf{j}_2 \in N \\ \mathbf{j}_1 \neq \mathbf{j}_2}} |\mathbb{P}_{\mathbf{j}_1,\mathbf{j}_2}(z)-\Phi^2(z)|\sum\limits_{\substack{\mathbf{j}_1,\mathbf{j}_2 \in N \\ \mathbf{j}_1 \neq \mathbf{j}_2}} w_{\mathbf{j}_1,N} w_{\mathbf{j}_2,N}  + \left(\sup\limits_{\mathbf{l} \in F_{N1}}|\mathbb{P}_\mathbf{l}(Z)-\Phi(z)|+|\Phi(z)-\Phi^2(z)|\right) \sum\limits_{\mathbf{j} \in N}{w^2_{\mathbf{j},N}}\\
	=\ & o(1)\,,
	\end{aligned}
	\end{equation*}
	which remains true uniformly in $\mathbf{s} \in S$ if $\{\{w_{\mathbf{j},N,\mathbf{s}}:\mathbf{j} \in N\}:\mathbf{s} \in S\}$ is a class of weights indexed by a countable index set $S$ satisfying $\sum_{\mathbf{j} \in N} w_{\mathbf{j},N,\mathbf{s}}=1$ for all $\mathbf{s} \in S$ and $\sup_{\mathbf{s} \in S}\sum_{\mathbf{j} \in N} w^2_{\mathbf{j},N,\mathbf{s}} \rightarrow 0$.
	Since 
	\begin{equation*}
	\begin{aligned}
	&\mathbb{E} \left(\sum\limits_{\mathbf{j} \in N}w_{\mathbf{j},N}\mathds{1}_{\left\{\widetilde{s}_T(\mathbf{j})\leq z\right\}}-\Phi(z)\right)^2\\
	=\ &\mathbb{E} \left(\sum\limits_{\mathbf{j} \in N} w_{\mathbf{j},N} \mathds{1}_{\left\{\widetilde{s}_T(\mathbf{j})\leq z\right\}}\right)^2-\Phi^2(z) -2\Phi(z)\left[\mathbb{E}\left(\sum\limits_{\mathbf{j} \in N} w_{\mathbf{j},N}\mathds{1}_{\{\widetilde{s}_T(\mathbf{j})\leq z\}}\right)-\Phi(z)\right] 
	=  o(1)\,,
	\end{aligned}
	\end{equation*}
	we get both assertions by the Chebyshev inequality, and the uniformity in $z$ follows from the continuity of $\Phi(z)$.
\end{proof}

\subsection{Proof of Theorem \ref{thm2} with lemmas}
Before we prove Theorem \ref{thm2}, we need the following lemmas.

\begin{lem}\label{lem1_thm5_1}
	Suppose that Assumption \ref{asslinear} with $r = 1/2$ holds, then
	\begin{equation}\label{Clemcond1}
	\max_{\mathbf{j},\mathbf{k} \in N, \mathbf{j}\neq\mathbf{k}}\left| \text{\rm {Cov}} (I(\mathbf{j}),I(\mathbf{k}))\right|=O\left(\frac{1}{|T|}\right)\,,
	\end{equation}
	\begin{equation}\label{Clemcond11}
	\max_{\mathbf{j} \in N}\left| \text{\rm {Cov}} (I(\mathbf{j}),I(\mathbf{j}))-(4\pi^2f(\lambda_\mathbf{j}))^2\right|=O\left(\frac{1}{\sqrt{|T|}}\right)\,,
	\end{equation}
	\begin{equation}\label{Clemcond2}
	\max_{\mathbf{j},\mathbf{k} \in N, \mathbf{j}\neq\mathbf{k}}\left| \text{\rm {Cov}} (I^2(\mathbf{j}),I^2(\mathbf{k}))\right|=O\left(\frac{1}{|T|}\right)\,.
	\end{equation}
	\begin{equation}\label{Clemcond22}
	\max_{\mathbf{j} \in N}\left| \text{\rm {Cov}} (I^2(\mathbf{j}),I^2(\mathbf{j}))-4(4\pi^2f(\lambda_\mathbf{j}))^4\right|=O\left(\frac{1}{\sqrt{|T|}}\right)\,.
	\end{equation}
\end{lem}

\begin{proof}
%
	An analogous proof of Theorem 10.3.2(ii) in \cite{brockwell1991time} yields \eqref{Clemcond1} and \eqref{Clemcond11} for linear random fields, i.e., $V(\mathbf{j})-\mu=\sum_{\mathbf{s} \in\mathbb{Z}^2}a_\mathbf{s} \varepsilon_{\mathbf{j}-\mathbf{s}}$, where $\{\varepsilon_\mathbf{i}\}_{\mathbf{i} \in\mathbb{Z}^2}$ is an i.i.d. random field with $\mathbb{E}(\varepsilon_{\mathbf{0}})=0$, existence of $4$-th moments $\mathbb{E}(\varepsilon_{\mathbf{0}}^4) < \infty$ and $\sum_{\mathbf{s} \in\mathbb{Z}^2} |a_\mathbf{s}| |\mathbf{s}|^{1/2} < \infty$. Under the existence of $8$-th moments $\mathbb{E}(\varepsilon_{\mathbf{0}}^8) < \infty$,
	one can also show that \eqref{Clemcond2} and \eqref{Clemcond22} hold. Hence, Assumption \ref{asslinear} with $r = 1/2$ yields the results.
\end{proof}

\begin{lem}\label{lem1_thm5_2}
	Suppose that Assumption \ref{assgmc} with $\mathbb{E}\left(|V(\mathbf{0})|^{16}\right)< \infty$, then
	\begin{equation*}
	\max_{\mathbf{j},\mathbf{k} \in N}\left| \text{\rm {Cov}} (I(\mathbf{j}),I(\mathbf{k}))-(4\pi^2f(\lambda_\mathbf{j}))^2\delta_{\mathbf{j},\mathbf{k}}\right|=O\left(\frac{1}{|T|}\right)\,,
	\end{equation*}
	\begin{equation*}
	\max_{\mathbf{j},\mathbf{k} \in N}\left| \text{\rm {Cov}} (I^2(\mathbf{j}),I^2(\mathbf{k}))-4(4\pi^2f(\lambda_\mathbf{j}))^4\delta_{\mathbf{j},\mathbf{k}}\right|=O\left(\frac{1}{|T|}\right)\,.
	\end{equation*}
\end{lem}

\begin{proof}
	Assumption \ref{assgmc} with $\mathbb{E}\left(|V(\mathbf{0})|^{16}\right)< \infty$ yields a cumulant condition derived in Lemma \ref{AlemC2} for $k \leq 16$ shown below. The cumulant condition derived in Lemma \ref{AlemC2} implies an exponential decay of the auto-covariance function, and hence the auto-covariance function $\gamma(\cdot)$ satisfies $\sum_{\mathbf{j}\in \mathbb{Z}^2} |\mathbf{j}| |\gamma(\mathbf{j})|< \infty$, and hence an analogous proof of Lemma A.4 in \cite{SX2007} yields the results. 
\end{proof}

\begin{proof}[Proof of Theorem \ref{thm2}]
	For assertion (a), we show that
	\begin{eqnarray}\label{Clemeq8}
	\begin{aligned}
	\sup\limits_{\mathbf{l},\mathbf{k} \in N}|\text{Cov}(x(\mathbf{l}),x(\mathbf{k}))-2\pi^2f(\mathbf{\lambda}_\mathbf{k})\delta_{\mathbf{l},\mathbf{k}}| & \rightarrow & 0\,, \\
	\sup\limits_{\mathbf{l},\mathbf{k} \in N}|\text{Cov}(y(\mathbf{l}),y(\mathbf{k}))-2\pi^2f(\mathbf{\lambda}_\mathbf{k})\delta_{\mathbf{l},\mathbf{k}}| & \rightarrow & 0\,.
	\end{aligned}
	\end{eqnarray}
	Note that
	\begin{equation}\label{Clemeq9}
	\left\{ \begin{array}{lcr}
	\text{Cov}(x(\mathbf{l}),x(\mathbf{k}))+\text{Cov}(y(\mathbf{l}),y(\mathbf{k}))=\text{Re}\left(\frac{1}{|T|}\sum\limits_{\mathbf{j},\mathbf{s} \in T}e^{-i(\mathbf{j}'\mathbf{\lambda}_\mathbf{l}-\mathbf{s}'\mathbf{\lambda}_\mathbf{k})}\text{Cov}(V(\mathbf{j}),V(\mathbf{s}))\right) \,,\\
	\text{Cov}(x(\mathbf{l}),x(\mathbf{k}))-\text{Cov}(y(\mathbf{l}),y(\mathbf{k}))=\text{Re}\left(\frac{1}{|T|}\sum\limits_{\mathbf{j},\mathbf{s} \in T}e^{-i(\mathbf{j}'\mathbf{\lambda}_\mathbf{l}+\mathbf{s}'\mathbf{\lambda}_\mathbf{k})}\text{Cov}(V(\mathbf{j}),V(\mathbf{s}))\right) \,.
	\end{array}\right.
	\end{equation}
	Furthermore, it holds that
	$\sum\limits_{\mathbf{j} \in T}e^{-i\mathbf{j}'(\mathbf{\lambda}_\mathbf{l}+\mathbf{\lambda}_\mathbf{k})}=\left(\sum\limits^{d_1}_{j_1=1}e^{-i\frac{2\pi(l_1+k_1)}{d_1}j_1}\right)\left(\sum\limits^{d_2}_{j_2=1}e^{-i\frac{2\pi(l_2+k_2)}{d_2}j_2}\right)=0\,$.\\
	Hence
	\begin{equation}\label{Clemeq10}
	\begin{aligned}
	&\left|\frac{1}{|T|}\sum\limits_{\mathbf{j},\mathbf{s} \in T}e^{-i(\mathbf{j}'\mathbf{\lambda}_\mathbf{l}+\mathbf{s}'\mathbf{\lambda}_\mathbf{k})}\text{Cov}(V(\mathbf{j}),V(\mathbf{s}))\right|\\
	=\ &\left|\frac{1}{|T|}\sum\limits_{\mathbf{j} \in T}e^{-i\mathbf{j}'(\mathbf{\lambda}_\mathbf{l}+\mathbf{\lambda}_\mathbf{k})}\left(\sum\limits_{\mathbf{s} \in T}e^{-i(\mathbf{s}-\mathbf{j})'\mathbf{\lambda}_\mathbf{k}}\gamma(\mathbf{s}-\mathbf{j})\right)\right|\\
	=\ &\left|\frac{1}{|T|}\sum\limits_{\mathbf{j} \in T}e^{-i\mathbf{j}'(\mathbf{\lambda}_\mathbf{l}+\mathbf{\lambda}_\mathbf{k})}\left(\sum\limits_{\mathbf{s} \in T}e^{-i(\mathbf{s}-\mathbf{j})'\mathbf{\lambda}_\mathbf{k}}\gamma(\mathbf{s}-\mathbf{j})-4\pi^2f(\mathbf{\lambda}_\mathbf{k})\right)\right|\\
	=\ &\left|\frac{1}{|T|}\sum\limits_{\mathbf{j} \in T}e^{-i\mathbf{j}'(\mathbf{\lambda}_\mathbf{l}+\mathbf{\lambda}_\mathbf{k})}\left(\sum\limits_{\mathbf{s} \in T}e^{-i(\mathbf{s}-\mathbf{j})'\mathbf{\lambda}_\mathbf{k}}\gamma(\mathbf{s}-\mathbf{j})-\sum
	\limits_{\mathbf{s} \in\mathbb{Z}^2}e^{-i(\mathbf{s}-\mathbf{j})'\mathbf{\lambda}_\mathbf{k}}\gamma(\mathbf{s}-\mathbf{j})\right)\right|\\
	=\ &\left|\frac{1}{|T|}\sum\limits_{\mathbf{j} \in T}e^{-i\mathbf{j}'(\mathbf{\lambda}_\mathbf{l}+\mathbf{\lambda}_\mathbf{k})}\left(\sum\limits_{\mathbf{s} \in {\mathbb{Z}^2 \backslash T}} -e^{-i(\mathbf{s}-\mathbf{j})'\mathbf{\lambda}_\mathbf{k}}\gamma(\mathbf{s}-\mathbf{j})\right)\right|\\
	\leq\ &\frac{1}{|T|}\sum\limits_{\mathbf{j} \in T}\sum\limits_{\mathbf{s} \in {\mathbb{Z}^2 \backslash T}}|\gamma(\mathbf{s}-\mathbf{j})| = \frac{1}{|T|}\sum\limits_{\mathbf{j} \in T}\left(\sum\limits_{\mathbf{s} \in {\mathbb{Z}^2 \backslash T^*}}|\gamma(\mathbf{s}-\mathbf{j})|+\sum\limits_{\mathbf{s} \in {T^* \backslash T}}|\gamma(\mathbf{s}-\mathbf{j})|\right)\\
	\leq\ &\sum\limits_{\mathbf{s} \in {\mathbb{Z}^2 \backslash T}^{**}}|\gamma(\mathbf{s})|+\frac{d_1\sqrt{d_2}}{d_1d_2}+\frac{d_2\sqrt{d_1}}{d_1d_2}+\frac{\sqrt{d_1}\sqrt{d_2}}{d_1d_2}
	=o(1)\,,
	\end{aligned}
	\end{equation}\label{Clemeq11}
where $T^*=\{(t_1,t_2): t_k \in\mathbb{Z}, 1 \leq t_k \leq d_k+\sqrt{d_k}, k=1,2\}$ and $T^{**}=\{(t_1,t_2): t_k \in\mathbb{Z}, 1 \leq t_k \leq \sqrt{d_k},k=1,2\}$, 
	uniformly in $\mathbf{l},\mathbf{k}$ by the absolute summability of the auto-covariance function. Analogously, we have, uniformly for $\mathbf{l} \neq \mathbf{k},$ i.e. $\mathbf{\lambda}_\mathbf{l} \neq \mathbf{\lambda}_\mathbf{k}$, that
	\begin{equation}\label{ADeq4}
	\frac{1}{|T|}\sum\limits_{\mathbf{j},\mathbf{s} \in T}e^{-i(\mathbf{j}'\mathbf{\lambda}_\mathbf{l}-\mathbf{s}'\mathbf{\lambda}_\mathbf{k})}\text{Cov}(V(\mathbf{j}),V(\mathbf{s}))=o(1)\,.
	\end{equation}
	
	Finally, we have that, uniformly in $\mathbf{k}$,
	\begin{equation}\label{Clemeq12}
	\begin{aligned}
	&\frac{1}{|T|}\sum\limits_{\mathbf{j},\mathbf{s} \in T}e^{-i(\mathbf{j}-\mathbf{s})'\mathbf{\lambda}_\mathbf{k}}\text{Cov}(V(\mathbf{j}),V(\mathbf{s}))-4\pi^2f(\mathbf{\lambda}_\mathbf{k})\\
	=\ &\frac{1}{|T|}\sum\limits_{\mathbf{h}:(|h_1|,|h_2|) \in T}(d_1-|h_1|)(d_2-|h_2|)e^{-i\mathbf{h}'\mathbf{\lambda}_\mathbf{k}}\gamma(\mathbf{h})-4\pi^2f(\mathbf{\lambda}_\mathbf{k})
	=\  o(1) \,.
	\end{aligned}
	\end{equation}
	
	Putting together (\ref{Clemeq9})-(\ref{Clemeq12}) yields (\ref{Clemeq8}). Note that a refined version of (\ref{Clemeq10})-(\ref{Clemeq12}) under the stronger assumption that $\sum_{\mathbf{j} \in \mathbb{Z}^2} |\mathbf{j}|^u |\gamma(\mathbf{j})|< \infty$ holds for some $u >0$, gives the uniform convergence rate
	\begin{equation*}
	\left\{ \begin{array}{lcr}
	O(|T|^{-u}), & 0 < u < 1 \,,\\
	O\left(\frac{\log|T|}{|T|}\right), & u=1 \,,\\
	O(|T|^{-1}), & u>1\,,
	\end{array}\right.
	\end{equation*} 
	which yields \eqref{Clemeq4} in assertion (b). Note that $\mathbb{E}(x(\mathbf{k}))=\mathbb{E}(y(\mathbf{k}))=0$ and $\sum\limits_{\mathbf{j} \in T}e^{-i\mathbf{j}'\mathbf{\lambda}_\mathbf{k}}=0$. We have
	$$\mathbb{E}(x(\mathbf{k})+iy(\mathbf{k}))=\frac{1}{\sqrt{|T|}}\mathbb{E}(V(\cdot))\sum\limits_{\mathbf{j} \in T}e^{-i\mathbf{j}'\mathbf{\lambda}_\mathbf{k}}=0.$$
	Thus, by (\ref{Clemeq8}), 
	a simple application of Markov-inequality yields 
	\begin{equation*}
	\frac{1}{2|N|}\sum\limits_{\mathbf{j} \in N}\frac{x(\mathbf{j})}{\sqrt{f(\mathbf{\lambda}_\mathbf{j})}}=o_p(1),\quad \frac{1}{2|N|}\sum\limits_{\mathbf{j} \in N}\frac{y(\mathbf{j})}{\sqrt{f(\mathbf{\lambda}_\mathbf{j})}}=o_p(1)\,,
	\end{equation*}
	hence assertion (a) follows. 
	Since 
	\begin{equation*}
	\begin{aligned}
	\mathbb{E}(I(\mathbf{j})) &= \mathbb{E}(x^2(\mathbf{j})+y^2(\mathbf{j})) = 
	\text{Cov}(x(\mathbf{j}),x(\mathbf{j}))+\text{Cov}(y(\mathbf{j}),y(\mathbf{j}))\,.
	\end{aligned}
	\end{equation*}
	By (\ref{Clemeq8}), we then have
	\begin{equation}\label{ADeq7}
	\sup\limits_{\mathbf{j} \in N}|\mathbb{E}(I(\mathbf{j}))-4\pi^2f(\mathbf{\lambda}_\mathbf{j})|=o(1)\,,
	\end{equation}
	and thus assertion (c) follows from an application of the Markov inequality, and Lemmas \ref{lem1_thm5_1} and \ref{lem1_thm5_2}.
	Since 
	\begin{equation*}
	\begin{aligned}
	 \mathbb{E}(I^2(\mathbf{j}))=\text{Var}(I(\mathbf{j}))+\mathbb{E}(I(\mathbf{j}))^2 =  (4\pi^2f(\mathbf{\lambda}_\mathbf{j}))^2+(4\pi^2f(\mathbf{\lambda}_\mathbf{j}))^2=2(4\pi^2f(\mathbf{\lambda}_\mathbf{j}))^2\,,
	\end{aligned}
	\end{equation*}
	it holds by Lemmas \ref{lem1_thm5_1} and \ref{lem1_thm5_2}, and (\ref{ADeq7}) that
	$\sup\limits_{\mathbf{j} \in N}|\mathbb{E}(I^2(\mathbf{j}))-2(4\pi^2f(\mathbf{\lambda}_\mathbf{j}))^2|=o(1)$.
	Hence by (\ref{Clemcond2}) and an application of the Markov inequality, assertion (d) follows.
	
	For assertion (e), if Assumption \ref{asslinear} with $r > 1/2$ holds, by an analogous proof of Theorem 1 in \cite{zhao1980distribution} and Corollary 1 in \cite{DR1996}, we have for any $s <5$,
	\begin{equation}\label{slln_period}
	\frac{1}{|N|} \sum_{\mathbf{j} \in N} \frac{I^s(\mathbf{j})}{f^s(\lambda_\mathbf{j})}<C_2+o_p(1)\,.
	\end{equation}
	Besides, if Assumption \ref{assgmc} with $\mathbb{E}\left(|V(\mathbf{0})|^{16}\right)< \infty$ holds, then an analogous proof of Lemma A.5 in \cite{SX2007} under the existence of $16$-th moment yields \eqref{slln_period} for any $s \leq 8$.
	
\end{proof}

\subsection{Proof of Theorem \ref{thm_periodogram}}
\begin{proof}[Proof of Theorem \ref{thm_periodogram}]
	By Assumptions \ref{asssummable} and \ref{assspecpostive}, we have the boundedness of the spectral density function. Together with the results of Theorem \ref{thm2}(c), we have 		
		$$\frac{1}{2|N|} \sum\limits_{\mathbf{j}\in N}(|x(\mathbf{j})|+|y(\mathbf{j})|) = O_p(1)\,.$$
	Let $\mathbf{m}_T=(m_{T1},m_{T2})=\left(\left\lfloor \frac{{d_1}^{1/2}}{{h_T}_1} \right\rfloor, \left\lfloor \frac{{d_2}^{1/2}}{{h_T}_2} \right\rfloor \right)$. Define $\mathbf{L}_M = \{(l_1,l_2)\in \mathbb{Z}^2: |l_1| \leq m_{T1} , |l_2| \leq m_{T2}\}$ and $\mathbf{Q}_M=\{(\mathbf{q},\mathbf{r})\in \mathbb{Z}^2 \times \mathbb{Z}^2 :|q_1-r_1| \leq \frac{d_1}{m_{T1}}+1, |q_2-r_2| \leq \frac{d_2}{m_{T2}}+1\}$.
	Then, the supremum in (a) can be decomposed as 
	\begin{equation*}
	\begin{aligned}
	 \sup\limits_{\mathbf{j} \in N} \left|\sum\limits_{\mathbf{s} \in\mathbb{Z}^2}p_{\mathbf{s},T} x(\mathbf{j}+\mathbf{s}) \right| 
	 \leq\ &\ \sup\limits_{\mathbf{l} \in \mathbf{L}_M} \left|\sum\limits_{\mathbf{s} \in\mathbb{Z}^2} p_{\mathbf{s},T} x(\mathbf{s}_\mathbf{l}+\mathbf{s}) \right|+\sup\limits_{(\mathbf{q},\mathbf{r}) \in \mathbf{Q}_M} \left|\sum\limits_{\mathbf{s} \in\mathbb{Z}^2} p_{\mathbf{s},T} (x(\mathbf{q}+\mathbf{s})-x(\mathbf{r}+\mathbf{s})) \right|\\
	=\ & O_p\left(|\mathbf{h}_T|^{-1}|T|^{-\frac{1}{4}}\right)\,,
	\end{aligned}
	\end{equation*}
	where $s_\mathbf{l}=\left(\left\lfloor \frac{l_1d_1}{m_{T1}}\right\rfloor, \left\lfloor \frac{l_2d_2}{m_{T2}}\right\rfloor \right)$ for $\mathbf{l}=(l_1,l_2) \in\mathbb{Z}^2$.	
	The last line follows by the following two arguments. The first summation follows from the Chebyshev's inequality, the assumptions on $K(\cdot)$,  $f(\cdot)$, and (\ref{lemeq4}) and (\ref{lemeq5}). Note that we have $\sum\limits_{\mathbf{s} \in\mathbb{Z}^2}p^2_{\mathbf{s},T}=\frac{O\left(|\mathbf{h}_T||T|\right)}{O\left(|\mathbf{h}_T|^2|T|^2\right)}$. Hence, we have
	\begin{eqnarray*}
	   \mathbb{P}\left(|\mathbf{h}_T||T|^{\frac{1}{4}}\sup\limits_{\mathbf{l} \in \mathbf{L}_M} \left|\sum\limits_{\mathbf{s} \in\mathbb{Z}^2}p_{\mathbf{s},T} x(\mathbf{s}_\mathbf{l}+\mathbf{s})\right| \geq \varepsilon \right)
		& \leq & \sum\limits_{\mathbf{l} \in \mathbf{L}_M} \frac{|\mathbf{h}_T|^2|T|^{\frac{2}{4}}}{\varepsilon^2} \text{Var} \left(\sum\limits_{\mathbf{s} \in\mathbb{Z}^2} p_{\mathbf{s},T} x(\mathbf{s}_\mathbf{l}+\mathbf{s})\right)\\
		& \asymp & 
		|m_{T1}||m_{T2}|\frac{|\mathbf{h}_T|^2|T|^{\frac{2}{4}}}{\varepsilon^2}O(1)\sum\limits_{\mathbf{s} \in\mathbb{Z}^2}p^2_{\mathbf{s},T}\\
		& = & \frac{|T|^{\frac{1}{2}}}{|\mathbf{h}_T|}\frac{|\mathbf{h}_T|^2|T|^{\frac{2}{4}}}{\varepsilon^2} O(1) \left(\frac{O\left(|\mathbf{h}_T||T|\right)}{O\left(|\mathbf{h}_T|^2|T|^2\right)}\right)
		 =  O(1)\,.
	\end{eqnarray*}
	
	For the second summand, using Assumptions \ref{assk1} and \ref{assk5}, we have
	\begin{eqnarray*}
		\sup\limits_{(\mathbf{q},\mathbf{r}) \in \mathbf{Q}_M} \left|\sum\limits_{\mathbf{s} \in\mathbb{Z}^2} p_{\mathbf{s},T}(x(\mathbf{q}+\mathbf{s})-x(\mathbf{r}+\mathbf{s})) \right|
		&=& \sup\limits_{(\mathbf{q},\mathbf{r}) \in \mathbf{Q}_M} \left|\sum\limits_{\mathbf{s} \in\mathbb{Z}^2}(p_{\mathbf{s}-\mathbf{q},T}-p_{\mathbf{s}-\mathbf{r},T})x(\mathbf{s}) \right|\\
		& \asymp &\frac{1}{|T||\mathbf{h}_T|^{3/2}|\mathbf{m}_T|^{1/2}}\sum\limits_{\mathbf{j} \in T}|x(\mathbf{j})|
		 = O_p\left(|\mathbf{h}_T|^{-1}|T|^{-\frac{1}{4}}\right)\,.
	\end{eqnarray*}
	Analogous arguments yield the assertion for $y(\cdot)$. Also, by using the results of Theorem \ref{thm2}, analogous arguments yield the results in $(b)$, $(c)$ and $(d)$.
\end{proof}

\subsection{Proof of Theorem \ref{thm3}(a)}
\begin{proof}[Proof of Theorem \ref{thm3} (a)]
	Under the Assumptions of Theorem \ref{thm3}(a) and by using the union bound, we have for any $\varepsilon > 0, \eta > 0$, and $|T|$ large enough, that
	\begin{eqnarray*}
		 \mathbb{P}\left(\max_{\mathbf{\lambda} \in [0,2\pi]^2}|\widehat{f}_T(\mathbf{\lambda})-f(\mathbf{\lambda})|>\eta\right)
		 & = & \mathbb{P}\left(\max_{\mathbf{\lambda} \in [0,2\pi]^2}|\widehat{f}_T(\mathbf{\lambda})-f(\lambda)|>\eta,~ \varepsilon^{-1} \leq |\mathbf{h}_T|^{-1} \leq \varepsilon|T|^v\right)\\
		 & & \quad\quad+\mathbb{P}\left(\max_{\mathbf{\lambda} \in [0,2\pi]^2}|\widehat{f}_T(\mathbf{\lambda})-f(\lambda)|>\eta,~ \varepsilon^{-1} > |\mathbf{h}_T|^{-1} \right)\\ 
		 & & \quad\quad\quad+\mathbb{P}\left(\max_{\mathbf{\lambda} \in [0,2\pi]^2}|\widehat{f}_T(\mathbf{\lambda})-f(\lambda)|>\eta,~|\mathbf{h}_T|^{-1} > \varepsilon|T|^v\right)\\ 
		& \leq & \mathbb{P}\left(\max_{\mathbf{\lambda} \in [0,2\pi]^2}|\widehat{f}_T(\mathbf{\lambda})-f(\lambda)|>\eta,~ \varepsilon^{-1} \leq |\mathbf{h}_T|^{-1} \leq \varepsilon|T|^v\right)\\
		& & \quad\quad\quad+\mathbb{P}(\varepsilon|\mathbf{h}_T|^{-1}<1)+\mathbb{P}(|\mathbf{h}_T|^{-1}>\varepsilon|T|^v)\,.
	\end{eqnarray*}
	The last two probabilities converge to $0$ under the conditions on $\mathbf{h}_T$. It remains to prove that
	\begin{equation}\label{proof_51}
	\max_{\varepsilon^{-1} \leq |\mathbf{h}_T|^{-1} \leq \varepsilon|T|^v} \max_{\mathbf{\lambda} \in [0,2\pi]^2}|\widehat{f}_T(\mathbf{\lambda})-f(\mathbf{\lambda})| \xrightarrow{p} 0, \mbox{ as } T \rightarrow \infty.
	\end{equation}
	Denote $\bar{f}_T(\mathbf{\lambda})=\frac{1}{4\pi^2}\sum\limits_{\mathbf{j} \in\mathbb{Z}^2}R_V(\mathbf{j})k(j_1 {h_T}_1,j_2 {h_T}_2)\exp(-i\mathbf{j}'\mathbf{\lambda})$, where $R_V$ is defined in \eqref{auto_sample}. Recall that	$\widehat{f}_T(\mathbf{\lambda})=\frac{1}{4\pi^2}\sum\limits_{\mathbf{j} \in\mathbb{Z}^2}\widehat{R}_V(\mathbf{j})k(j_1 {h_T}_1,j_2 {h_T}_2)\exp(-i\mathbf{j}'\mathbf{\lambda})+o(1)$ where $\widehat{R}_V$ is defined in \eqref{auto_sample} with $\mu$ replaced by $\bar{V}_T$. 
	Then, we have
	\begin{equation*}
	\begin{aligned}
	\widehat{f}_T(\mathbf{\lambda})-f(\mathbf{\lambda})=a_1+a_2+a_3\,,
	\end{aligned}
	\end{equation*} 
	where $a_1=\widehat{f}_T(\mathbf{\lambda})-\bar{f}_T(\mathbf{\lambda})$,
	$a_2=\bar{f}_T(\mathbf{\lambda})-\mathbb{E}(\bar{f}_T(\mathbf{\lambda}))$, and 
	$a_3=\mathbb{E}(\bar{f}_T(\mathbf{\lambda}))-f(\mathbf{\lambda})$.
	For the third term $a_3$, by a similar argument as in the proof of Theorem \ref{thm3}(b), see \eqref{dev_proof_51b} below, it holds uniformly for $\varepsilon^{-1} \leq |\mathbf{h}_T|^{-1} \leq \varepsilon|T|^v$ that
	\begin{equation}\label{proof_51_1}
	\begin{aligned}
	\max_{\substack{\mathbf{\lambda} \in [0,2\pi]^2 \\ \varepsilon^{-1} \leq |\mathbf{h}_T|^{-1} \leq \varepsilon|T|^v}}|a_3| \leq \max_{\mathbf{\lambda} \in [0,2\pi]^2}|\mathbb{E}(\bar{f}_T(\mathbf{\lambda}))-f(\mathbf{\lambda})|=o(1)\,.
	\end{aligned}
	\end{equation}
	For the second term $a_2$, by Assumptions \ref{assautocov} and \ref{assk3}, we have
	\begin{equation}\label{proof_51_2}
	\begin{aligned}
	\max_{\substack{\mathbf{\lambda} \in [0,2\pi]^2 \\ \varepsilon^{-1} \leq |\mathbf{h}_T|^{-1} \leq \varepsilon|T|^v}}|a_2|\ \leq\ &\max_{\varepsilon^{-1} \leq |\mathbf{h}_T|^{-1} \leq \varepsilon|T|^v}\frac{1}{4\pi^2}\sum\limits_{\mathbf{j} \in\mathbb{Z}^2}|k(j_1 {h_T}_1,j_2 {h_T}_2)||R_V(\mathbf{r})-\mathbb{E}(R_V(\mathbf{r}))|\\
	=\ &O_p\left(|T|^{-v}\sum\limits_{\mathbf{j} \in\mathbb{Z}^2}\widetilde{k}\left(\frac{j_1}{\sqrt{\varepsilon}d^v_1},\frac{j_2}{\sqrt{\varepsilon}d^v_2}\right)\right) 
	= O_p\left(\varepsilon \int_{\mathbb{R}^2_+}\widetilde{k}(\mathbf{x})d\mathbf{x}\right)\\
	=\ &O_p(\varepsilon)\,.
	\end{aligned}
	\end{equation}   
	Since $\varepsilon$ is arbitrary, we have $\max_{\substack{\mathbf{\lambda} \in [0,2\pi]^2 \\ \varepsilon^{-1} \leq |\mathbf{h}_T|^{-1} \leq \varepsilon|T|^v}}|a_2| \xrightarrow{p} 0$. 
	Finally, for the first term $a_1$, by Assumption \ref{assk3}, we have
	\begin{equation}\label{proof_51_3}
	\begin{aligned}
	 \max_{\substack{\mathbf{\lambda} \in [0,2\pi]^2 \\ \varepsilon^{-1} \leq |\mathbf{h}_T|^{-1} \leq \varepsilon|T|^v}}|a_1|  \leq\ &\max_{\varepsilon^{-1} \leq |\mathbf{h}_T|^{-1} \leq \varepsilon|T|^v} \Bigg\{ \frac{1}{4\pi^2|T|}\sum\limits_{\mathbf{j} \in\mathbb{Z}^2}\left|\widetilde{k}(j_1 {h_T}_1,j_2 {h_T}_2) \right| 
	 \times \Bigg[\left|\sum\limits_{\mathbf{l},\mathbf{l}+\mathbf{j} \in T}(V(\mathbf{l})-\mu)(\bar{V}_T-\mu)\right|\\
	& \quad\quad\quad\quad\quad\quad+\left|\sum\limits_{\mathbf{l},\mathbf{l}+\mathbf{j} \in T}(V(\mathbf{l}+\mathbf{j})-\mu)\right|\left|\bar{V}_T-\mu\right|+|\bar{V}_T-\mu|^2\Bigg]\Bigg\}\\
	=\ &O_p\left(\frac{1}{|T|}\sum\limits_{\mathbf{j} \in\mathbb{Z}^2}\widetilde{k}\left(\frac{j_1}{\sqrt{\varepsilon}d^v_1},\frac{j_2}{\sqrt{\varepsilon}d^v_2}\right)\right) \xrightarrow{p} 0\,,
	\end{aligned}
	\end{equation}
	where boundedness of the spectral density $f$ is used. As a result, by (\ref{proof_51_1}), (\ref{proof_51_2}) and (\ref{proof_51_3}), we have \eqref{proof_51} and the proof is completed.  
\end{proof}

\subsection{Proof of Theorem \ref{thm3}(b) with lemmas}
Before we prove Theorem \ref{thm3}(b), we need the following subsection about summability of joint cumulants implied by the geometric-moment contraction condition GMC($\alpha$) in Assumption \ref{assgmc}.
\subsubsection{Summability of Cumulants under GMC($\alpha$) for spatial process}
Assume that $V(\mathbf{j})=G(\varepsilon_{\mathbf{j}-\mathbf{s}}: \mathbf{s} \in\mathbb{Z}^2),~\mathbf{j} \in\mathbb{Z}^2,$ where $G(\cdot)$ is a measurable function and $\{\varepsilon_\mathbf{i}\}_{\mathbf{i} \in\mathbb{Z}^2}$ is an i.i.d. random field. Assume the geometric-moment contraction condition GMC($\alpha$) in Assumption \ref{assgmc} holds.
Let $(U_1,...,U_k)$ be a random vector. Then the joint cumulant is defined as 
\begin{equation*}
\text{cum}(U_1,...U_k)=\sum(-1)^p(p-1)!\mathbb{E}\left(\prod\limits_{j \in {V_1}}U_j\right)\ldots\mathbb{E}\left(\prod\limits_{j \in {V_p}}U_j\right)\,,
\end{equation*}    
where $V_1,...,V_p$ is a partition of the set $\{1,2,...,k\}$ and the sum is taken over all such partition.


\begin{lem}\label{AlemC2}
	Assume that there exists $C_1 >0$, $\rho \in (0,1)$ and $k \in\mathbb{N},~k \geq 2$ such that $\mathbb{E}\{|V(\mathbf{j})|^k\} < \infty$ and $\mathbb{E}\{|V(\mathbf{j})-\widetilde{V}(\mathbf{j})|^k\} \leq C_1\rho^{\|\mathbf{j}\|}$ for $\mathbf{j} \in\mathbb{N}^2$. Then, whenever $\mathbf{0} \leq \|\mathbf{m}_1\| \leq \|\mathbf{m}_2\| \leq...\leq\|\mathbf{m}_{k-1}\|$ where $\mathbf{m}_1,..., \mathbf{m}_{k-1} \in\mathbb{N}^2$, $\mathbf{m}_k=(m_{k1},m_{k2})$ and $\|\mathbf{m}_k\|=\max\{|m_{k1}|,|m_{k2}|\}$,        
	\begin{equation*}
	\Big|\mathrm{cum}(V(\mathbf{0}),V(\mathbf{m}_1),V(\mathbf{m}_2),...,V(\mathbf{m}_{k-1}))\Big| \leq {C\rho}^{\frac{\|\mathbf{m}_{k-1}\|}{2k(k-1)}}\,,
	\end{equation*}
	where the constant $C>0$ is independent of $\mathbf{m}_1,...,\mathbf{m}_{k-1}$.
\end{lem}

\begin{proof}[Proof of Lemma \ref{AlemC2}]
	Let $C>0$ be a generic constant which is independent of $\mathbf{m}_1,\ldots, \mathbf{m}_{k-1}$. In the proof, $C$ may vary from line to line and it only depends on $C_1,\rho$ and the moments $\mathbb{E}(|V(\mathbf{j})|^i), 1 \leq i \leq k$. Let $J=\text{cum}(V(\mathbf{0}),V(\mathbf{m}_1),...,V(\mathbf{m}_{k-1}))$, where $\mathbf{m}_0=\mathbf{0}$, $l \in \mathbb{N}$, $1 \leq l \leq k-1$.  Define coupled version $\widetilde{V}_l(\mathbf{m}_i)$ as 
	\begin{equation*}
	\widetilde{V}_l(\mathbf{m}_i)=G\left(\varepsilon^*_{\mathbf{m}_i-\mathbf{s}}, \mathbf{s} \in\mathbb{Z}^2\right) \quad \forall \mathbf{m}_i \in\mathbb{Z}^2,~l \in \mathbb{N}\,,
	\end{equation*}
	where $$\varepsilon^*_{\mathbf{m}_i-\mathbf{s}} = \left\{ \begin{array}{lcr}
	\varepsilon_{\mathbf{m}_i-\mathbf{s}}, & \|\mathbf{s}\|<l \,, \\
	\widetilde{\varepsilon}_{\mathbf{m}_i-\mathbf{s}}, & \|\mathbf{s}\| \geq l \,.
	\end{array}\right.$$
	Also, denote
	\begin{equation*}
	\widetilde{\widetilde{V}}_l(\mathbf{m}_i)=G\left(\varepsilon^{**}_{\mathbf{m}_i-\mathbf{s}}, \mathbf{s} \in\mathbb{Z}^2\right)\quad \forall \mathbf{m}_i \in\mathbb{Z}^2\,,
	\end{equation*}
	where $$\varepsilon^{**}_{\mathbf{m}_i-\mathbf{s}} = \left\{ \begin{array}{lcr}
	\varepsilon_{\mathbf{m}_i-\mathbf{s}}, & \|\mathbf{s}\|<l\,, \\
	\widetilde{\widetilde{\varepsilon}}_{\mathbf{m}_i-\mathbf{s}}, & \|\mathbf{s}\| \geq l\,.
	\end{array}\right.$$
	
	
	Note that $\{\varepsilon_\mathbf{i}\}_{\mathbf{i} \in\mathbb{Z}^2}, \{\widetilde{\varepsilon}_\mathbf{i}\}_{\mathbf{i} \in\mathbb{Z}^2}$ and $\{\widetilde{\widetilde{\varepsilon}}_\mathbf{i}\}_{\mathbf{i} \in\mathbb{Z}^2}$ are i.i.d. random fields.	By the GMC property, we have
	\begin{equation*}
	\begin{aligned}
	\mathbb{E}\left(|V(\mathbf{m}_i)-\widetilde{V}_l(\mathbf{m}_i)|^{\alpha}\right) \leq C\rho^l \quad \quad 
	\mbox{and} \quad\quad \mathbb{E}\left(|V(\mathbf{m}_i)-\widetilde{\widetilde{V}}_l(\mathbf{m}_i)|^{\alpha}\right) \leq C\rho^l\,.
	\end{aligned}
	\end{equation*}
	Define $n_l=\|\mathbf{m}_l\|-\|\mathbf{m}_{l-1}\|, 1 \leq l \leq k-1$ and $\mathbf{m}_0=\mathbf{0}$. We have
	\begin{equation*}
	\begin{aligned}
	J =\ &\text{cum}\big(V(\mathbf{0}),V(\mathbf{m}_1),\ldots,V(\mathbf{m}_{l-1}),\ldots,V(\mathbf{m}_{k-1})\big)\\
	=\ &\text{cum}\big(V(\mathbf{m}_0-\mathbf{m}_{l-1}),V(\mathbf{m}_1-\mathbf{m}_{l-1}),\ldots,V(\mathbf{0}),\ldots,V(\mathbf{m}_{k-1}-\mathbf{m}_{l-1})\big)\\
	=\ &\text{cum}\big(V(\mathbf{m}_0-\mathbf{m}_{l-1}),\ldots,V(\mathbf{0}),V(\mathbf{m}_l-\mathbf{m}_{l-1})-\widetilde{V}_{\frac{n_l}{2}}(\mathbf{m}_l-\mathbf{m}_{l-1}),V(\mathbf{m}_{l+1}-\mathbf{m}_{l-1}),\ldots,V(\mathbf{m}_{k-1}-\mathbf{m}_{l-1})\big)\\
	&\quad\quad +\sum\limits^{k-l-1}_{j=1}\text{cum}\big(V(\mathbf{m}_0-\mathbf{m}_{l-1}),\ldots,V(\mathbf{0}),\widetilde{V}_{\frac{n_l}{2}}(\mathbf{m}_l-\mathbf{m}_{l-1}),\ldots,\widetilde{V}_{\frac{n_l}{2}}(\mathbf{m}_{l+j-1}-\mathbf{m}_{l-1}),\\
	&\quad\quad V(\mathbf{m}_{l+j}-\mathbf{m}_{l-1})-\widetilde{V}_{\frac{n_l}{2}}(\mathbf{m}_{l+j}-\mathbf{m}_{l-1}),V(\mathbf{m}_{l+j+1}-\mathbf{m}_{l-1}),\ldots,V(\mathbf{m}_{k-1}-\mathbf{m}_{l-1})\big)\\
	&\quad\quad +\text{cum}\big(V(\mathbf{m}_0-\mathbf{m}_{l-1}),\ldots,V(\mathbf{0}),\widetilde{V}_{\frac{n_l}{2}}(\mathbf{m}_l-\mathbf{m}_{l-1}),\ldots,\widetilde{V}_{\frac{n_l}{2}}(\mathbf{m}_{k-1}-\mathbf{m}_{l-1})\big)\\
	=\ &\widetilde{A}_0+\sum\limits^{k-l-1}_{j=1}\widetilde{A}_j+\text{cum}\big(V(\mathbf{m}_0-\mathbf{m}_{l-1}),\ldots,V(\mathbf{0}),\widetilde{V}_{\frac{n_l}{2}}(\mathbf{m}_l-\mathbf{m}_{l-1}),\ldots,\widetilde{V}_{\frac{n_l}{2}}(\mathbf{m}_{k-1}-\mathbf{m}_{l-1})\big)\\
	=\ &\widetilde{A}_0+\sum\limits^{k-l-1}_{j=1}\widetilde{A}_j+\widetilde{\widetilde{B}}_0+\sum\limits^{l-1}_{i=1}\widetilde{\widetilde{B}}_i+C_0\,,
	\end{aligned}
	\end{equation*}
	where
	\begin{equation*}
	\begin{aligned}
	&\widetilde{A}_0=\text{cum}\big(V(\mathbf{m}_0-\mathbf{m}_{l-1}),\ldots,V(\mathbf{0}),V(\mathbf{m}_l-\mathbf{m}_{l-1})-\widetilde{V}_{\frac{n_l}{2}}(\mathbf{m}_l-\mathbf{m}_{l-1}),\\
	&\quad\quad\quad\quad\quad\quad\quad\quad\quad\quad\quad\quad\quad\quad\quad\quad\quad\quad\quad V(\mathbf{m}_{l+1}-\mathbf{m}_{l-1}),\ldots,V(\mathbf{m}_{k-1}-\mathbf{m}_{l-1})\big)\,,\\
	&\widetilde{A}_j=\text{cum}\big(V(\mathbf{m}_0-\mathbf{m}_{l-1}),\ldots,V(\mathbf{0}),\widetilde{V}_{\frac{n_l}{2}}(\mathbf{m}_l-\mathbf{m}_{l-1}),\ldots,\widetilde{V}_{\frac{n_l}{2}}(\mathbf{m}_{l+j-1}-\mathbf{m}_{l-1}),\\
	&\quad\quad V(\mathbf{m}_{l+j}-\mathbf{m}_{l-1})-\widetilde{V}_{\frac{n_l}{2}}(\mathbf{m}_{l+j}-\mathbf{m}_{l-1}),V(\mathbf{m}_{l+j+1}-\mathbf{m}_{l-1}),\ldots,V(\mathbf{m}_{k-1}-\mathbf{m}_{l-1})\big)\,,\\	
	&\widetilde{\widetilde{B}}_0=\text{cum}\big(V(\mathbf{m}_0-\mathbf{m}_{l-1}),\ldots,V(\mathbf{0})-\widetilde{\widetilde{V}}_{\frac{n_l}{2}}(\mathbf{0}),\widetilde{V}_{\frac{n_l}{2}}(\mathbf{m}_l-\mathbf{m}_{l-1}),\ldots,\widetilde{V}_{\frac{n_l}{2}}(\mathbf{m}_{k-1}-\mathbf{m}_{l-1})\big)\,,\\
	&\widetilde{\widetilde{B}}_i=\text{cum}\big(V(\mathbf{m}_0-\mathbf{m}_{l-1}),\ldots,V(\mathbf{m}_{l-i-2}-\mathbf{m}_{l-1}), V(\mathbf{m}_{l-i-1}-\mathbf{m}_{l-1})-\widetilde{\widetilde{V}}_{\frac{n_l}{2}}(\mathbf{m}_{l-i-1}-\mathbf{m}_{l-1}),\ldots,\\
	&\quad\quad\quad\quad\,\, \widetilde{\widetilde{V}}_{\frac{n_l}{2}}(\mathbf{0}),\widetilde{V}_{\frac{n_l}{2}}(\mathbf{m}_l-\mathbf{m}_{l-1}),\ldots,\widetilde{V}_{\frac{n_l}{2}}(\mathbf{m}_{k-1}-\mathbf{m}_{l-1})\big)\,,\\
	&C_0=\text{cum}\big(\widetilde{\widetilde{V}}_{\frac{n_l}{2}}(\mathbf{m}_0-\mathbf{m}_{l-1}),\ldots,\widetilde{\widetilde{V}}_{\frac{n_l}{2}}(\mathbf{0}),\widetilde{V}_{\frac{n_l}{2}}(\mathbf{m}_l-\mathbf{m}_{l-1}),\ldots,\widetilde{V}_{\frac{n_l}{2}}(\mathbf{m}_{k-1}-\mathbf{m}_{l-1})\big)\,.
	\end{aligned}
	\end{equation*}       
	Since $\big(\widetilde{\widetilde{V}}_{\frac{n_l}{2}}(\mathbf{m}_0-\mathbf{m}_{l-1}),\ldots,\widetilde{\widetilde{V}}_{\frac{n_l}{2}}(\mathbf{0})\big)$ and $\big(\widetilde{\widetilde{V}}_{\frac{n_l}{2}}(\mathbf{m}_l-\mathbf{m}_{l-1}),\ldots,\widetilde{\widetilde{V}}_{\frac{n_l}{2}}(\mathbf{m}_{k-1}-\mathbf{m}_{l-1})\big)$ are independent, we have $C_0=0$. We shall now show that $|\widetilde{A}_0| \leq {C\rho}^\frac{n_l}{2k}$. To this end, let $U_j=V(\mathbf{m}_j-\mathbf{m}_{l-1})$ for $j=0,1,\ldots,k-1$, $j \neq l$ and $U_l=V(\mathbf{m}_l-\mathbf{m}_{l-1})-\widetilde{V}_{\frac{n_l}{2}}(\mathbf{m}_l-\mathbf{m}_{l-1}).$ Let $|F|$ be the cardinality of the set $F$. For any subset $F \subset \{0,1,\ldots,k-1\}$ such that $l \notin F,$ by H{\"o}lder's and Jensen's inequity, we have
	\begin{equation*}
	\quad \left|\mathbb{E}\left(\prod\limits_{j \in F}U_j\right)\right
	| \leq \mathbb{E}\left(|V(\mathbf{0})|^{|F|}\right)\,,
	\end{equation*}
	and
	\begin{equation*}
	\begin{aligned}
	 \left|\mathbb{E}\left(U_l\prod\limits_{j \in F}U_j\right)\right| &\leq \|U_l\|_{1+|F|}\mathbb{E}\left(\prod\limits_{j \in F}|U_j|^{\frac{|F|+1}{|F|}}\right)^{\frac{|F|}{1+|F|}} 
	 \leq \|U_l\|_k\left(\mathbb{E}\left(|V(\mathbf{0})|^{|F|+1}\right)\right)^{\frac{|F|}{1+|F|}}
	\leq \left({C_1\rho}^{\frac{n_l}{2k}}\right)C^{*}\,,
	\end{aligned}
	\end{equation*}
	where $C^{*}=\sum\limits^{k-1}_{i=0}\mathbb{E}\left(|V(\mathbf{0})|^{i+1}\right)^{\frac{i}{1+i}}$. By definition of joint cumulant, $|\widetilde{A}_0| \leq {C\rho}^{\frac{n_l}{2k}}$ for some constant $C$. 
Similarly, for $j=1,\ldots,k-l-1$, we have $|\widetilde{A}_j| \leq {C\rho}^{\frac{n_l}{2k}}$, 
	and $|\widetilde{\widetilde{B}}_i| \leq {C\rho}^{\frac{n_l}{2k}}$ for $i=0,\ldots,l-1$,
	which implies that $|J| \leq {C\rho}^{\frac{n_l}{2k}}$ for $1 \leq l \leq k-1$. Hence, we have
	$ |J| \leq C \min\limits_{1 \leq l \leq {k-1}}{\rho}^{\frac{n_l}{2k}}$.	
     Since 
	\begin{equation*}
	\begin{aligned}
	&\|\mathbf{m}_{k-1}\|=\sum\limits^{k-1}_{l=1} (\|\mathbf{m}_l\|-\|\mathbf{m}_{l-1}\|)=\sum\limits^{k-1}_{l=1}n_l \leq (k-1)\max\limits_{1 \leq l \leq {k-1}}n_l \,,
	\end{aligned}
	\end{equation*}
	we have $\frac{\|\mathbf{m}_{k-1}\|}{k-1} \leq \max\limits_{1 \leq l \leq {k-1}}n_l$. Finally, we have
	$|J| \leq {C\rho}^{\max\limits_{1 \leq l \leq {k-1}}\frac{n_l}{2k}} \leq {C\rho}^{\frac{\|\mathbf{m}_{k-1}\|}{2k(k-1)}} \,$.

\end{proof}

\begin{lem}\label{AlemC3}
	Let the sequence of sets indexed by $T$ be $\mathbb{S}_T \subset \mathbb{N}^2$ satisfying $\mathbb{S}_T=\{(t_1,t_2),1 \leq t_k \leq S_{Tk}, k=1,2\} \subset T$, where $T=\{(t_1,t_2),1 \leq t_k \leq d_{k}, k=1,2\}$, $S_{T1} \leq d_1$ and $S_{T2} \leq d_2$, and define another sequence of sets indexed by $T$ be $\mathbb{B}_T \subset \mathbb{N}^2$, $\mathbb{B}_T=\{(l_1,l_2),1 \leq l_k \leq B_{Tk}, k=1,2\}$ with $(B_{T1},B_{T2})=\left(\frac{1}{{h_T}_1},\frac{1}{{h_T}_2}\right)$ and $|\mathbb{B}_T|=o(|\mathbb{S}_T|)$, and
	\begin{equation*}
	U_\mathbf{j}=U_\mathbf{j}(\lambda)=(4\pi^2)^{-1}\sum\limits_{\mathbf{l} \in \mathbf{L}_B}V(\mathbf{j})V(\mathbf{j}+\mathbf{l})k(l_1 {h_T}_1,l_2 {h_T}_2)\cos(\mathbf{l}'\mathbf{\lambda})\,,
	\end{equation*}
	where $\mathbf{L}_B = \{(l_1,l_2)\in \mathbb{Z}^2: (|l_1|,|l_2|) \in {\mathbb{B}_T}\}$. Then, under GMC(4), we have $\left\|\sum\limits_{\mathbf{j} \in \mathbb{S}_T}\big(U_\mathbf{j}-\mathbb{E}(U_\mathbf{j})\big)\right\|^2_2 \sim |\mathbb{S}_T||\mathbb{B}_T|\sigma^2$. $\\$ 
\end{lem}

\begin{proof}[Proof of Lemma \ref{AlemC3}]
	Let $L(s)=\{(\mathbf{m}_1,\mathbf{m}_2,\mathbf{m}_3) \in {\mathbb{Z}^2} \times {\mathbb{Z}^2} \times {\mathbb{Z}^2}:\max\limits_{1 \leq i \leq 3}\|\mathbf{m}_i\|=s\}$
	and $C(\mathbf{m}_1,\mathbf{m}_2,\mathbf{m}_3)=\text{cum}(X_{\mathbf{0}},X_{\mathbf{m}_1},X_{\mathbf{m}_2},X_{\mathbf{m}_3})$.
	Thus, $|L(s)| \leq 24(2s+1)^3$, By Lemma \ref{AlemC2}, we have
	\begin{equation*}
	\begin{aligned}
	\sum\limits_{\mathbf{m}_1,\mathbf{m}_2,\mathbf{m}_3 \in\mathbb{Z}^2}|C(\mathbf{m}_1,\mathbf{m}_2,\mathbf{m}_3)| &\leq C^{**}\sum\limits^{\infty}_{s=0}\sum\limits_{(\mathbf{m}_1,\mathbf{m}_2,\mathbf{m}_3) \in L(s)}|C(\mathbf{m}_1,\mathbf{m}_2,\mathbf{m}_3)|
	\leq C^{***}\sum\limits^{\infty}_{s=0} s^3\rho^{\frac{s}{2(4)(4-1)}}
	< \infty \,.
	\end{aligned}
	\end{equation*}
	for some constants $C^{**}>0$ and $C^{***}>0$. Then, the lemma follows from similar arguments in (3.9)-(3.12) of \cite{RM1984}, page 1174.
\end{proof}

\subsubsection{Uniform consistency of the smoothed periodogram spectral density estimators}

Define the smoothed periodogram spectral density estimator as
\begin{equation*}
\widetilde{f}_T(\mathbf{\lambda})=\frac{1}{4\pi^2}\sum\limits_{\mathbf{j} \in\mathbb{Z}^2}R_V(\mathbf{j})k(j_1 {h_T}_1,j_2 {h_T}_2)\exp(-i\mathbf{j}'\mathbf{\lambda})\,,
\end{equation*}
where $R_V(\mathbf{j})=\frac{1}{|T|}\sum\limits_{\mathbf{l},\mathbf{l}+\mathbf{j} \in T}\big[V(\mathbf{l})-\mu\big]\big[V(\mathbf{l}+\mathbf{j})-\mu\big]$, 
and $k(\cdot)$ is defined as in \eqref{smallk} satisfying Assumption \ref{assk4}.
The following theorem prove the uniform consistency of the smoothed periodogram spectral density estimator $\widetilde{f}_T(\mathbf{\lambda})$ using the results of Lemma \ref{AlemC3}. Without loss of generality, we assume that $\mu=\mathbb{E}(V(\mathbf{0}))=0$ in the following.
\begin{thm}\label{AthmC4}
	Define $\mathbb{B}_T=\{(l_1,l_2)\in \mathbb{Z}^2:1 \leq l_k \leq B_{Tk}, k=1,2\}$, where $(B_{T1},B_{T2})=\left(\frac{1}{{h_T}_1},\frac{1}{{h_T}_2}\right) \rightarrow \infty$.
	Assume $GMC(\alpha)$, $\alpha > 0$, $V(\mathbf{j}) \in \mathbb{L}^{4+\delta}$ for some $\delta \in (0,4]$, and $|\mathbb{B}_T|=O(|T|^{\eta})$, i.e., $(|\mathbf{h}_T||T|^{\eta})^{-1}=O(1)$	
%
%
	for some $0 < \eta < \frac{\delta}{4+\delta}$ and $f_*=\min\limits_{\mathbf{\lambda} \in\mathbb{R}^2}f(\mathbf{\lambda}) > 0$. Then, 
	\begin{equation*}
	\max\limits_{\mathbf{\lambda} \in [0,2\pi]^2}\sqrt{|T||\mathbf{h}_T|}\left|\widetilde{f}_T(\mathbf{\lambda})-\mathbb{E}(\widetilde{f}_T(\mathbf{\lambda}))\right|=O_p\left((\log|T|)^{\frac{1}{2}}\right)\,.
	\end{equation*}		
\end{thm}

\begin{rem}\label{AremC2}
	If $GMC(\alpha)$, $\alpha > 0$ and $V(\mathbf{j}) \in \mathbb{L}^{4+\delta}$, then $V(\mathbf{j})$ satisfies GMC(4), see Lemma 2 of \cite{wu2005linear}. 
\end{rem}

\begin{proof}[Proof of Theorem \ref{AthmC4}]
	Let
	\begin{equation*}
	\begin{aligned}
	U_\mathbf{j}=U_\mathbf{j}(\mathbf{\lambda})&=\frac{1}{4\pi^2}\sum\limits_{\mathbf{l}\in \mathbf{L}_B}V(\mathbf{j})V(\mathbf{j}+\mathbf{l})k(l_1 {h_T}_1,l_2 {h_T}_2)\cos(\mathbf{l}'\mathbf{\lambda})\\
	&=\frac{1}{4\pi^2}\sum\limits_{\mathbf{l} \in \mathbf{L}_B}V(\mathbf{j})V(\mathbf{j}+\mathbf{l})\alpha_\mathbf{l}\,,
	\end{aligned}
	\end{equation*} 
	where $\mathbf{L}_B = \{(l_1,l_2)\in \mathbb{Z}^2: (|l_1|,|l_2|) \in {\mathbb{B}_T}\}$ and  $\alpha_\mathbf{l}=k(l_1 {h_T}_1,l_2 {h_T}_2)\cos(\mathbf{l}'\mathbf{\lambda})$. Note that 
	\begin{equation*}
	\begin{aligned}
	\widetilde{f}_T(\mathbf{\lambda}) &=\frac{1}{4\pi^2}\sum\limits_{\mathbf{l}\in \mathbf{L}_B}\left[\frac{1}{|T|}\sum\limits_{\mathbf{j},\mathbf{j}+\mathbf{l} \in T}(V(\mathbf{j})V(\mathbf{j}+\mathbf{l})\alpha_\mathbf{l})\right]
	=\frac{1}{|T|}\sum\limits_{\mathbf{j} \in T} U_\mathbf{j}-\frac{1}{|T|}\left(\sqrt{|T||\mathbb{B}_T|}q_T(\lambda)\right)\,,
	\end{aligned}
	\end{equation*}
	where  
	\begin{equation*}
	q_T(\mathbf{\lambda})=\frac{1}{4\pi^2\sqrt{|T||\mathbb{B}_T|}}\left(\sum\limits_{\mathbf{l}\in \mathbf{L}_B} \sum\limits_{\substack{\mathbf{j}: \mathbf{j} \in T, \mathbf{j}+\mathbf{l} \notin T}}V(\mathbf{j})V(\mathbf{j}+\mathbf{l})\alpha_\mathbf{l}\right) \,.
	\end{equation*}  
	By the summability of cumulants of order 2 and 4, we have
	\begin{equation*}
	\begin{aligned}
	\|q_T(\mathbf{\lambda})\|&=(|T||\mathbb{B}_T|)^{-\frac{1}{2}}O(|\mathbb{B}_T|) =O\left(\sqrt{\frac{|\mathbb{B}_T|}{|T|}}\right) =O\left(\sqrt{\frac{|T|^{\eta}}{|T|}}\right)=o(1).
	\end{aligned}
	\end{equation*}
	Denote $g_T=g_T(\mathbf{\lambda})=\sum\limits_{\mathbf{j} \in T}U_\mathbf{j}$. We have
	\begin{equation}\label{AeqC1}
	\sqrt{|T||\mathbf{h}_T|}\big(\widetilde{f}_T(\mathbf{\lambda})-\mathbb{E}(\widetilde{f}_T(\mathbf{\lambda}))\big)=\frac{g_T-\mathbb{E}(g_T)}{\sqrt{|T||\mathbb{B}_T|}}-q_T(\lambda)+\mathbb{E}(q_T(\lambda)) \,.
	\end{equation}
	Let $\rho=\rho(4)<1$ as in GMC(4). For $\mathbf{l} \in \mathbb{Z}^2$, define $\widehat{V}(\mathbf{l})=\mathbb{E}(V(\mathbf{l})|\mathcal{F}_{\mathbf{l},<m>})$, where $\mathcal{F}_{\mathbf{l},<m>}$ is defined the same as in the proof of Theorem \ref{thm21} with  
	$m=m_T=\left\lfloor-\frac{8\log|T|}{\log \rho}\right\rfloor \in \mathbb{N}$.	Let $\widehat{U}_\mathbf{j}=\widehat{U}_\mathbf{j}(\mathbf{\lambda})$ be the corresponding sum with $V(\mathbf{l})$ replaced by $\widehat{V}(\mathbf{l})$. Observe that $\widehat{V}(\mathbf{i})$ and $\widehat{V}(\mathbf{j})$ are i.i.d. if $\|\mathbf{i}-\mathbf{j}\| \geq 2m$ and $\widehat{U}_\mathbf{i}$ and $\widehat{U}_\mathbf{j}$ are i.i.d. if $\|\mathbf{i}-\mathbf{j}\| \geq 2 \max \{B_{T1},B_{T2}\}+2m$. Define $\widehat{g}_T=\widehat{g}_T(\mathbf{\lambda})=\sum\limits_{\mathbf{j} \in T}\widehat{U}_\mathbf{j}(\mathbf{\lambda})$. Then,  $\|g_T-\widehat{g}_T\|_2=o(1)$ since
	\begin{equation}\label{AeqC2}
	\begin{aligned}
	\|U_\mathbf{j}-\widehat{U}_\mathbf{j}\|_2 &\leq \frac{1}{4\pi^2}\sum\limits_{\mathbf{l}\in \mathbf{L}_B}\|V(\mathbf{j})V(\mathbf{j}+\mathbf{l})-\widehat{V}(\mathbf{j})\widehat{V}(\mathbf{j}+\mathbf{l})\|_2|\alpha_\mathbf{l}|\\
	&=O(|\mathbb{B}_T|)O\left(\|V(\mathbf{j})V(\mathbf{j}+\mathbf{l})-\widehat{V}(\mathbf{j})\widehat{V}(\mathbf{j}+\mathbf{l})\|_2\right) 
	= O\left(|\mathbb{B}_T|\rho^{\frac{m}{4}}\right) \,.
	\end{aligned}
	\end{equation}
	Therefore, we have
	\begin{equation*}
	\begin{aligned}
	\|g_T-\widehat{g}_T\|_2 & \leq \sum\limits_{\mathbf{j} \in T}\|U_\mathbf{j}-\widehat{U}_\mathbf{j}\|_2=O\left(|T||\mathbb{B}_T|\rho^{\frac{m}{4}}\right)
	= O\left(|T||\mathbb{B}_T|\frac{1}{|T|^2}\right) = o(1)\,.
	\end{aligned}
	\end{equation*}
	
	We now define four sets of blocks such that the blocks within the same set are i.i.d.. Let $p_T=\left\lfloor |T|^{1-\frac{4\eta}{\delta}}(\log|T|)^{-\frac{8}{\delta}-4}\right\rfloor$. We have $p_T \rightarrow \infty,~|\mathbb{B}_T|=o(p_T)$, and 
	\begin{equation*}
	k_T=\left\{(n_1,n_2):1 \leq n_1 \leq \left\lfloor \frac{d_1}{p_T} \right\rfloor, 1 \leq n_2 \leq \left\lfloor \frac{d_2}{p_T} \right\rfloor\right\}\,.
	\end{equation*} 
	Define 
	\begin{equation*}
	\begin{aligned}
	k_{T_{\mathcal{A}}} &=\{(n_1,n_2) \in k_T: n_1=2z_1+1, n_2=2z_2+1\quad \exists \, z_1,z_2 \in\mathbb{N}\} \,,\\
	k_{T_{\mathcal{B}}} &=\{(n_1,n_2) \in k_T: n_1=2z_1, n_2=2z_2+1\quad \exists  \, z_1,z_2 \in\mathbb{N}\} \,,\\
	k_{T_{\mathcal{C}}} &=\{(n_1,n_2) \in k_T: n_1=2z_1+1, n_2=2z_2\quad \exists \, z_1,z_2 \in\mathbb{N}\} \,,\\
	k_{T_{\mathcal{D}}} &=\{(n_1,n_2) \in k_T: n_1=2z_1, n_2=2z_2\quad \exists \, z_1,z_2 \in\mathbb{N}\} \,.
	\end{aligned}
	\end{equation*}
	Then, define the blocks 
	\begin{equation*}
	\mathcal{A}_\mathbf{r}=\{(j_1,j_2) \in\mathbb{N}^2: (r_1-1)p_T+1 \leq j_1 \leq r_1p_T,~(r_2-1)p_T+1 \leq j_2 \leq r_2p_T\}\,,
	\end{equation*}
	for $\mathbf{r} = (r_1,r_2) \in k_{T_{\mathcal{A}}}$
	and $\mathcal{B}_\mathbf{r},\, \mathcal{C}_\mathbf{r}$ and $\mathcal{D}_\mathbf{r}$ are defined similarly.
	Let
	\begin{equation*}
	A_\mathbf{r}=\sum\limits_{\mathbf{j} \in {\mathcal{A}_\mathbf{r}}}\widehat{U}(\mathbf{j}), B_\mathbf{r}=\sum\limits_{\mathbf{j} \in {\mathcal{B}_\mathbf{r}}}\widehat{U}(\mathbf{j}), C_\mathbf{r}=\sum\limits_{\mathbf{j} \in {\mathcal{C}_\mathbf{r}}}\widehat{U}(\mathbf{j}), D_\mathbf{r}=\sum\limits_{\mathbf{j} \in {\mathcal{D}_\mathbf{r}}}\widehat{U}(\mathbf{j})\,.
	\end{equation*}
	Observe that $\{A_\mathbf{r}\}_{\mathbf{r} \in k_{T_\mathcal{A}}}$ are i.i.d. Similarly, $\{B_\mathbf{r}\}_{\mathbf{r} \in k_{T_\mathcal{B}}}$, $\{C_\mathbf{r}\}_{\mathbf{r} \in k_{T_\mathcal{C}}}$ and $\{D_\mathbf{r}\}_{\mathbf{r} \in k_{T_\mathcal{D}}}$ are i.i.d. Also, define the remaining part $\mathcal{R}_T:=T \backslash(k_{T_\mathcal{A}} \bigcup k_{T_\mathcal{B}} \bigcup k_{T_\mathcal{C}} \bigcup k_{T_\mathcal{D}})$ and $R=R(\lambda)=\sum\limits_{\mathbf{j} \in \mathcal{R}_T} \widehat{U}_\mathbf{j}(\mathbf{\lambda})$.
	
	The following two lemmas state the order of the asymptotic behavior of $A_\mathbf{r}$, $B_\mathbf{r}$, $C_\mathbf{r}$, $D_\mathbf{r}$ and $R$ respectively.
	
	\begin{lem}\label{AlemC5}
		Let $\tau=2+\delta/2$. We have ${\max\limits_{\mathbf{\lambda} \in [0,2\pi]^2}}\|A_\mathbf{1}(\mathbf{\lambda})\|_{\tau}=O\left(m^2\sqrt{|\mathbb{B}_T|}p_T\right)\,.$
	\end{lem}
	\begin{proof}[Proof of Lemma \ref{AlemC5}]
		By the definition of $A_\mathbf{1}(\mathbf{\lambda})$,
		\begin{equation*}
		\begin{aligned}
		\|A_\mathbf{1}(\mathbf{\lambda})\|_{\tau} &=\left\|\sum\limits_{\mathbf{j} \in {\mathcal{A}_\mathbf{1}}}\widehat{U}(\mathbf{j})\right\|_{\tau}  =\frac{1}{4\pi^2}\left\|\sum\limits_{\mathbf{j} \in {\mathcal{A}_\mathbf{1}}}\sum\limits_{\mathbf{l}\in \mathbf{L}_B}\widehat{V}(\mathbf{j})\widehat{V}(\mathbf{j}+\mathbf{l})\alpha_\mathbf{l}\right\|_{\tau}\\
		&\leq \frac{1}{4\pi^2}\left\|\sum\limits_{\mathbf{j} \in {\mathcal{A}_\mathbf{1}}}\sum\limits_{\substack{\mathbf{l} \in {\mathbf{L}_f}}}\widehat{V}(\mathbf{j})\widehat{V}(\mathbf{j}+\mathbf{l})\alpha_\mathbf{l}\right\|_{\tau} + \frac{1}{4\pi^2}\left\|\sum\limits_{\mathbf{j} \in {\mathcal{A}_\mathbf{1}}}\sum\limits_{\substack{\mathbf{l} \in {\mathbf{L}_s}}}\widehat{V}(\mathbf{j})\widehat{V}(\mathbf{j}+\mathbf{l})\alpha_\mathbf{l}\right\|_{\tau}\,,
		\end{aligned}
		\end{equation*}
		where $\mathbf{L}_f = \{(l_1,l_2)\in \mathbb{Z}^2: (|l_1|,|l_2|) \in {\mathbb{B}_T}, |l_1|> 2m~\text{or}~|l_2| > 2m \},\ \mathbf{L}_s = \{(l_1,l_2)\in \mathbb{Z}^2: (|l_1|,|l_2|) \in {\mathbb{B}_T}, |l_1|\leq 2m,|l_2| \leq 2m \}$.
		For the first summation, we have 
		\begin{align*}
	 \left\|\sum\limits_{\mathbf{j} \in {\mathcal{A}_\mathbf{1}}}\sum\limits_{\substack{\mathbf{l} \in {\mathbf{L}_f}}}\widehat{V}(\mathbf{j})\widehat{V}(\mathbf{j}+\mathbf{l})\alpha_\mathbf{l}\right\|_{\tau} 
		\leq\ & \sum\limits_{\substack{\mathbf{h} \in \mathbf{H}_f }}\left\|\sum\limits_{\substack{\mathbf{j} \in \mathbf{J}_f}}\sum\limits_{\substack{\mathbf{l} \in {\mathbf{L}_f}}}\widehat{V}(\mathbf{h}+2m\mathbf{j})\widehat{V}(\mathbf{h}+2m\mathbf{j}+\mathbf{l})\alpha_\mathbf{l}\right\|_{\tau}\\
		\leq\ & O\left(m^2\right)\frac{p_T}{m}\left\|\sum\limits_{\substack{\mathbf{l} \in {\mathbf{L}_f}}}\widehat{V}(\mathbf{l})\alpha_\mathbf{l}\right\|_{\tau}\\
		\leq\ & O(p_Tm)\sum\limits_{\substack{\mathbf{h} \in \mathbf{H}_{f} }}\left\|\sum\limits_{\substack{\mathbf{j} \in \mathbf{J}_{B_T} }}\widehat{V}(\mathbf{h}+2m\mathbf{j})\alpha_{\mathbf{h}+2m\mathbf{j}}\right\|_{\tau} \displaybreak[3]\\
		=\ & O(p_Tm)O\left(m^2\sqrt{\frac{|\mathbb{B}_T|}{m^2}}\right) =  O\left(p_T\sqrt{|\mathbb{B}_T|}m^2\right) \,,
		\end{align*}
		where $\mathbf{H}_f = \{(h_1,h_2)\in \mathbb{Z}^2 : 1 \leq h_1 \leq 2m, 1 \leq h_2 \leq 2m \},\ \mathbf{J}_f = \{(j_1,j_2)\in \mathbb{Z}^2 : |j_1| \leq \left\lfloor \frac{p_T-h_1}{2m} \right\rfloor, |j_2| \leq \left\lfloor \frac{p_T-h_2}{2m} \right\rfloor \}$, and $\mathbf{J}_{B_T} = \{(j_1,j_2)\in \mathbb{Z}^2 : |j_1| \leq \left\lfloor \frac{B_{T1}-h_1}{2m} \right\rfloor, |j_2| \leq \left\lfloor \frac{B_{T2}-h_2}{2m} \right\rfloor, (|j_1|,|j_2|) \neq (0,0) , (1,1) \}$.
		
		On the other hand, for the second summation, we have
		\begin{equation*}
		\begin{aligned}
		 \left\|\sum\limits_{\mathbf{j} \in {\mathcal{A}_\mathbf{1}}}\sum\limits_{\substack{\mathbf{l} \in \mathbf{L}_s}}\widehat{V}(\mathbf{j})\widehat{V}(\mathbf{j}+\mathbf{l})\alpha_\mathbf{l}\right\|_{\tau}
		\leq &\sum\limits_{\substack{\mathbf{l} \in \mathbf{L}_s}}\sum\limits_{\substack{\mathbf{h} \in \mathbf{H}_s }}\left\|\sum\limits_{\substack{\mathbf{j} \in \mathbf{J}_s }} \widehat{V}\left(\mathbf{h}+\mathbf{j}'(6m,6m)+\mathbf{l}\right)\widehat{V}\left(\mathbf{h}+\mathbf{j}'(6m,6m)\right)\alpha_\mathbf{l}\right\|_{\tau}\\
		=\ & O\left(m^4\sqrt{\frac{{p_T}^2}{m^2}}\right) =  O\left(m^3p_T\right) \,,
		\end{aligned}
		\end{equation*}
		where $\mathbf{H}_s = \{(h_1,h_2)\in \mathbb{Z}^2 : h_1 \leq 6m, h_2 \leq 6m \},\ \mathbf{J}_s = \{(j_1,j_2)\in \mathbb{Z}^2 : |j_1| \leq \left\lfloor \frac{p_T-h_1}{6m} \right\rfloor, |j_2| \leq \left\lfloor \frac{p_T-h_2}{6m} \right\rfloor \}$.
		\\Then, we have $\|A_\mathbf{1}(\mathbf{\lambda})\|_{\tau} = O\left(p_T\sqrt{|\mathbb{B}_T|}m^2\right)+O\left(m^3p_T\right) =O\left(p_T\sqrt{|\mathbb{B}_T|}m^2\right)$. 
        Thus, it holds uniformly over $\mathbf{\lambda} \in [0,2\pi]^2$ that 
		${\max\limits_{\mathbf{\lambda} \in [0,2\pi]^2}}\|A_\mathbf{1}(\mathbf{\lambda})\|_{\tau}=O\left(m^2\sqrt{|\mathbb{B}_T|}p_T\right)\,.$
	\end{proof}
	
	Next, define $c_T = \sqrt{|T||\mathbb{B}_T|}(\log|T|)^{-\frac{1}{2}}$ and the truncated version
	$A^{*}_\mathbf{r}(\lambda)=A_\mathbf{r}(\lambda)\mathds{1}_{\left\{|A_\mathbf{r}(\mathbf{\lambda})| \leq c_T\right\}}$. 
	Also, we define $B^{*}_\mathbf{r}(\mathbf{\lambda})$, $C^{*}_\mathbf{r}(\mathbf{\lambda})$, $D^{*}_\mathbf{r}(\mathbf{\lambda})$ similarly. The following lemma investigates the asymptotic behavior of each term.
	
	\begin{lem}\label{AlemC6}
		Under the Assumptions in Theorem \ref{AthmC4}, we have 
		\begin{eqnarray}
		\label{AeqC3} \mathbb{E}\left(\max_{\mathbf{\lambda} \in [0,2\pi]^2}|R(\mathbf{\lambda})|\right) &=&O\left(\sqrt{p_T(d_1+d_2-p_T)}m|\mathbb{B}_T|\right) \,,\\ 
		\label{AeqC4} \mathbb{E}\left(\max_{\mathbf{\lambda} \in [0,2\pi]^2}|q_T(\lambda)|\right) &=& o(1) \,,\\
		\label{AeqC5} \max_{\mathbf{\lambda} \in [0,2\pi]^2}\text{Var}(A_\mathbf{1}(\mathbf{\lambda})) &=& O\left(p_T^2|\mathbb{B}_T|\right) \,,\\
		\label{AeqC6} \text{Var}(A^{*}_\mathbf{1}(\lambda)) &=& \text{Var}(A_\mathbf{1}(\mathbf{\lambda}))[1+o(1)] \,,
		\end{eqnarray}
		uniformly on $\lambda \in [0,2\pi]^2$.
	\end{lem} 
	
	\begin{proof}[Proof of Lemma \ref{AlemC6}]
		Define $\tau=2+\frac{\delta}{2}$. We have
		\begin{equation*}
		\mathbb{E}\left(\max_{\mathbf{\lambda} \in [0,2\pi]^2}|R(\mathbf{\lambda})|\right) \leq C\sum\limits_{\mathbf{l}\in \mathbf{L}_B} \mathbb{E}\left|\sum\limits_{\mathbf{j} \in \mathcal{R}_T}\widehat{V}(\mathbf{j})\widehat{V}(\mathbf{j}+\mathbf{l})\right|.
		\end{equation*}
		
		For $|l_1| \leq 2m$ and $|l_2| \leq 2m$, since $\widehat{V}(\mathbf{j})\widehat{V}(\mathbf{j}+\mathbf{l})$ is 4$m$-dependent, we have
		\begin{equation*}
		\left\|\sum\limits_{\mathbf{j} \in \mathcal{R}_T}\widehat{V}(\mathbf{j})\widehat{V}(\mathbf{j}+\mathbf{l})\right\|_2=O\left(\sqrt{|\mathcal{R}_T|}m\right) \,.
		\end{equation*}	
		For $|l_1| > 2m$ or $|l_2| > 2m,$
		\begin{equation*}
		\begin{aligned}
		\left\|\sum\limits_{\mathbf{j} \in \mathcal{R}_T}\widehat{V}(\mathbf{j})\widehat{V}(\mathbf{j}+\mathbf{l})\right\|^2_2 =&\sum\limits_{\mathbf{j}_1 \in \mathcal{R}_T}\sum\limits_{\mathbf{j}_2 \in \mathcal{R}_T}\mathbb{E}\left(\widehat{V}(\mathbf{j}_1)\widehat{V}(\mathbf{j}_1+\mathbf{l})\widehat{V}(\mathbf{j}_2)\widehat{V}(\mathbf{j}_2+\mathbf{l})\right) 
		= O\left(|\mathcal{R}_T|m^2\right) \,,
		\end{aligned}
		\end{equation*}
		since the sum vanishes if $\|\mathbf{j}_1-\mathbf{j}_2\|_{\infty}>2m$, $\mathbb{E}(\widehat{V}(\mathbf{j}))=0$.
		
		As a result,
		\begin{equation*}
		\begin{aligned}
		\mathbb{E}\left(\max_{\mathbf{\lambda} \in [0,2\pi]^2}|R(\mathbf{\lambda})|\right) =&\ O\left(\sqrt{|\mathcal{R}_T|}m|\mathbb{B}_T|\right)\\
		=&\ O\left(\sqrt{p_T(d_1+d_2-p_T)}m|\mathbb{B}_T|\right) \,.
		\end{aligned}
		\end{equation*}
		Let $\widehat{q}_T(\mathbf{\lambda})$ be the corresponding sum of $q_T(\mathbf{\lambda})$ with $V(\mathbf{j})V(\mathbf{j}+\mathbf{l})$ replaced by $\widehat{V}(\mathbf{j})\widehat{V}(\mathbf{j}+\mathbf{l})$. As in (\ref{AeqC2}), we have
		$\mathbb{E}\left(\max_{\mathbf{\lambda} \in [0,2\pi]^2}|q_T(\mathbf{\lambda})-\widehat{q}_T(\mathbf{\lambda})|\right)=o(1)\,$.

		To show (\ref{AeqC4}), it suffices to show $\mathbb{E}\left(\max_{\mathbf{\lambda} \in [0,2\pi]^2}|\widehat{q}_T(\mathbf{\lambda})|\right)=o(1)$, which follows from a similar arguments as in the proof of (\ref{AeqC3}).
		
		Regarding (\ref{AeqC5}) and let $\mathbf{l} = (l_1,l_2),\mathbf{l}^* = (l^*_1,l^*_2)$, we have 
		\begin{equation*}
		\begin{aligned}
		\text{Var}(A_\mathbf{1}(\mathbf{\lambda})) &=\left\|\sum\limits_{\mathbf{j} \in {\mathcal{A}_\mathbf{1}}} \sum\limits_{\mathbf{l}\in \mathbf{L}_B}[V(\mathbf{j})V(\mathbf{j}+\mathbf{l})-\gamma(\mathbf{l})]\alpha_\mathbf{l}\right\|^2\\
		&= \sum\limits_{\mathbf{j},\mathbf{j}^* \in {\mathcal{A}_\mathbf{1}}}\sum\limits_{\substack{\mathbf{l},\mathbf{l}^* \in \mathbf{L}_B}}\big[\gamma(\mathbf{j}-\mathbf{j}^*)\gamma(\mathbf{j}-\mathbf{j}^*+\mathbf{l}-\mathbf{l}^*)+\gamma(\mathbf{j}^*-\mathbf{j}+\mathbf{l}^*)\gamma(\mathbf{j}^*-\mathbf{j}-\mathbf{l})\\
		&\quad\quad\quad\quad\quad\quad\quad\quad\quad\quad +\text{cum}(V(\mathbf{0}),V(\mathbf{l}),V(\mathbf{j}^*-\mathbf{j}),V(\mathbf{j}^*-\mathbf{j}+\mathbf{l}^*))\big]\alpha_\mathbf{l} \alpha_{\mathbf{l}^*}\\
		&= I_1+I_2+I_3\,,
		\end{aligned}
		\end{equation*}
		where 
		\begin{eqnarray*}
			I_1 & = & \sum\limits_{\mathbf{j},\mathbf{j}^* \in {\mathcal{A}_\mathbf{1}}}\sum\limits_{\substack{\mathbf{l},\mathbf{l}^* \in \mathbf{L}_B}} \gamma(\mathbf{j}-\mathbf{j}^*)\gamma(\mathbf{j}-\mathbf{j}^*+\mathbf{l}-\mathbf{l}^*) \alpha_\mathbf{l}\alpha_{\mathbf{l}^*} \,,\\
			I_2 &=&\sum\limits_{\mathbf{j},\mathbf{j}^* \in {\mathcal{A}_\mathbf{1}}}\sum\limits_{\substack{\mathbf{l},\mathbf{l}^* \in \mathbf{L}_B}} \gamma(\mathbf{j}^*-\mathbf{j}+\mathbf{l}^*)\gamma(\mathbf{j}^*-\mathbf{j}-\mathbf{l}) \alpha_\mathbf{l}\alpha_{\mathbf{l}^*} \,,\\
			I_3&=&\sum\limits_{\mathbf{j},\mathbf{j}^* \in {\mathcal{A}_\mathbf{1}}}\sum\limits_{\substack{\mathbf{l},\mathbf{l}^* \in \mathbf{L}_B}}\text{cum}(V(\mathbf{0}),V(\mathbf{l}),V(\mathbf{j}^*-\mathbf{j}),V(\mathbf{j}^*-\mathbf{j}+\mathbf{l}^*))\alpha_\mathbf{l} \alpha_{\mathbf{l}^*} \,.
		\end{eqnarray*}
		Note that $I_1$ is bounded by 
		\begin{equation*}
		C\sum\limits_{\mathbf{j} \in {\mathcal{A}_\mathbf{1}}}(p_T-|j_1|)(p_T-|j_2|)|\gamma(\mathbf{j})|\sum\limits_{\mathbf{l} \in \mathbf{L}_{2B}} (2B_{T1}+1-|l_1|)(2B_{T2}+1-|l_2|)|\gamma(\mathbf{j}+\mathbf{l})|\,,
		\end{equation*}
		where $\mathbf{L}_{2B}=\{(l_1,l_2)\in \mathbb{Z}^2: (|l_1|,|l_2|) \in 2{\mathbb{B}_T}\}$ with $2\mathbb{B}_T=\{(l_1,l_2)\in \mathbb{Z}^2:1 \leq l_k \leq 2B_{Tk}, k=1,2\}$, where $(2B_{T1},2B_{T2})=\left(\frac{2}{{h_T}_1},\frac{2}{{h_T}_2}\right)$, which is less than 
		\begin{equation*}
		Cp^2_T\left(4|\mathbb{B}_T|+o(|\mathbb{B}_T|)\right) \left(\sum\limits_{\mathbf{l} \in\mathbb{Z}^2}|\gamma(\mathbf{l})| \right)^2=O\left(p^2_T|\mathbb{B}_T|\right) \,.
		\end{equation*}
		Similarly, smaller bounds can be obtained for $I_2$ and $I_3$ due to the summability of the second and fourth cumulants. Thus, we have
		$\max_{\mathbf{\lambda} \in [0,2\pi]^2} \text{Var}(A_\mathbf{1}(\mathbf{\lambda}))=O\left(p^2_T|\mathbb{B}_T|\right)$.
 
		For (\ref{AeqC6}), let $\nu=\text{Var}(A_\mathbf{1}(\mathbf{\lambda})-A^{*}_\mathbf{1}(\mathbf{\lambda}))$ 
		and $c=\mathbb{E}(A^{*}_\mathbf{1}(\mathbf{\lambda}))\mathbb{E}(A_\mathbf{1}(\mathbf{\lambda})-A^{*}_\mathbf{1}(\mathbf{\lambda}))$. 
		Then, 
		\begin{equation*}
		\text{Var}(A^{*}_\mathbf{1}(\mathbf{\lambda}))=\text{Var}(A_\mathbf{1}(\mathbf{\lambda}))-\nu+2c\,.
		\end{equation*}
		By Markov's inequality, Holder's inequality and Lemma \ref{AlemC5},
		\begin{equation*}
		\begin{aligned}
		\nu &= \text{Var}(A_\mathbf{1}(\mathbf{\lambda})-A^{*}_\mathbf{1}(\mathbf{\lambda})) = \text{Var}\left(A_\mathbf{1}(\mathbf{\lambda})\mathds{1}_{\left\{|A_\mathbf{1}(\mathbf{\lambda})|>c_T\right\}}\right)\\
& \leq \mathbb{E}\left(|A^2_\mathbf{1}(\mathbf{\lambda})|^{\frac{\tau}{2}}\right)^{\frac{2}{\tau}}\left[\mathbb{E}\left(\mathds{1}_{\left\{|A_\mathbf{1}(\mathbf{\lambda})|>c_T\right\}}\right)\right]^{\frac{\tau-2}{\tau}}\\
		&\leq \frac{\mathbb{E}(|A_\mathbf{1}(\mathbf{\lambda})|^{\tau})}{{c_T}^{\tau-2}} = \frac{O\left(m^2\sqrt{|\mathbb{B}_T|}p_T\right)^{\tau}}{{c_T}^{\tau-2}} = o\left(p^2_T|\mathbb{B}_T|\right)\,,
		\end{aligned}
		\end{equation*}
		and similarly, 
		$c \leq \|A_\mathbf{1}(\mathbf{\lambda})\|^{\tau+1}_{\tau}/{c_T}^{\tau-1}= o\left(p^2_T|\mathbb{B}_T|\right)$. 
		By Lemma \ref{AlemC3} and the fact that $f$ is every-where positive, (\ref{AeqC6}) follows.
	\end{proof}
	
	Using Lemmas \ref{AlemC5} and \ref{AlemC6}, we now continue the proof of Theorem \ref{AthmC4}. Let
	\begin{equation*}
	\begin{aligned}
	&G_T(\mathbf{\lambda})\ =\sum\limits_{\mathbf{r} \in k_{T_A}}[A_\mathbf{r}(\mathbf{\lambda})-\mathbb{E}(A_\mathbf{r}(\mathbf{\lambda}))],
	&G^{*}_T(\mathbf{\lambda}) =\sum\limits_{\mathbf{r} \in k_{T_A}}[A^{*}_\mathbf{r}(\mathbf{\lambda})-\mathbb{E}(A^{*}_\mathbf{r}(\mathbf{\lambda}))] \,.
	\end{aligned}
	\end{equation*}
	Let $\mathbf{\lambda}_\mathbf{j}=\left(\frac{\pi j_1}{t_{T1}},\frac{\pi j_2}{t_{T2}}\right)$ where $j_1=0,\ldots,t_{T1}:=\lfloor B_{T1}\log(B_{T1}) \rfloor, ~j_2=0,\ldots,t_{T2}:=\lfloor B_{T2}\log(B_{T2}) \rfloor$.
	Let 
	\begin{equation*}
	\begin{aligned}
	C_{Ti}=\frac{1}{1-\frac{3\pi}{\log B_{Ti}}} \rightarrow 1 \,,
	\end{aligned}
	\end{equation*}
	for $i=1,2$. By the argument similar to Corollary 2.1 in \cite{WM1967}, we have
	\begin{equation*}
	\max_{\mathbf{\lambda} \in [0,2\pi]^2}|G_T(\mathbf{\lambda})| \leq C_{T1}C_{T2} \max_{\substack{j_1 \leq t_{T1} \\ j_2 \leq t_{T2}}} |G_T(\mathbf{\lambda}_\mathbf{j})| \,.
	\end{equation*}
	By (\ref{AeqC5}) and (\ref{AeqC6}), there exists a constant $C_1>1$ such that 
	$\max_{\mathbf{\lambda} \in [0,2\pi]^2} \text{Var}(A^{*}_\mathbf{1}(\mathbf{\lambda})) \leq C_1 p^2_T|\mathbb{B}_T| \,$.\\
	Let $\beta_T=(C_1|T||\mathbb{B}_T|\log|T|)^{\frac{1}{2}}$, By Bernstein's inequality, we have 
	\begin{equation*}
	\begin{aligned}
	\mathbb{P}\left( \max_{\substack{0 \leq j_1 \leq t_{T1} \\ 0 \leq j_2 \leq t_{T2}}}|G^{*}_T(\mathbf{\lambda}_\mathbf{j})| \geq 4\beta_T\right) &\leq \sum\limits^{t_{T1}}_{j_1=0} \sum\limits^{t_{T2}}_{j_2=0}\mathbb{P}(|G^{*}_T(\mathbf{\lambda}_\mathbf{j})| \geq 4\beta_T)\\
	&= O(t_{T1}t_{T2})\exp \left(\frac{-16\beta^2_T}{2|k_{T_\mathcal{A}}|C_1p^2_T|\mathbb{B}_T|+16c_T\beta_T}\right)
	=o(1) \,.
	\end{aligned}
	\end{equation*}
	Let $A^{**}_\mathbf{r}(\mathbf{\lambda})=A_\mathbf{r}(\mathbf{\lambda})-A^{*}_\mathbf{r}(\mathbf{\lambda})$ and $G^{**}_T(\mathbf{\lambda})=G_T(\mathbf{\lambda})-G^{*}_T(\mathbf{\lambda}).$ Then, by Markov's inequality and Lemma \ref{AlemC5},
	\begin{equation*}
	\begin{aligned}
	\mathbb{P}\left( \max_{\substack{0 \leq j_1 \leq t_{T1} \\ 0 \leq j_2 \leq t_{T2}}}|G^{**}_T(\mathbf{\lambda}_\mathbf{j})| \geq 4\beta_T \right) &\leq \sum\limits^{t_{T1}}_{j_1=0} \sum\limits^{t_{T2}}_{j_2=0}\mathbb{P}(|G^{**}_T(\mathbf{\lambda}_\mathbf{j})| \geq 4\beta_T) 
	\leq \sum\limits^{t_{T1}}_{j_1=0} \sum\limits^{t_{T2}}_{j_2=0} \frac{\text{Var}(A^{**}_\mathbf{r}(\mathbf{\lambda}_\mathbf{j}))|k_{T_{\mathcal{A}}}|}{16\beta^2_T}\\
	&= \frac{O\left(t_{T1}t_{T2}|k_{T_{\mathcal{A}}}|p^{\tau}_T|\mathbb{B}_T|^{\frac{\tau}{2}}m^{2\tau}\right)}{\beta^2_Tc^{\tau-2}_T}
	= O\left((\log|T|)^{-\frac{\delta}{4}}\right) = o(1) \,.
	\end{aligned}
	\end{equation*}
	Therefore, we have $\max\limits_{\mathbf{\lambda} \in [0,2\pi]^2}|G_T(\mathbf{\lambda})|=O_p(\beta_T)$. Using similar arguments, the same bound also holds for the sums
$\sum\limits_{\mathbf{r} \in k_{T_\mathcal{B}}}[B_\mathbf{r}(\lambda)-\mathbb{E}(B_\mathbf{r}(\lambda))]$, 
$\sum\limits_{\mathbf{r} \in k_{T_\mathcal{C}}}[C_\mathbf{r}(\lambda)-\mathbb{E}(C_\mathbf{r}(\lambda))]$, and 
$\sum\limits_{\mathbf{r} \in k_{T_\mathcal{D}}}[D_\mathbf{r}(\lambda)-\mathbb{E}(D_\mathbf{r}(\lambda))]$. 
	By (\ref{AeqC2}), we have $\mathbb{E}\left(\max_{\mathbf{\lambda} \in [0,2\pi]^2}|\widehat{g}_T(\mathbf{\lambda})-g_T(\mathbf{\lambda})|\right)=o(1)$. As a result, by (\ref{AeqC3}), (\ref{AeqC4}) and (\ref{AeqC1}), the proof is completed.    
\end{proof}

The following lemma states that $\widehat{f}_T(\mathbf{\lambda})$ in \eqref{kernelesti} and $\widetilde{f}_T(\mathbf{\lambda})$ are asymptotically equivalent.
\begin{lem}\label{lemequivalent}
	Under Assumptions \ref{assk1} and \ref{assdepenmeasure} with $p>2$, we have
	$\max_{\mathbf{\lambda} \in {[0,2\pi]^2}} \left|\widetilde{f}_T(\mathbf{\lambda})-\widehat{f}_T(\mathbf{\lambda}) \right|=o(1)$. 
\end{lem}
\begin{proof}[Proof of Lemma \ref{lemequivalent}]
	For any $\mathbf{\lambda}=(\lambda_1,\lambda_2) \in {[0,2\pi]^2}$, we have
	\begin{equation*}
	\begin{aligned}
	\widehat{f}_T(\mathbf{\lambda})=\ &\frac{\sum_{\mathbf{j} \in\mathbb{Z}^2}K\left(\frac{\lambda_1-\lambda_{j1}}{{h_T}_1},\frac{\lambda_2-\lambda_{j2}}{{h_T}_2}\right)I(\mathbf{j})}{4\pi^2\sum_{\mathbf{j} \in\mathbb{Z}^2}K\left(\frac{\lambda_{j1}}{{h_T}_1},\frac{\lambda_{j2}}{{h_T}_2}\right)}\\
	=\ &\frac{1}{|\mathbf{h}_T||T|}\sum\limits_{\mathbf{j} \in\mathbb{Z}^2}K\left(\frac{\lambda_1-\lambda_{j1}}{{h_T}_1},\frac{\lambda_2-\lambda_{j2}}{{h_T}_2}\right)I(\mathbf{j})+o(1)\\
	=\ &\frac{1}{|\mathbf{h}_T||T|}\sum\limits_{\mathbf{t} \in T}I(\mathbf{t})\sum\limits_{\mathbf{j} \in\mathbb{Z}^2}K\left(\frac{\lambda_1-\lambda_{t1}+2\pi j_1}{{h_T}_1},\frac{\lambda_2-\lambda_{t2}+2\pi j_2}{{h_T}_2}\right)+o(1)\\
	=\ &\frac{1}{|T|}\sum\limits_{\mathbf{t} \in T}I(\mathbf{t})K_\mathbf{h}(\mathbf{\lambda}-\mathbf{\lambda}_\mathbf{t})+o(1)\\
	=\ &\frac{1}{4\pi^2|T|}\sum\limits_{\mathbf{t} \in T}I(\mathbf{t})\sum\limits_{\mathbf{j} \in\mathbb{Z}^2}k(j_1 {h_T}_1,j_2 {h_T}_2)\exp(-i\mathbf{j}'(\mathbf{\lambda}-\mathbf{\lambda}_\mathbf{t}))+o(1)\\
	=\ &\frac{1}{4\pi^2}\sum\limits_{\mathbf{j} \in\mathbb{Z}^2}\left(\frac{1}{|T|}\sum\limits_{\mathbf{t} \in T}I(\mathbf{t})\exp(i\mathbf{j}'\mathbf{\lambda}_\mathbf{t})\right)k(j_1 {h_T}_1,j_2 {h_T}_2)\exp(-i\mathbf{j}'\mathbf{\lambda})+o(1)\\
	=\ &\frac{1}{4\pi^2}\sum\limits_{\mathbf{j} \in\mathbb{Z}^2}\widehat{R}_V(\mathbf{j})k(j_1 {h_T}_1,j_2 {h_T}_2)\exp(-i\mathbf{j}'\mathbf{\lambda})+o(1)
	= \widetilde{f}_T(\mathbf{\lambda})+o(1)\,.
	\end{aligned}
	\end{equation*}
	For the last two equalities, note that by Assumption \ref{assdepenmeasure} with $p>2$, $[\bar{V}_T-\mathbb{E}(V(\mathbf{0}))]^2$ is asymptotically negligible.
\end{proof}

\begin{proof}[Proof of Theorem \ref{thm3}(b)]	
	As a direct consequence of Theorem \ref{AthmC4} and Lemma \ref{lemequivalent}, we have 
	\begin{equation*}
	\max_{\mathbf{\lambda} \in {[0,2\pi]^2}} \left|\widehat{f}_T(\mathbf{\lambda})-\mathbb{E} (\widetilde{f}_T(\mathbf{\lambda})) \right|=o(1)\,.
	\end{equation*}
	It remains to show that $\max_{\mathbf{\lambda} \in {[0,2\pi]^2}} \left|\mathbb{E}(\widetilde{f}_T(\mathbf{\lambda}))-f(\mathbf{\lambda}) \right|=o(1)$.
	This holds since by Assumption \ref{assk4}, $k(\cdot)$ is bounded, continuous and with compact support $[-1,1]^2$. Also, $k(0)=1$ as $T \rightarrow \infty$. Thus, for any ${\mathbf{\lambda} \in {[0,2\pi]^2}}$, we have
	\begin{equation}\label{dev_proof_51b}
	\begin{aligned}
	&\left|\mathbb{E}(\widetilde{f}_T(\mathbf{\lambda}))-f(\mathbf{\lambda})\right|\\
	=\ &\left|\frac{1}{4\pi^2}\sum\limits_{\mathbf{j} \in \mathbf{L}_T}\frac{(d_1-|j_1|)(d_2-|j_2|)}{|T|}\gamma(\mathbf{j})k(j_1 {h_T}_1,j_2 {h_T}_2)\exp(-i\mathbf{j}'\mathbf{\lambda})-\frac{1}{4\pi^2}\sum\limits_{\mathbf{j} \in\mathbb{Z}^2}\gamma(\mathbf{j})\exp(i\mathbf{j}'\mathbf{\lambda})\right|\\
	\asymp\ &\sum\limits_{\mathbf{j}:\|\mathbf{j}\| \geq \sqrt{\frac{1}{\|\mathbf{h}_T\|}}}|\gamma(\mathbf{j})|\ +\sum\limits_{\mathbf{j}:\|\mathbf{j}\| \leq \sqrt{\frac{1}{\|\mathbf{h}_T\|}}}\left[\left(\frac{(d_1-|j_1|)(d_2-|j_2|)}{|T|}\right)k(j_1 {h_T}_1,j_2 {h_T}_2)-1\right]|\gamma(\mathbf{j})|\\
	=\ &o(1)\,,
	\end{aligned}
	\end{equation}
	where $\mathbf{L}_T=\{\mathbf{j} \in \mathbb{Z}^2: (|j_1|,|j_2|) \in T\}$. 
	The above arguments used the fact that for $\|\mathbf{j}\| \geq \sqrt{\frac{1}{\|\mathbf{h}_T\|}}$, we have
	\begin{equation*}
	\begin{aligned}
	\left|\left(\frac{(d_1-|j_1|)(d_2-|j_2|)}{|T|}\right)k(j_1 {h_T}_1,j_2 {h_T}_2)-1\right|
	\leq  \left|\frac{(d_1-|j_1|)(d_2-|j_2|)}{|T|}C^*\right|+1
	\leq  C^*+1\,,
	\end{aligned}
	\end{equation*}
	and for $\|\mathbf{j}\| \leq \sqrt{\frac{1}{\|\mathbf{h}_T\|}}$, we have
	\begin{equation*}
	\begin{aligned}
	&\left(\frac{(d_1-|j_1|)(d_2-|j_2|)}{|T|}\right)k(j_1 {h_T}_1,j_2 {h_T}_2)-k(\mathbf{0})\\
	=\ &\left(1-\frac{|j_1|}{d_1}\right)\left(1-\frac{|j_2|}{d_2}\right)\left[k(j_1 {h_T}_1,j_2 {h_T}_2)-k(\mathbf{0})\right]+\frac{|j_1|}{d_1}+\frac{|j_2|}{d_2}-\frac{|j_1||j_2|}{d_1d_2}\\
	\asymp\ & \sup_{\|\mathbf{x}\| \leq \sqrt{\|\mathbf{h}_T\|}}{|k(x_1,x_2)-k(\mathbf{0})|}+\frac{1}{\sqrt{\|\mathbf{h}_T\|}d_1}+\frac{1}{\sqrt{\|\mathbf{h}_T\|}d_2}\\
	=\ & o(1)+o(1)+o(1)	=  o(1)\,.
	\end{aligned}
	\end{equation*} 
	
	
\end{proof}

\subsection{Proof of Theorem \ref{thm_estimated}}

\begin{proof}[Proof of Theorem \ref{thm_estimated}]
	First, note that we have
	\begin{equation*}
	\left|\sum\limits_{\mathbf{j} \in T}\left[\cos(\mathbf{\lambda}_\mathbf{j}' \mathbf{t_1})\cos(\mathbf{\lambda}_\mathbf{j}' \mathbf{t_2})+\sin(\mathbf{\lambda}_\mathbf{j}' \mathbf{t_1})\sin(\mathbf{\lambda}_\mathbf{j}' \mathbf{t_2})\right]\right| = \left\{ \begin{array}{lcr}
	|T|, & \mathbf{t_1}=\mathbf{t_2}\,, \\
	0, & \mathbf{t_1} \neq \mathbf{t_2}\,.
	\end{array}\right.
	\end{equation*}	
	
	For $\mathbf{t_1}=\mathbf{t_2}$, the equation is trivial by considering $\cos^2x+\sin^2x=1$.
	
	For $\mathbf{t_1} \neq \mathbf{t_2}$, we have
	\begin{equation*}
	\begin{aligned}
	&\sum\limits_{\mathbf{j} \in T}\left[\cos(\mathbf{\lambda}_\mathbf{j}'\mathbf{t_1})\cos(\mathbf{\lambda}_\mathbf{j}'\mathbf{t_2})+\sin(\mathbf{\lambda}_\mathbf{j}' \mathbf{t_1})\sin(\mathbf{\lambda}_\mathbf{j}' \mathbf{t_2})\right]\\
	=\ & \sum\limits_{\mathbf{j} \in T}\cos(\mathbf{\lambda}_\mathbf{j}' (\mathbf{t_1}-\mathbf{t_2})) 
	=\ \sum\limits_{\mathbf{j} \in T}\cos(\mathbf{\lambda}_\mathbf{j}'\mathbf{c})\quad\quad\quad (\text{By letting}~\mathbf{t_1}-\mathbf{t_2}=\mathbf{c} \in\mathbb{Z}^2)\\
	=\ & \sum\limits_{\mathbf{j} \in T}\left[\cos(\lambda_{j1}c_1)\cos(\lambda_{j2}c_2)-\sin(\lambda_{j1}c_1)\sin(\lambda_{j2}c_2)\right]\\
	=\ & \sum\limits^{d_1}_{j_1=1} \sum\limits^{d_2}_{j_2=1}\left[\cos\left(\frac{2\pi j_1}{d_1}c_1\right)\cos\left(\frac{2\pi j_2}{d_2}c_2\right)-\sin\left(\frac{2\pi j_1}{d_1}c_1\right)\sin\left(\frac{2\pi j_2}{d_2}c_2\right)\right]\\
	=\ &\left[\sum\limits^{d_1}_{j_1=1}\cos\left(\frac{2\pi j_1}{d_1}c_1\right)\right]\left[\sum\limits^{d_2}_{j_2=1}\cos\left(\frac{2\pi j_2}{d_2}c_2\right)\right]-\left[\sum\limits^{d_1}_{j_1=1}\sin\left(\frac{2\pi j_1}{d_1}c_1\right)\right]\left[\sum\limits^{d_2}_{j_2=1}\sin\left(\frac{2\pi j_2}{d_2}c_2\right)\right]\\
	=\ &0\,.
	\end{aligned}
	\end{equation*}

	Furthermore, denote  
	\begin{equation*}
	\begin{aligned}
	&F_T(\mathbf{j})= \sum\limits_{\mathbf{t} \in T}(V(\mathbf{t})-\widehat{V}(\mathbf{t}))\cos(\mathbf{\lambda}_\mathbf{j}' \mathbf{t}),
	&G_T(\mathbf{j})= \sum\limits_{\mathbf{t} \in T}(V(\mathbf{t})-\widehat{V}(\mathbf{t}))\sin(\mathbf{\lambda}_\mathbf{j}' \mathbf{t}), \quad \mathbf{j} \in T\,.
	\end{aligned}
	\end{equation*}
	By the previous equality and the definition of the kernel spectral density estimator in (\ref{kernelesti}), applying the Cauchy-Schwarz inequality yields
	\begin{equation*}
	\begin{aligned}
	|F_T(\mathbf{j})| & \leq \sum\limits_{\mathbf{t} \in T}|V(\mathbf{t})-\widehat{V}(\mathbf{t})| \leq \sqrt{\sum\limits_{\mathbf{t} \in T}(V(\mathbf{t})-\widehat{V}(\mathbf{t}))^2\sum\limits_{\mathbf{t} \in T}1^2}\\
	& =\sqrt{\left(\frac{1}{|T|}\sum\limits_{\mathbf{t} \in T}(V(\mathbf{t})-\widehat{V}(\mathbf{t}))^2\right)|T|^2} = o_p\left(|T|\alpha^{-\frac{1}{2}}_T\right)\,.
	\end{aligned}
	\end{equation*}
	Similarly, $|G_T(\mathbf{j})|=o_p\left(|T|\alpha^{-\frac{1}{2}}_T\right)$.
	Also, we have
	\begin{equation*}
	\begin{aligned}
	&\sum\limits_{\mathbf{j} \in T}(F^2_T(\mathbf{j})+G^2_T(\mathbf{j}))\\
	=\ &\sum\limits_{\mathbf{t_1} \in T} \sum\limits_{\mathbf{t_2} \in T}(V(\mathbf{t_1})-\widehat{V}(\mathbf{t_1}))(V(\mathbf{t_2})-\widehat{V}(\mathbf{t_2}))\left[\sum\limits_{\mathbf{j} \in T}(\cos(\mathbf{\lambda}_\mathbf{j}' \mathbf{t_1})\cos(\mathbf{\lambda}_\mathbf{j}' \mathbf{t_2}) +\sin(\mathbf{\lambda}_\mathbf{j}' \mathbf{t_1})\sin(\mathbf{\lambda}_\mathbf{j}' \mathbf{t_2}))\right]\\
	\asymp\ & |T|\sum\limits_{\mathbf{t} \in T}(V(\mathbf{t})-\widehat{V}(\mathbf{t}))^2+\sum\limits_{\mathbf{t_1} \neq \mathbf{t_2}}|(V(\mathbf{t_1})-\widehat{V}(\mathbf{t_1}))(V(\mathbf{t_2})-\widehat{V}(\mathbf{t_2}))|\times 0\\
	=\ &o_p\left(|T|^2\alpha^{-1}_T\right)\,.
	\end{aligned}
	\end{equation*}
	With this definition, we get
	\begin{eqnarray*}
	x_{V}(\mathbf{j})-x_{\widehat{V}}(\mathbf{j})&=&\frac{1}{\sqrt{|T|}}\sum\limits_{\mathbf{t} \in T}(V(\mathbf{t})-\widehat{V}(\mathbf{t}))\cos(-\mathbf{\lambda}_\mathbf{j}' \mathbf{t})=|T|^{-\frac{1}{2}}F_T(\mathbf{j})\,,\\
	y_V(\mathbf{j})-y_{\widehat{V}}(\mathbf{j})&=&\frac{1}{\sqrt{|T|}}\sum\limits_{\mathbf{t} \in T}(V(\mathbf{t})-\widehat{V}(\mathbf{t}))\sin(-\mathbf{\lambda}_\mathbf{j}' \mathbf{t})=|T|^{-\frac{1}{2}}(-G_T(\mathbf{j}))\,.
	\end{eqnarray*}
	Since $a^2-b^2=-(a-b)^2+2a(a-b)$, we have  
	\begin{equation*}
	\begin{aligned}
	I_V(\mathbf{j})-I_{\widehat{V}}(\mathbf{j}) 
	=\ &[x^2_V(\mathbf{j})-x^2_{\widehat{V}}(\mathbf{j})]+[y^2_V(\mathbf{j})-y^2_{\widehat{V}}(\mathbf{j})]\\
	=\ &-(x_V(\mathbf{j})-x_{\widehat{V}}(\mathbf{j}))^2+2x_V(\mathbf{j})(x_V(\mathbf{j})-x_{\widehat{V}}(\mathbf{j}))\\
	& \quad\quad -(y_V(\mathbf{j})-y_{\widehat{V}}(\mathbf{j}))^2+2y_V(\mathbf{j})(y_V(\mathbf{j})-y_{\widehat{V}}(\mathbf{j}))\\
	=\ &-\frac{1}{|T|}(F^2_T(\mathbf{j})+G^2_T(\mathbf{j}))+2\frac{1}{\sqrt{|T|}}x_V(\mathbf{j})F_T(\mathbf{j})-2\frac{1}{\sqrt{|T|}}y_V(\mathbf{j})G_T(\mathbf{j})\,.
	\end{aligned}
	\end{equation*}
	
		
		Now, we can prove the assertion in Theorem \ref{thm_estimated}.
		For the proof of (a), recall that by the assumption $K(\mathbf{x}) \geq 0$, we have $\sup|K(\mathbf{x})| < \infty$, 
$		\sup_{\mathbf{\lambda} \in [0,2\pi]^2}|K_\mathbf{h}(\mathbf{\lambda})|=O\left(|\mathbf{h}_T|^{-1}\right)=O\left(|{h_T}_1 {h_T}_2|^{-1}\right)
$, and
		\begin{equation*}
		\frac{4\pi^2}{|T||\mathbf{h}_T|}\sum\limits_{\mathbf{j} \in\mathbb{Z}^2}K \left(\frac{2\pi j_1}{{h_T}_1 d_1},\frac{2\pi j_2}{{h_T}_2 d_2} \right)=1+o(1)\,.
		\end{equation*}
		Let
		\begin{equation*}
		p_{\mathbf{l},T}=\frac{K\left(\frac{2\pi l_1}{{h_T}_1 d_1},\frac{2\pi l_2}{{h_T}_2 d_2}\right)}{\sum\limits_{\mathbf{j} \in\mathbb{Z}^2}K\left(\frac{2\pi j_1}{{h_T}_1 d_1},\frac{2\pi j_2}{{h_T}_2 d_2}\right)}\,.
		\end{equation*}
		By the Cauchy-Schwarz inequality,
		\begin{equation*}
		\begin{aligned}
		&\sup_{\mathbf{j} \in T}|\widehat{f}_V(\mathbf{\lambda}_\mathbf{j})-\widehat{f}_{\widehat{V}}(\mathbf{\lambda}_\mathbf{j})| 
		\ \leq\ \sup_{\mathbf{j} \in T} \sum\limits_{\mathbf{k} \in\mathbb{Z}^2}p_{\mathbf{j}-\mathbf{k},T}|I_V(\mathbf{k})-I_{\widehat{V}}(\mathbf{k})|\\
		\asymp\  &\frac{1}{|\mathbf{h}_T||T|^2}\sum\limits_{\mathbf{k} \in T}(F^2_T(\mathbf{k})+G^2_T(\mathbf{k}))+\sup_{\mathbf{j} \in T}\frac{1}{\sqrt{|T|}}\sum\limits_{\mathbf{k} \in\mathbb{Z}^2}p_{\mathbf{j}-\mathbf{k},T}\left(x_V(\mathbf{k})F_T(\mathbf{k})-y_V(\mathbf{k})G_T(\mathbf{k})\right)\\
		\asymp\  &o_p\left(\frac{1}{|\mathbf{h}_T|\alpha_T}\right)+\sup_{\mathbf{j} \in T}\frac{1}{\sqrt{|T|}}\sqrt{\left(\sum\limits_{\mathbf{k} \in\mathbb{Z}^2}p_{\mathbf{j}-\mathbf{k},T}I_V(\mathbf{k})\right)\left(\sum\limits_{\mathbf{k} \in\mathbb{Z}^2}p_{\mathbf{j}-\mathbf{k},T}(F^2_T(\mathbf{k})+G^2_T(\mathbf{k}))\right)}\\
		=\ &o_p\left((|\mathbf{h}_T|\alpha_T)^{-1}\right)+o_p\left((|\mathbf{h}_T|\alpha_T)^{-\frac{1}{2}}\right) 
		= o_p(1)\,,
		\end{aligned}
		\end{equation*}
		which yields the assertion in part (a). 
		
		For the proof of (b)(i), note that
		\begin{align*}
		\frac{1}{|N|}\sum\limits_{\mathbf{j} \in N}\frac{x_V(\mathbf{j})-x_{\widehat{V}}(\mathbf{j})+y_V(\mathbf{j})-y_{\widehat{V}}(\mathbf{j})}{\sqrt{f(\mathbf{\lambda}_\mathbf{j})}} 
         =\ & \frac{1}{|N|}\sum\limits_{\mathbf{j} \in N}\frac{1}{\sqrt{|T|}}\frac{F_T(\mathbf{j})-G_T(\mathbf{j})}{\sqrt{f(\mathbf{\lambda}_\mathbf{j})}}\\
		\leq\ &\frac{\sqrt{2}}{|N|\sqrt{|T|}}\sqrt{\sum\limits_{\mathbf{j} \in N}\frac{1}{f(\mathbf{\lambda}_\mathbf{j})}\sum\limits_{\mathbf{j} \in N}(F^2_T(\mathbf{j})+G^2_T(\mathbf{j}))}\\
		\asymp\ &\frac{1}{|T|^{\frac{3}{2}}}\sqrt{O(|N|)o_p\left(|T|^2\alpha^{-1}_T\right)} = o_p\left(\alpha^{-\frac{1}{2}}_T\right)	= o_p(1)\,.
		\end{align*}
		It yields the assertion in part (b)(i). 
		
		For the proof of (b)(ii), using Assumption \ref{assspecpostive}, similar to the argument above, we get
		\begin{equation*}
		\begin{aligned}
		&\left|\frac{1}{|N|}\sum\limits_{\mathbf{j} \in N}\frac{I_V(\mathbf{j})-I_{\widehat{V}}(\mathbf{j})}{f(\mathbf{\lambda}_\mathbf{j})}\right|\\
		=\ &\left|\frac{1}{|N|}\sum\limits_{\mathbf{j} \in N}\frac{-\frac{1}{|T|}(F^2_T(\mathbf{j})+G^2_T(\mathbf{j}))+\frac{2}{\sqrt{|T|}}(x_V(\mathbf{j})F_T(\mathbf{j})-y_V(\mathbf{j})G_T(\mathbf{j}))}{f(\mathbf{\lambda}_\mathbf{j})}\right|\\
		\leq\ &\left|\frac{1}{|N||T|}\sum\limits_{\mathbf{j} \in N}\frac{(F^2_T(\mathbf{j})+G^2_T(\mathbf{j}))}{f(\mathbf{\lambda}_\mathbf{j})}\right|+\left|\frac{2}{|N|\sqrt{|T|}}\sqrt{\sum\limits_{\mathbf{j} \in N}\frac{I_V(\mathbf{j})}{f(\mathbf{\lambda}_\mathbf{j})}\sum\limits_{\mathbf{j} \in N}(F^2_T(\mathbf{j})+G^2_T(\mathbf{j}))}\right|\\
		\leq\ &\left|\frac{2}{|T|^2}\frac{o_p\left(|T|^2\alpha^{-1}_T\right)}{c}\right|+\left|\frac{2}{|T|^{\frac{3}{2}}}\sqrt{\sum\limits_{\mathbf{j} \in N}\frac{I_V(\mathbf{j})}{f(\mathbf{\lambda}_\mathbf{j})}\sum\limits_{\mathbf{j} \in T}(F^2_T(\mathbf{j})+G^2_T(\mathbf{j}))}\right|\\
		=\ &o_p\left(\alpha^{-1}_T\right)+\left|\frac{2}{|T|^{\frac{3}{2}}}o_p\left(|T|^{\frac{3}{2}}\alpha^{-\frac{1}{2}}_T\right)\right| 
		= o_p\left(\alpha^{-1}_T\right)+o_p\left(\alpha^{-\frac{1}{2}}_T\right) 
		= o_p(1)\,.
		\end{aligned}
		\end{equation*}
		It yields the assertion in part (b)(ii).

		For the proof of (b)(iii), similarly, 
		\begin{align*}
		&\left|\frac{1}{|N|}\sum\limits_{\mathbf{j} \in N}\frac{(I_V(\mathbf{j})-I_{\widehat{V}}(\mathbf{j}))^2}{f^2(\mathbf{\lambda}_\mathbf{j})}\right|\\
		=\ &\left|\frac{1}{N}\sum\limits_{\mathbf{j} \in N}\frac{\left[-\frac{1}{|T|}(F^2_T(\mathbf{j})+G^2_T(\mathbf{j}))+\frac{2}{\sqrt{|T|}}(x_V(\mathbf{j})F_T(\mathbf{j})-y_V(\mathbf{j})G_T(\mathbf{j}))\right]^2}{f^2(\mathbf{\lambda}_\mathbf{j})}\right|\\
		\asymp\ & \left|\frac{1}{N}\sum\limits_{\mathbf{j} \in N}\frac{\frac{1}{|T|^2}(F^2_T(\mathbf{j})+G^2_T(\mathbf{j}))^2+\frac{4}{|T|}(x_V(\mathbf{j})F_T(\mathbf{j})-y_V(\mathbf{j})G_T(\mathbf{j}))^2}{f^2(\mathbf{\lambda}_\mathbf{j})}\right|\\
		\leq\ & \left|\frac{1}{|N||T|^2}\sum\limits_{\mathbf{j} \in N}\frac{(F^2_T(\mathbf{j})+G^2_T(\mathbf{j}))^2}{f^2(\mathbf{\lambda}_\mathbf{j})}\right|+\left|\frac{4}{|N||T|}\sum\limits_{\mathbf{j} \in N}\frac{(x_V(\mathbf{j})F_T(\mathbf{j})-y_V(\mathbf{j})G_T(\mathbf{j}))^2}{f^2(\mathbf{\lambda}_\mathbf{j})}\right|\\
		\asymp\ & \left|\frac{1}{|N||T|^2}\sum\limits_{\mathbf{j} \in N}\frac{(F^2_T(\mathbf{j})+G^2_T(\mathbf{j}))^2}{f^2(\mathbf{\lambda}_\mathbf{j})}\right|+\left|\frac{4}{|N||T|}\sum\limits_{\mathbf{j} \in N}\frac{x^2_V(\mathbf{j})F^2_T(\mathbf{j})+y^2_V(\mathbf{j})G^2_T(\mathbf{j})}{f^2(\mathbf{\lambda}_\mathbf{j})}\right| \displaybreak[3] \\
		=\ &\left|\frac{1}{|N||T|^2}o_p\left(|T|^4\alpha^{-2}_T\right)\right|+\left|\frac{4}{|N||T|}\sqrt{\left(\sum\limits_{\mathbf{j} \in N}\frac{x^4_V(\mathbf{j})+y^4_V(\mathbf{j})}{f^2(\mathbf{\lambda}_\mathbf{j})}\right)\left(\sum\limits_{\mathbf{j} \in N}\frac{F^4_T(\mathbf{j})+G^4_T(\mathbf{j})}{f^2(\mathbf{\lambda}_\mathbf{j})}\right)}\right|  
		\end{align*}
		\begin{align*}
        \leq\ &o_p\left(\frac{|T|}{\alpha^2_T}\right)+\left|\frac{4}{|N||T|}\sqrt{\left(\sum\limits_{\mathbf{j} \in N}\frac{I^2(\mathbf{j})}{f^2(\mathbf{\lambda}_\mathbf{j})}\right)\left(\sum\limits_{\mathbf{j} \in T}\frac{F_T^2(\mathbf{j})+G_T^2(\mathbf{j})}{f(\mathbf{\lambda}_\mathbf{j})}\right)^2}\right|\\
		=\ &o_p\left(\frac{|T|}{\alpha^2_T}\right)+\left|\frac{4}{|N||T|}\sqrt{C|N|+o_p(|N|)}o_p\left(|T|^2\alpha^{-1}_T\right)\right|
		= o_p\left(\frac{|T|}{\alpha^2_T}\right)+o_p\left(\frac{|T|^{\frac{1}{2}}}{\alpha_T}\right)\\
		=\ & o_p(1)\,.
		\end{align*}
		
		From this we get, by $a^2-b^2=-(a-b)^2+2a(a-b)$,
		\begin{equation*}
		\begin{aligned}
		&\frac{1}{|N|}\sum\limits_{\mathbf{j} \in N}\frac{I^2_V(\mathbf{j})-I^2_{\widehat{V}}(\mathbf{j})}{f^2(\mathbf{\lambda}_\mathbf{j})}\\
		=\ &-\frac{1}{|N|}\sum\limits_{\mathbf{j} \in N}\frac{(I_V(\mathbf{j})-I_{\widehat{V}}(\mathbf{j}))^2}{f^2(\mathbf{\lambda}_\mathbf{j})}+\frac{2}{|N|}\sum\limits_{\mathbf{j} \in N}\frac{I_V(\mathbf{j})(I_V(\mathbf{j})-I_{\widehat{V}}(\mathbf{j}))}{f^2(\mathbf{\lambda}_\mathbf{j})}\\
		\leq\ &\frac{1}{|N|}\sum\limits_{\mathbf{j} \in N}\frac{(I_V(\mathbf{j})-I_{\widehat{V}}(\mathbf{j}))^2}{f^2(\mathbf{\lambda}_\mathbf{j})}+\frac{2}{|N|}\sqrt{\left(\sum\limits_{\mathbf{j} \in N}\frac{I^2_V(\mathbf{j})}{f^2(\mathbf{\lambda}_\mathbf{j})}\right)\left(\sum\limits_{\mathbf{j} \in N}\frac{(I_V(\mathbf{j})-I_{\widehat{V}}(\mathbf{j}))^2}{f^2(\mathbf{\lambda}_\mathbf{j})}\right)}\\
		=\ &o_p\left(\frac{|T|}{\alpha^2_T}\right)+o_p\left(\frac{|T|^{\frac{1}{2}}}{\alpha_T}\right)+\frac{2}{|N|}\sqrt{(C|N|+o(|N|))\left[o_p\left(\frac{|N||T|}{\alpha^2_T}\right)+o_p\left(\frac{|N||T|^{\frac{1}{2}}}{\alpha_T}\right)\right]}\\
		=\ &o_p\left(\frac{|T|}{\alpha^2_T}\right)+o_p\left(\frac{|T|^{\frac{1}{2}}}{\alpha_T}\right)+o_p\left(\frac{|T|^{\frac{1}{4}}}{\alpha_T^{\frac{1}{2}}}\right)
		=\ o_p(1)\,,
		\end{aligned}
		\end{equation*}
		which yields the assertion in part (b)(iii). Next, we have that the following holds: 
		\begin{equation*}
		\begin{aligned}
		\frac{1}{|N|}\sum\limits_{\mathbf{j} \in N}\frac{I^q_V(\mathbf{j})-I^q_{\widehat{V}}(\mathbf{j})}{f^q(\mathbf{\lambda}_\mathbf{j})}
		\asymp\ &\left|\frac{1}{|N|}\sum\limits_{\mathbf{j} \in N}\frac{(I_V(\mathbf{j})-I_{\widehat{V}}(\mathbf{j}))^q}{f^q(\mathbf{\lambda}_\mathbf{j})}\right| 
		\asymp \left|\frac{1}{|N|} \left(\sum\limits_{\mathbf{j} \in N}\frac{(I_V(\mathbf{j})-I_{\widehat{V}}(\mathbf{j}))^2}{f^2(\mathbf{\lambda}_\mathbf{j})}\right)^{\frac{q}{2}}\right|\\
		=\ &\left|\frac{1}{|N|} \left(|N|o_p\left(\frac{|T|}{\alpha^2_T}\right)+|N|o_p\left(\frac{|T|^{\frac{1}{2}}}{\alpha_T}\right)\right)^{\frac{q}{2}}\right| 
		= o_p\left(\frac{|T|^{q-1}}{\alpha_T^{q}}\right) 
		= o_p(1)\,,
		\end{aligned}
		\end{equation*}
		which yields the assertion in part (b)(iv).
		 For the proof of (c)(i), we have
		 \begin{align*}
		 &\sup_{\mathbf{k} \in N}\left|\sum\limits_{\mathbf{j} \in\mathbb{Z}^2}p_{\mathbf{j},T}[x_V(\mathbf{k}+\mathbf{j})-x_{\widehat{V}}(\mathbf{k}+\mathbf{j})+y_V(\mathbf{k}+\mathbf{j})-y_{\widehat{V}}(\mathbf{k}+\mathbf{j})]\right|\\
		 =\ &\sup_{\mathbf{k} \in N}\left|\sum\limits_{\mathbf{j} \in\mathbb{Z}^2}p_{\mathbf{j},T}\left(\frac{1}{\sqrt{|T|}}(F_T(\mathbf{k}+\mathbf{j})-G_T(\mathbf{k}+\mathbf{j}))\right)\right| \displaybreak[3] \\
		 \leq\ &\sup_{\mathbf{k} \in N}\frac{\sqrt{2}}{\sqrt{|T|}}\sqrt{\sum\limits_{\mathbf{j} \in\mathbb{Z}^2}p_{\mathbf{j},T}\sum\limits_{\mathbf{j} \in\mathbb{Z}^2}p_{\mathbf{j},T}(F^2_T(\mathbf{k}+\mathbf{j})+G^2_T(\mathbf{k}+\mathbf{j}))}\\
		 \asymp \ &\sup_{\mathbf{k} \in N}\frac{\sqrt{2}}{\sqrt{|T|}}\sqrt{\frac{1}{|\mathbf{h}_T||T|}\sum\limits_{\mathbf{j} \in T}(F^2_T(\mathbf{k}+\mathbf{j})+G^2_T(\mathbf{k}+\mathbf{j}))}\\
		 \asymp \ &\sup_{\mathbf{k} \in N}\frac{\sqrt{2}}{\sqrt{|T|}}\frac{1}{\sqrt{|\mathbf{h}_T||T|}}\sqrt{o_p\left(|T|^2\alpha^{-1}_T\right)} 
		 = o_p\left((|\mathbf{h}_T|\alpha_T)^{-\frac{1}{2}}\right)\\
		 =\ &o_p\left((|\mathbf{h}_T||T|)^{-\frac{q-1}{2q}}\right) 
		 = o_p(1) \quad \text{as} \quad |\mathbf{h}_T||T| \rightarrow \infty \,,
		 \end{align*}
		 which yields the assertion in part (c)(i). By the same arguments as for part (a), we get
		\begin{equation*}
		\sup_{\mathbf{k} \in N}\left|\sum\limits_{\mathbf{j} \in\mathbb{Z}^2}p_{\mathbf{j},T}(I_V(\mathbf{k}+\mathbf{j})-I_{\widehat{V}}(\mathbf{k}+\mathbf{j}))\right|=o_p\left((|\mathbf{h}_T|\alpha_T)^{-1}\right)+o_p\left((|\mathbf{h}_T|\alpha_T)^{-\frac{1}{2}}\right)=o_p(1)\,,
		\end{equation*}
		which yields the assertion in part (c)(ii). Similar to the proof for part (b)(iii), we get
		\begin{align*}
		&\sup_{\mathbf{k} \in N}\sum\limits_{\mathbf{j} \in\mathbb{Z}^2}p_{\mathbf{j},T}(I_V(\mathbf{k}+\mathbf{j})-I_{\widehat{V}}(\mathbf{k}+\mathbf{j}))^2\\
		\asymp\ &\sup_{\mathbf{k} \in N}\sum\limits_{\mathbf{j} \in\mathbb{Z}^2}p_{\mathbf{j},T}\left[\frac{1}{|T|^2}(F^2_T(\mathbf{j})+G^2_T(\mathbf{j}))^2+\frac{4}{|T|}(x_V(\mathbf{j})F_T(\mathbf{j})-y_V(\mathbf{j})G_T(\mathbf{j}))^2\right]\\
		\asymp\ &\frac{1}{|T|^2}\frac{1}{|\mathbf{h}_T||T|}\sum\limits_{\mathbf{j} \in T}[F^2_T(\mathbf{j})+G^2_T(\mathbf{j})]^2+\frac{4}{|T|}\sum\limits_{\mathbf{j} \in\mathbb{Z}^2}p_{\mathbf{j},T}[x^2_V(\mathbf{j})F^2_T(\mathbf{j})+y^2_V(\mathbf{j})G^2_T(\mathbf{j})]\\
		\asymp\ &o_p\left(\frac{|T|}{|\mathbf{h}_T|\alpha^2_T}\right)+\frac{4}{|T|}\sqrt{\sum\limits_{\mathbf{j} \in\mathbb{Z}^2}p_{\mathbf{j},T}(x^4_V(\mathbf{j})+y^4_V(\mathbf{j}))\sum\limits_{\mathbf{j} \in\mathbb{Z}^2}p_{\mathbf{j},T}(F^4_T(\mathbf{j})+G^4_T(\mathbf{j}))} \displaybreak[3] \\
		\asymp\ &o_p\left(\frac{|T|}{|\mathbf{h}_T|\alpha^2_T}\right)+\frac{4}{|T|}\sqrt{\sum\limits_{\mathbf{j} \in\mathbb{Z}^2}p_{\mathbf{j},T}I^2(\mathbf{j})\frac{1}{|\mathbf{h}_T||T|}\sum\limits_{\mathbf{j} \in T}(F^4_T(\mathbf{j})+G^4_T(\mathbf{j}))}\\
		\asymp\ &o_p\left(\frac{|T|}{|\mathbf{h}_T|\alpha^2_T}\right)+\frac{4}{|T|}\sqrt{(C+o_p(1))\frac{1}{|\mathbf{h}_T||T|}\left[\sum\limits_{\mathbf{j} \in T}(F^2_T(\mathbf{j})+G^2_T(\mathbf{j}))\right]^2}\\
		\asymp\ &o_p\left(\frac{|T|}{|\mathbf{h}_T|\alpha^2_T}\right)+\frac{4}{|T|}\frac{1}{\sqrt{|\mathbf{h}_T|}\sqrt{|T|}}o_p\left(|T|^2\alpha^{-1}_T\right) 
		=  o_p\left(\frac{|T|}{|\mathbf{h}_T|\alpha^2_T}\right)+o_p\left(\frac{|T|^{\frac{1}{2}}}{|\mathbf{h}_T|^{\frac{1}{2}}\alpha_T}\right) 
		=  o_p(1)\,,
		\end{align*}
		which yields 
		\begin{equation*}
		\begin{aligned}
		&\sup_{\mathbf{k} \in N}\sum\limits_{\mathbf{j} \in\mathbb{Z}^2}p_{\mathbf{j},T}(I^2_V(\mathbf{k}+\mathbf{j})-I^2_{\widehat{V}}(\mathbf{k}+\mathbf{j}))\\
		=\ &\sup_{\mathbf{k} \in N}\sum\limits_{\mathbf{j} \in\mathbb{Z}^2}p_{\mathbf{j},T}[-(I_V(\mathbf{k}+\mathbf{j})-I_{\widehat{V}}(\mathbf{k}+\mathbf{j}))^2+2I_V(\mathbf{k}+\mathbf{j})(I_V(\mathbf{k}+\mathbf{j})-I_{\widehat{V}}(\mathbf{k}+\mathbf{j}))]\\
		=\ &o_p\left(\frac{|T|}{|\mathbf{h}_T|\alpha^2_T}\right)+o_p\left(\frac{|T|^{\frac{1}{2}}}{|\mathbf{h}_T|^{\frac{1}{2}}\alpha_T}\right) 
		= o_p(1)\,,
		\end{aligned}
		\end{equation*}
		which yields the assertion in part (c)(iii). Finally, we have
		\begin{equation*}
		\begin{aligned}
		\sup_{\mathbf{k} \in N}\sum\limits_{\mathbf{j} \in\mathbb{Z}^2}p_{\mathbf{j},T}(I^q_V(\mathbf{k}+\mathbf{j})-I^q_{\widehat{V}}(\mathbf{k}+\mathbf{j}))
		\asymp\ &\sup_{\mathbf{k} \in N}\sum\limits_{\mathbf{j} \in\mathbb{Z}^2}p_{\mathbf{j},T}(I_V(\mathbf{k}+\mathbf{j})-I_{\widehat{V}}(\mathbf{k}+\mathbf{j}))^q\\
		\asymp\ &\sup_{\mathbf{k} \in N}(|\mathbf{h}_T||T|)^{q-1}\left|\sum\limits_{\mathbf{j} \in\mathbb{Z}^2}p_{\mathbf{j},T}(I_V(\mathbf{k}+\mathbf{j})-I_{\widehat{V}}(\mathbf{k}+\mathbf{j}))\right|^q	\\ 
		=\ & o_p\left(\frac{|T|^{q-1}}{|\mathbf{h}_T|\alpha_T^q}\right)
		=o_p(1)\,,
		\end{aligned}
		\end{equation*}
		which yields the assertion in part (c)(iv).
		
		
	
\end{proof}

\bibliographystyle{rss} 
\bibliography{myRef}

\end{document}